\documentclass[11pt]{amsart}

\pdfoutput=1

\usepackage{etex}
\usepackage{amsmath, amssymb}
\usepackage{array}
\usepackage[frame,cmtip,arrow,matrix,line,graph,curve]{xy}
\usepackage{graphpap, color, paralist, pstricks}
\usepackage[mathscr]{eucal}
\usepackage{graphicx}
\usepackage[bookmarks=false]{hyperref}
\usepackage{tikz}
\usepackage{leftidx}

\numberwithin{equation}{section}

\newtheorem{theorem}{Theorem}[section]
\newtheorem{proposition}[theorem]{Proposition}
\newtheorem{corollary}[theorem]{Corollary}
\newtheorem{lemma}[theorem]{Lemma}

\newtheorem{conjecture}[theorem]{Conjecture}

\theoremstyle{definition}
\newtheorem{definition}[theorem]{Definition}

\bgroup\catcode`\#=12\egroup

\newcommand{\CC}{\mathbb{C}}
\newcommand{\GG}{\mathbb{G}}
\newcommand{\HH}{\mathbb{H}}
\newcommand{\PP}{\mathbb{P}}
\newcommand{\QQ}{\mathbb{Q}}
\newcommand{\RR}{\mathbb{R}}
\newcommand{\ZZ}{\mathbb{Z}}
\newcommand{\bA}{\mathbf{A}}
\newcommand{\bfB}{\mathbf{B}}
\newcommand{\bD}{\mathbf{D}}

\newcommand{\bM}{\mathbf{M}}
\newcommand{\bN}{\mathbf{N}}
\newcommand{\bR}{\mathbf{R}}
\newcommand{\bK}{\mathbf{K}}
\newcommand{\bg}{\mathbf{g}}

\newcommand{\cA}{\mathcal{A}}
\newcommand{\cB}{\mathcal{B}}
\newcommand{\cC}{\mathcal{C}}
\newcommand{\cG}{\mathcal{G}}
\newcommand{\cH}{\mathcal{H}}
\newcommand{\cL}{\mathcal{L}}
\newcommand{\cM}{\mathcal{M}}
\newcommand{\cN}{\mathcal{N}}
\newcommand{\cO}{\mathcal{O}}

\newcommand{\tcC}{\tilde{\mathcal{C}}}

\newcommand{\tq}{\tilde{q}}
\newcommand{\tC}{\tilde{C}}
\newcommand{\ttC}{\tilde{\tilde{C}}}

\newcommand{\tU}{\tilde{U}}
\newcommand{\tx}{\tilde{x}}

\newcommand{\tQ}{\tilde{Q}}
\newcommand{\tal}{\tilde{\alpha}}
\newcommand{\tbe}{\tilde{\beta}}
\newcommand{\ovg}{\overline{g}}
\newcommand{\ovh}{\overline{h}}
\newcommand{\ovq}{\overline{q}}
\newcommand{\dss}{\displaystyle}
\newcommand{\B}{\mathrm{B}}
\newcommand{\dv}{\mathrm{div}}
\newcommand{\Dol}{\mathrm{Dol}}
\newcommand{\diag}{\mathrm{diag}}
\newcommand{\Hom}{\mathrm{Hom}}
\newcommand{\Isom}{\mathrm{Isom}}
\newcommand{\Lie}{\mathrm{Lie}}
\newcommand{\End}{\mathrm{End}}

\newcommand{\Sig}{\Sigma}
\newcommand{\om}{\omega}

\def\ic{\mathbf{IC}}
\def\coker{\mathrm{coker} }

\def\Aut{\mathrm{Aut} }
\def\id{\mathrm{id} }
\def\im{\mathrm{im} }
\def\rk{\mathrm{rk}\,}
\def\GL{\mathrm{GL}}
\def\PGL{\mathrm{PGL}}
\def\EPGL{\mathrm{EPGL}}
\def\SL{\mathrm{SL}}
\def\SU{\mathrm{SU}}
\def\Sl{\mathrm{Sl}}
\def\E{\mathrm{E}}
\def\ESL{\mathrm{ESL}}
\def\SO{\mathrm{SO}}
\def\ESO{\mathrm{ESO}}
\def\BSO{\mathrm{BSO}}
\def\O{\mathrm{O}}
\def\EO{\mathrm{EO}}
\def\Pic{\mathrm{Pic}}
\def\BSL{\mathrm{BSL}}

\def\st{\mathrm{Stab}}
\def\spec{\mathrm{Spec}}
\def\bproj{\mathbf{Proj}}
\def\Tr{\mathrm{Tr}}
\def\proj{\mathrm{Proj}}

\newcommand{\Om}{\Omega}

\begin{document}

\title[A blowing-up formula for Intersection cohomology of moduli of Higgs]{A blowing-up formula for the Intersection cohomology of the moduli of rank $2$ Higgs bundles over a curve with trivial determinant}

\author{Sang-Bum Yoo}
\address{Sang-Bum Yoo \\ Department of Mathematics Education \\ Gongju National University of Education \\ Gongju-si, Chungcheongnam-do, 32553, Republic of Korea}
\email{sbyoo@gjue.ac.kr}



\begin{abstract}
We prove that a blowing-up formula for the intersection cohomology of the moduli space of rank $2$ Higgs bundles over a curve with trivial determinant holds. As an application, we derive the Poincar\'{e} polynomial of the intersection cohomology of the moduli space under a technical assumption.
\end{abstract}

\maketitle

\section{Introduction}

Let $X$ be a smooth complex projective curve of genus $g\ge2$ and let $G$ be $\GL(r,\CC)$ or $\SL(r,\CC)$. Let $\bM_{\Dol}^{d}(G)$ be the moduli space of $G$-Higgs bundles $(E,\phi)$ of rank $r$ and degree $d$ on $X$ (with fixed $\det E$ and traceless $\phi$ in the case $G=\SL(r,\CC)$) and let $\bM_{\B}^{d}(G)$ be the character variety for $G$ defined by
$$\bM_{\B}^{d}(G)=\{(A_{1},B_{1},\cdots,A_{g},B_{g})\in G\,|\,[A_{1},B_{1}]\cdots[A_{g},B_{g}]=e^{2\pi\sqrt{-1}d/r}I_{r}\}/\!/G.$$
By the theory of harmonic bundles (\cite{Co88}, \cite{Simp92}), we have a homeomorphism $\bM_{\Dol}^{d}(G)\cong\bM_{\B}^{d}(G)$ as a part of the nonabelian Hodge theory. If $r$ and $d$ are relatively prime, these moduli spaces are smooth and their underlying differentiable manifold is hyperk\"{a}hler. But the complex structures do not coincide under this homeomorphism.

Under the assumption that $r$ and $d$ are relatively prime, motivated by this fact, there have been several works calculating invariants of these moduli spaces on both sides over the last 30 years. The Poincar\'{e} polynomial of the ordinary cohomology is calculated, for $\bM_{\Dol}^{d}(\SL(2,\CC))$ by N. Hitchin in \cite{Hit87}, and for $\bM_{\Dol}^{d}(\SL(3,\CC))$ by P. Gothen in \cite{Go94}. The compactly supported Hodge polynomial and the compactly supported Poincar\'{e} polynomial for $\bM_{\Dol}^{d}(\GL(4,\CC))$ can be obtained by the motivic calculation in \cite{GHS14}. By counting the number of points of these moduli spaces over finite fields with large characteristics, the compactly supported Poincar\'{e} polynomials for $\bM_{\Dol}^{d}(\GL(r,\CC))$ and $\bM_{\B}^{d}(\GL(r,\CC))$ are obtained in \cite{Sch16}. By using arithmetic methods, T. Hausel and F. Rodriguez-Villegas expresses the E-polynomial of $\bM_{\B}^{d}(\GL(r,\CC))$ in terms of a simple generating function in \cite{HR08}. By the same way, M. Mereb calculates the E-polynomial of $\bM_{\B}^{d}(\SL(2,\CC))$ and expresses the E-polynomial of $\bM_{\B}^{d}(\SL(r,\CC))$ in terms of a generating function in \cite{Me15}.

Without the assumption that $r$ and $d$ are relatively prime, there have been also some works calculating invariants of $\bM_{\B}^{d}(\SL(r,\CC))$. For $g=1,2$ and any $d$, explicit formulas for the E-polynomials of $\bM_{\B}^{d}(\SL(2,\CC))$ are obtained by a geometric technique in \cite{LMN13}. The E-polynomial of $\bM_{\B}^{d}(\SL(2,\CC))$ is calculated, for $g=3$ and any $d$ by a further geometric technique in \cite{MM16-1}, and for any $g$ and $d$ in \cite{MM16-2}.

When we deal with a singular variety $\bM_{\Dol}^{d}(G)$ under the condition that $r$ and $d$ are not relatively prime, the intersection cohomology is a natural invariant. Our interest is focused on the intersection cohomology of $\bM:=\bM_{\Dol}^0(\SL(2,\CC))$.

For a quasi-projective variety $V$, $IH^{i}(V)$ and $\ic^{\bullet}(V)$ denote the $i$-th intersection cohomology of $V$ of the middle perversity and the complex of sheaves on $V$ whose hypercohomology is $IH^{*}(V)$ respectively. $IP_{t}(V)$ denotes the Poincar\'{e} polynomial of $IH^{*}(V)$ defined by
$$IP_{t}(V)=\sum_{i}\dim IH^{i}(V).$$

Recently, $IP_{t}(\bM)$ was calculated in \cite{Ma21} by using ways different from ours. First of all, the $E$-polynomial of the compactly supported intersection cohomology of $\bM$ was calculated in \cite{Fel18} over a smooth curve of genus $2$ and in \cite{Ma21} over a smooth curve of genus $g\ge 2$. Then the author of \cite{Ma21} proved the purity of $IH^{*}(\bM)$ from the observation of the semiprojectivity of $\bM$. He used the purity of $IH^{*}(\bM)$ and the Poincar\'{e} duality for the intersection cohomology (Theorem \ref{GPD}) to calculate $IP_{t}(\bM)$.

From now on, $\GL(n,\CC)$, $\SL(n,\CC)$, $\PGL(n,\CC)$, $\O(n,\CC)$ and $\SO(n,\CC)$ will be denoted by $\GL(n)$, $\SL(n)$, $\PGL(n)$, $\O(n)$ and $\SO(n)$ respectively for the simplicities of notations.

\subsection{Main result}

In this paper, we prove that a blowing-up formula for $IH^{*}(\bM)$ holds.

It is known that $\bM$ is a good quotient $\bR/\!/\SL(2)$ for some quasi-projective variety $\bR$ (Theorem \ref{M is a good quotient}, Theorem \ref{R and M are quasi-projective}). $\bM$ is decomposed into
$$\bM^{s}\bigsqcup(T^{*}J/\ZZ_{2}-\ZZ_{2}^{2g})\bigsqcup\ZZ_{2}^{2g},$$
where $\bM^{s}$ denotes the stable locus of $\bM$ and $J:=\Pic^0(X)$. The singularity along the locus $\ZZ_{2}^{2g}$ is the quotient $\Upsilon^{-1}(0)/\!/\SL(2)$, where $\Upsilon^{-1}(0)$ is the affine cone over a reduced irreducible complete intersection of three quadrics in $\PP(\CC^{2g}\otimes sl(2))$ and $\SL(2)$ acts on $\CC^{2g}\otimes sl(2)$ as the adjoint representation. The singularity along the locus $T^{*}J/\ZZ_{2}-\ZZ_{2}^{2g}$ is $\Psi^{-1}(0)/\!/\CC^{*}$, where $\Psi^{-1}(0)$ is the affine cone over a smooth quadric in $\PP((\CC^{g-1})^{4})$ and $\CC^{*}$ acts on $(\CC^{g-1})^{4}$ with weights $-2,2,2$ and $-2$. Let us consider the Kirwan's algorithm consisting of three blowing-ups $\bK:=\bR_{3}^{s}/\SL(2)\to\bR_{2}^{s}/\SL(2)\to\bR_{1}^{ss}/\!/\SL(2)\to\bR/\!/\SL(2)=\bM$ induced from the composition of blowing-ups $\pi_{\bR_{1}}:\bR_{1}\to\bR$ along the locus $\ZZ_{2}^{2g}$ with the exceptional divisor $E_{1}$, $\pi_{\bR_{2}}:\bR_{2}\to\bR_{1}^{ss}$ along the strict transform $\Sigma$ of the locus of $\bR_{1}^{ss}$ over $T^{*}J/\ZZ_{2}-\ZZ_{2}^{2g}$ with the exceptional divisor $E_{2}$ and $\bR_{3}\to\bR_{2}^{s}$ along the locus of points with stabilizers larger than the center $\ZZ_{2}$ in $\SL(2)$ (Section \ref{Kirwan desing}).

We also have local pictures of the Kirwan's algorithm mentioned above. For any $x\in\ZZ_{2}^{2g}$, we have $\pi_{\bR_{1}}^{-1}(x)=\PP\Upsilon^{-1}(0)$ which is a subset of $\PP\Hom(sl(2),\CC^{2g})$ and $\pi_{\bR_{1}}^{-1}(x)\cap\Sigma$ is the locus of rank $1$ matrices (Section \ref{local picture}). Thus the strict transform of $\PP\Upsilon^{-1}(0)^{ss}/\!/\PGL(2)$ in the second blowing-up of the Kirwan's algorithm is just the blowing-up $$Bl_{\PP\Hom_{1}}\PP\Upsilon^{-1}(0)^{ss}/\!/\PGL(2)\to\PP\Upsilon^{-1}(0)^{ss}/\!/\PGL(2)$$
along the image of the locus of rank $1$ matrices in $\PP\Upsilon^{-1}(0)^{ss}/\!/\PGL(2)$.

In this setup, we have the following main result.

\begin{theorem}[Theorem \ref{computable intersection blowing-up formula}]\label{main theorem}\label{thm1}
Let $I_{2g-3}$ be the incidence variety given by
$$I_{2g-3}=\{(p,H)\in\PP^{2g-3}\times\breve{\PP}^{2g-3}|p\in H\}.$$
Then we have
\begin{enumerate}
\item $\dim IH^{i}(\bR_{1}^{ss}/\!/\SL(2))=\dim IH^{i}(\bM)$
$$+2^{2g}\dim IH^{i}(\PP\Upsilon^{-1}(0)^{ss}/\!/\PGL(2))-2^{2g}\dim IH^{i}(\Upsilon^{-1}(0)/\!/\PGL(2))$$
for all $i\ge0$.

\item $\dim H^{i}(\bR_{2}^{s}/\SL(2))=\dim IH^{i}(\bR_{1}^{ss}/\!/\SL(2))$
$$+\sum_{p+q=i}\dim[H^{p}(\widetilde{T^{*}J})\otimes H^{t(q)}(I_{2g-3})]^{\ZZ_{2}}$$
for all $i\ge0$, where $t(q)=q-2$ for $q\le\dim I_{2g-3}=4g-7$ and $t(q)=q$ otherwise, where $\alpha:\widetilde{T^{*}J}\to T^{*}J$ is the blowing-up along $\ZZ_{2}^{2g}$.

\item $\dim IH^{i}(Bl_{\PP\Hom_{1}}\PP\Upsilon^{-1}(0)^{ss}/\!/\SL(2))=\dim IH^{i}(\PP\Upsilon^{-1}(0)^{ss}/\!/\SL(2))$
$$+\sum_{p+q=i}\dim[H^{p}(\PP^{2g-1})\otimes H^{t(q)}(I_{2g-3})]^{\ZZ_{2}}$$
for all $i\ge0$, where $t(q)=q-2$ for $q\le\dim I_{2g-3}=4g-7$ and $t(q)=q$ otherwise.
\end{enumerate}
\end{theorem}

It is an essential process to apply this blowing-up formula to calculate $IP_{t}(\bM)$.

\subsection{Method of proof of Theorem \ref{main theorem}}

We follow the same steps as in the proof of \cite[Proposition 2.1]{K86}, but we give a proof in each step because our setup is different from that of \cite{K86}. We start with the following formulas coming from the decomposition theorem (Proposition \ref{3 consequences}-(1)) and the same argument as in the proof of \cite[Lemma 2.8]{K86} :
$$\dim IH^{i}(\bR_{1}^{ss}/\!/\SL(2))=\dim IH^{i}(\bM)+\dim IH^{i}(\tU_{1})-\dim IH^{i}(U_{1}),$$
$$\dim IH^{i}(\bR_{2}^{s}/\SL(2))=\dim IH^{i}(\bR_{1}^{ss}/\!/\SL(2))+\dim IH^{i}(\tU_{2})-\dim IH^{i}(U_{2})$$
and
$$\dim IH^{i}(Bl_{\PP\Hom_{1}}\PP\Upsilon^{-1}(0)^{ss}/\!/\SL(2))=\dim IH^{i}(\PP\Upsilon^{-1}(0)^{ss}/\!/\SL(2))+\dim IH^{i}(\tU)-\dim IH^{i}(U),$$
where $U_{1}$ is a disjoint union of sufficiently small open neighborhoods of each point of $\ZZ_{2}^{2g}$ in $\bM$, $\tU_{1}$ is the inverse image of the first blowing-up, $U_{2}$ is an open neighborhood of the strict transform of $T^{*}J/\ZZ_{2}$ in $\bR_{1}^{ss}/\!/\SL(2)$, $\tU_{2}$ is the inverse image of the second blowing-up, $U$ is an open neighborhood of the locus of rank $1$ matrices in $\PP\Upsilon^{-1}(0)^{ss}/\!/\PGL(2)$ and $\tU$ is the inverse image of the blowing-up map $Bl_{\PP\Hom_{1}}\PP\Upsilon^{-1}(0)^{ss}/\!/\PGL(2)\to\PP\Upsilon^{-1}(0)^{ss}/\!/\PGL(2)$. By Proposition \ref{normal slice is the quadratic cone}, we can see that $U_{1},\tU_{1},U_{2},\tU_{2},U$ and $\tU$ are analytically isomorphic to relevant normal cones respectively. By Section \ref{Kirwan desing} and Lemma \ref{geometric descriptions of second cones}, these relevant normal cones are described as free $\ZZ_{2}$-quotients of nice fibrations with concrete expressions of bases and fibers. By the calculations of the intersection cohomologies of fibers (Lemma \ref{int coh of affine cone of git quotient} and Lemma \ref{suffices-to-show-on-projectivized-git}) and applying the Leray spectral sequences (Proposition \ref{3 consequences}-(2)) of intersection cohomologies associated to these fibrations, we complete the proof.

\subsection{Towards a formula for the Poincar\'{e} polynomial of $IH^{*}(\bM)$}
For a topological space $W$ on which a reductive group $G$ acts, $H_{G}^{i}(W)$ and $P_{t}^{G}(W)$ denote the $i$-th equivariant cohomology of $W$ and the Poincar\'{e} series of $H_{G}^{*}(W)$ defined by
$$P_{t}^{G}(W)=\sum_{i}\dim H_{G}^{i}(W).$$
We start with the formula of $P_{t}^{\SL(2)}(\bR)$ that comes from that of \cite{DWW11}. Then we use the following conjectural formulas.
\begin{conjecture}[Conjecture \ref{equivariant blowing up formula conjecture}, Proposition \ref{equivariant blowing up formula on the 2nd blowing-up}]
\begin{enumerate}
\item $P_{t}^{\SL(2)}(\bR_{1}^{ss})=P_{t}^{\SL(2)}(\bR)+2^{2g}(P_{t}^{\SL(2)}(\PP\Upsilon^{-1}(0)^{ss})-P_{t}(\BSL(2)))$.

\item $P_{t}^{\SL(2)}(\bR_{2}^{s})=P_{t}^{\SL(2)}(\bR_{1}^{ss})+P_{t}^{\SL(2)}(E_2^{ss})-P_{t}^{\SL(2)}(\Sig)$ under some technical assumptions, where $E_{2}^{ss}=E_{2}\cap\bR_{2}^{ss}$.
\end{enumerate}
\end{conjecture}

We use this conjectural blowing-up formula for the equivariant cohomology to get $P_{t}^{\SL(2)}(\bR_{2}^{s})$ from $P_{t}^{\SL(2)}(\bR)$. Since $\bR_{2}^{s}/\SL(2)$ has at worst orbifold singularities (Section \ref{Kirwan desing}), $P_{t}^{\SL(2)}(\bR_{2}^{s})=P_{t}(\bR_{2}^{s}/\SL(2))$ (Section \ref{strategy}). Now we use the blowing-up formula for the intersection cohomology (Theorem \ref{thm1}) to get $IP_{t}(\bM)$ from $P_{t}(\bR_{2}^{s}/\SL(2))$.

\begin{theorem}[Theorem \ref{A formula for IP(M)}]
Assume that Conjecture \ref{conjecture}, Conjecture \ref{equivariant blowing up formula conjecture} and Conjecture \ref{assumption 9.5} hold. Then
$$IP_{t}(\bM)=\frac{(1+t^3)^{2g}-(1+t)^{2g}t^{2g+2}}{(1-t^2)(1-t^4)}$$
$$-t^{4g-4}+\frac{t^{2g+2}(1+t)^{2g}}{(1-t^2)(1-t^4)}+\frac{(1-t)^{2g}t^{4g-4}}{4(1+t^2)}$$
$$+\frac{(1+t)^{2g}t^{4g-4}}{2(1-t^2)}(\frac{2g}{t+1}+\frac{1}{t^2-1}-\frac{1}{2}+(3-2g))$$
$$+\frac{1}{2}(2^{2g}-1)t^{4g-4}((1+t)^{2g-2}+(1-t)^{2g-2}-2)$$
$$+2^{2g}\big[\big(\frac{1-t^{12}}{1-t^2}-\frac{1-t^6}{1-t^2}+(\frac{1-t^6}{1-t^2})^{2}\big)\frac{(1-t^{4g-8})(1-t^{4g-4})(1-t^{4g})}{(1-t^2)(1-t^4)(1-t^6)}$$
$$-\frac{1-t^6}{1-t^2}\frac{(1-t^{4g-4})(1-t^{4g})}{(1-t^2)(1-t^4)}\frac{t^2(1-t^{2(2g-5)})}{1-t^2}$$
$$-\frac{(1-t^{4g-4})^{2}}{(1-t^2)(1-t^4)}\cdot\frac{1-t^{4g}}{1-t^2}+\frac{1}{1-t^4}\frac{1-t^{4g}}{1-t^2}\big]-\frac{2^{2g}}{1-t^4}$$
$$+(\frac{1}{2}((1+t)^{2g}+(1-t)^{2g})+2^{2g}(\frac{1-t^{4g}}{1-t^2}-1))\frac{(1-t^{4g-4})^{2}}{(1-t^2)(1-t^4)}$$
$$+\frac{1}{2}((1+t)^{2g}-(1-t)^{2g})\frac{t^2(1-t^{4g-4})(1-t^{4g-8})}{(1-t^2)(1-t^4)}$$
$$-\frac{1}{(1-t^4)}(\frac{1}{2}((1+t)^{2g}+(1-t)^{2g})+2^{2g}(\frac{1-t^{4g}}{1-t^2}-1))$$
$$-\frac{t^2}{(1-t^4)}\frac{1}{2}((1+t)^{2g}-(1-t)^{2g})$$
$$-\frac1 2(
(1+t)^{2g}+(1-t)^{2g})+2^{2g}(\frac{1-t^{4g}}{1-t^2}-1))\frac{t^2(1-t^{4g-4})(1-t^{4g-6})}{(1-t^2)(1-t^4)}$$
$$-\frac1 2(
(1+t)^{2g}-(1-t)^{2g})
(\frac{t^4(1-t^{4g-4})(1-t^{4g-10})}{(1-t^2)(1-t^4)}+t^{4g-6})$$
$$-2^{2g}\big[\frac{(1-t^{8g-8})(1-t^{4g})}{(1-t^{2})(1-t^{4})}-\frac{1-t^{4g}}{1-t^{4}}\big]$$
which is a polynomial with degree $6g-6$.
\end{theorem}

This conjectural formula for $IP_{t}(\bM)$ coincides with that of \cite{Ma21}.

\subsection*{Notations}
Throughout this paper, $X$ denotes a smooth complex projective curve of genus $g\ge2$ and $K_X$ the canonical bundle of $X$.

\section{Higgs bundles}

In this section, we introduce two kinds of constructions of the moduli space of Higgs bundles on $X$. For details, see \cite{Hit87}, \cite{Simp94I} and \cite{Simp94II}.

\subsection{Simpson's construction}

An \textbf{$\SL(2)$-Higgs bundle} on $X$ is a pair of a rank $2$ vector bundle $E$ with trivial determinant on $X$ and a section $\phi\in H^0(X,\End_{0}(E)\otimes K_{X})$, where $\End(E)$ denotes the bundle of endomorphisms of $E$ and $\End_0(E)$ the subbundle of traceless endomorphisms of $\End(E)$. We must impose a notion of stability to construct a separated moduli space.

\begin{definition}[\cite{Hit87}, \cite{Simp94I}]
An $\SL(2)$-Higgs bundle $(E,\phi)$ on $X$ is \textbf{stable} (respectively, \textbf{semistable}) if for any $\phi$-invariant line subbundle $F$ of $E$, we have
$$\deg(F)<0\text{ (respectively, }\le\text{)}.$$
\end{definition}

Let $N$ be a sufficiently large integer and $p=2N+2(1-g)$. We list C.T. Simpson's results to construct a moduli space of $\SL(2)$-Higgs bundles.

\begin{theorem}[Theorem 3.8 of \cite{Simp94I}]
There is a quasi-projective scheme $Q$ representing the moduli functor which parametrizes the isomorphism classes of triples $(E,\phi,\alpha)$ where $(E,\phi)$ is a semistable $\SL(2)$-Higgs bundle and $\alpha$ is an isomorphism $\alpha:\CC^{p}\to H^0(X,E\otimes\cO_{X}(N))$.
\end{theorem}

\begin{theorem}[Theorem 4.10 of \cite{Simp94I}]\label{construction of bR}
Fix $x\in X$. Let $\tQ$ be the frame bundle at $x$ of the universal object restricted to $x$. Then the action of $\GL(p)$ lifts to $\tQ$ and $\SL(2)$ acts on the fibers of $\tQ\to Q$ in an obvious fashion. Every point of $\tQ$ is stable with respect to the free action of $\GL(p)$ and
$$\bR=\tQ/\GL(p)$$
represents the moduli functor which parametrizes triples $(E,\phi,\beta)$ where $(E,\phi)$ is a semistable $\SL(2)$-Higgs bundle and $\beta$ is an isomorphism $\beta:E|_{x}\to\CC^{2}$.
\end{theorem}

\begin{theorem}[Theorem 4.10 of \cite{Simp94I}]
Every point in $\bR$ is semistable with respect to the action of $\SL(2)$. The closed orbits in $\bR$ correspond to polystable $\SL(2)$-Higgs bundles, i.e. $(E,\phi)$ is stable or $(E,\phi)=(L,\psi)\oplus(L^{-1},-\psi)$ for $L\in\Pic^0(X)$ and $\psi\in H^0(K_{X})$. The set $\bR^{s}$ of stable points with respect to the action of $\SL(2)$ is exactly the locus of stable $\SL(2)$-Higgs bundles.
\end{theorem}

\begin{theorem}[Theorem 4.10 of \cite{Simp94I}]\label{M is a good quotient}
The good quotient $\bR/\!/\SL(2)$ is $\bM$.
\end{theorem}

\begin{theorem}[Theorem 11.1 of \cite{Simp94II}]\label{R and M are quasi-projective}
$\bR$ and $\bM$ are both irreducible normal quasi-projective varieties.
\end{theorem}

\subsection{Hitchin's construction}

Let $E$ be a complex Hermitian vector bundle of rank $2$ and degree $0$ on $X$. Let $\cA$ be the space of traceless connections on $E$ compatible with the Hermitian metric. $\cA$ can be identified with the space of holomorphic structures on $E$ with trivial determinant. Let
$$\cB=\{(A,\phi)\in\cA\times\Om^0(\End_0(E)\otimes K_{X}):d''_{A}\phi=0\}.$$
Let $\cG$ (respectively, $\cG_{\CC}$) be the gauge group of $E$ with structure group $SU(2)$ (respectively, $SL(2)$). These groups act on $\cB$ by
$$g\cdot(A,\phi)=(g^{-1}A'' g+g^{*}A'(g^{*})^{-1}+g^{-1}d'' g-(d' g^{*})(g^{*})^{-1},g^{-1}\phi g),$$
where $A''$ and $A'$ denote the $(0,1)$ and $(1,0)$ parts of $A$ respectively.

The cotangent bundle $T^{*}\cA\cong\cA\times\Om^0(\End_0(E)\otimes K_{X})$ admits a hyperk\"{a}hler structure preserved by the action of $\cG$ with the moment maps for this action
$$\begin{matrix}\mu_{1}=F_{A}+[\phi,\phi^{*}]\\
\mu_{2}=-i(d''_{A}\phi+d'_{A}\phi^{*})\\
\mu_{3}=-d''_{A}\phi+d'_{A}\phi^{*}.\end{matrix}$$
$\mu_{\CC}=\mu_{2}+i\mu_{3}=-2id''_{A}\phi$ is the complex moment map. Then
$$\cB=\mu_{2}^{-1}(0)\cap\mu_{3}^{-1}(0)=\mu_{\CC}^{-1}(0).$$
Consider the hyperk\"{a}hler quotient
$$\cM:=T^{*}\cA/\!/\!/\cG=\mu_{1}^{-1}(0)\cap\mu_{2}^{-1}(0)\cap\mu_{3}^{-1}(0)/\cG=\mu_{1}^{-1}(0)\cap\cB/\cG.$$

Let $\cB^{ss}=\{(A,\phi)\in\cB:((E,d''_{A}),\phi)\text{ is semistable}\}$.

\begin{theorem}[Theorem 2.1 and Theorem 4.3 of \cite{Hit87}, Theorem 1 and Proposition 3.3 of \cite{Simp88}]\label{Hitchin construction}
$$\cM\cong\cB^{ss}/\!/\cG_{\CC}\cong\bM.$$
\end{theorem}

\section{Intersection cohomology theory}

In this section, we introduce some basics on the intersection cohomology (\cite{GM80}, \cite{GM83}) and the equivariant intersection cohomology (\cite{BL94}, \cite{GKM98}) of a quasi-projective complex variety. Let $V$ be a quasi-projective complex variety of pure dimension $n$ throughout this section.

\subsection{Intersection cohomology}

It is well-known that $V$ has a Whitney stratification
$$V=V_{n}\supseteq V_{n-1}\supseteq\cdots\supseteq V_{0}$$
which is embedded into a topological pseudomanifold of dimension $2n$ with filtration
$$W_{2n}\supseteq W_{2n-1}\supseteq\cdots\supseteq W_{0},$$
where $V_{j}$ are closed subvarieties such that $V_{j}-V_{j-1}$ is either empty or a nonsingular quasi-projective variety of pure dimension $j$ and $W_{2k}=W_{2k+1}=V_{k}$.

Let $\bar{p}=(p_{2},p_{3},\cdots,p_{2n})$ be a perversity. For a triangulation $T$ of $V$, $(C_{\bullet}^{T}(V),\partial)$ denotes the chain complex of chains with respect to $T$ with coefficients in $\QQ$. We define $I^{\bar{p}}C_{i}^{T}(V)$ to be the subspace of $C_{i}^{T}(V)$ consisting of those chains $\xi$ such that
$$\dim_{\RR}(|\xi|\cap V_{n-c})\le i-2c+p_{2c}$$
and
$$\dim_{\RR}(|\partial\xi|\cap V_{n-c})\le i-1-2c+p_{2c}.$$
Let $IC_{i}^{\bar{p}}(V)=\dss\varinjlim_{T}I^{\bar{p}}C_{i}^{T}(V)$. Then $(IC_{\bullet}^{\bar{p}}(V),\partial)$ is a chain complex. The $i$-th \textbf{intersection homology} of $V$ of perversity $\bar{p}$, denoted by $IH_{i}^{\bar{p}}(V)$, is the $i$-th homology group of the chain complex $(IC_{\bullet}^{\bar{p}}(V),\partial)$. The $i$-th \textbf{intersection cohomology} of $V$ of perversity $\bar{p}$, denoted by $IH_{\bar{p}}^{i}(V)$, is the $i$-th homology group of the chain complex $(IC_{\bullet}^{\bar{p}}(V)^{\vee},\partial^{\vee})$.

When we consider a chain complex $(IC_{\bullet}^{cl,\bar{p}}(V),\partial)$ of chains with closed support instead of usual chains, we can define the $i$-th \textbf{intersection homology with closed support} (respectively, \textbf{intersection cohomology with closed support}) of $V$ of perversity $\bar{p}$, denoted by $IH_{i}^{cl,\bar{p}}(V)$ (respectively, $IH_{cl,\bar{p}}^{i}(V)$)

There is an alternative way to define the intersection homology and cohomology with closed support. Let $\ic_{\bar{p}}^{-i}(V)$ be the sheaf given by $U\mapsto IC_{i}^{cl,\bar{p}}(U)$ for each open subset $U$ of $V$. Then $\ic_{\bar{p}}^{\bullet}(V)$ is a complex of sheaves as an object in the bounded derived category $D^{b}(V)$.  Then we have $IH_{i}^{cl,\bar{p}}(V)=\cH^{-i}(\ic_{\bar{p}}^{\bullet}(V))$ and $IH_{cl,\bar{p}}^{i}(V)=\cH^{i-2\dim(V)}(\ic_{\bar{p}}^{\bullet}(V))$, where $\cH^{i}(\bA^{\bullet})$ is the $i$-th hypercohomology of a complex of sheaves $\bA^{\bullet}$.

When $\bar{p}$ is the middle perversity $\bar{m}$, $IH_{i}^{\bar{m}}(V)$, $IH_{\bar{m}}^{i}(V)$, $IH_{i}^{cl,\bar{m}}(V)$, $IH_{cl,\bar{m}}^{i}(V)$ and $\ic_{\bar{m}}^{\bullet}(V)$ are denoted by $IH_{i}(V)$, $IH^{i}(V)$, $IH_{i}^{cl}(V)$, $IH_{cl}^{i}(V)$ and $\ic^{\bullet}(V)$ respectively.

\subsection{Equivariant intersection cohomology}

Assume that a compact connected algebraic group $G$ acts on $V$ algebraically. For the universal principal bundle $\E G\to \B G$, we have the quotient $V\times_{G}\E G$ of $V\times \E G$ by the diagonal action of $G$. Let us consider the following diagram
$$\xymatrix{V&V\times \E G\ar[l]_{p}\ar[r]^{q}&V\times_{G}\E G}.$$

\begin{definition}[2.1.3 and 2.7.2 in \cite{BL94}]
The \textbf{equivariant derived category} of $V$, denoted by $D_{G}^{b}(V)$, is defined as follows:
\begin{enumerate}
\item An object is a triple $(F_{V},\bar{F},\beta)$, where $F_{V}\in D^{b}(V)$, $\bar{F}\in D^{b}(V\times_{G}\E G)$ and $\beta:p^{*}(F_{V})\to q^{*}(\bar{F})$ is an isomorphism in $D^{b}(V\times \E G)$.

\item A morphism $\alpha:(F_{V},\bar{F},\beta)\to(G_{V},\bar{G},\gamma)$ is a pair $\alpha=(\alpha_{V},\bar{\alpha})$, where $\alpha_{V}:F_{V}\to G_{V}$ and $\bar{\alpha}:\bar{F}\to\bar{G}$ such that $\beta\circ p^{*}(\alpha_{V})=q^{*}(\bar{\alpha})\circ\beta$.
\end{enumerate}
\end{definition}

$\ic_{G,\bar{p}}^{\bullet}(V)$ (respectively, $\QQ_{V}^{G}$) denotes $(\ic_{\bar{p}}^{\bullet}(V),\ic_{\bar{p}}^{\bullet}(V\times_{G}\E G),\beta)$ (respectively, $(\QQ_{V},\QQ_{V\times_{G}\E G},\id)$) as an object of $D_{G}^{b}(V)$. The $i$-th equivariant cohomology of $V$ can be obtained by $H_{G}^{i}(V)=\cH^{-i}(\QQ_{V\times_{G}\E G})$. The \textbf{equivariant intersection cohomology} of $V$ of perversity $\bar{p}$, denoted by $IH_{G,\bar{p}}^{*}(V)$, is defined by $IH_{G,\bar{p}}^{i}(V):=\cH^{-i}(\ic_{\bar{p}}^{\bullet}(V\times_{G}\E G))$.

When $\bar{p}$ is the middle perversity $\bar{m}$, $IH_{G,\bar{m}}^{i}(V)$ and $\ic_{G,\bar{m}}^{\bullet}(V)$ are denoted by $IH_{G}^{i}(V)$ and $\ic_{G}^{\bullet}(V)$ respectively.

The equivariant cohomology and the equivariant intersection cohomology can be described as a limit of a projective limit system coming from a sequence of finite dimensional submanifolds of $\E G$. Let us consider a sequence of finite dimensional submanifolds $\E G_{0}\subset \E G_{1}\subset\cdots\subset \E G_{n}\subset\cdots$ of $\E G$, where $G$ acts on all of $\E G_{n}$ freely, $\E G_{n}$ are $n$-acyclic, $\E G_{n}\subset \E G_{n+1}$ is an embedding of a submanifold, $\dim \E G_{n}<\dim \E G_{n+1}$ and $\dss \E G=\bigcup_{n}\E G_{n}$. Since $G$ is connected, such a sequence exists by \cite[Lemma 12.4.2]{BL94}. Then we have a sequence of finite dimensional subvarieties $V\times_{G}\E G_{0}\subset V\times_{G}\E G_{1}\subset\cdots\subset V\times_{G}\E G_{n}\subset\cdots$ of $V\times_{G}\E G$. Hence we have $\dss H_{G}^{*}(V)=\varprojlim_{n}H^{*}(V\times_{G}\E G_{n})$ and $\dss IH_{G,\bar{p}}^{*}(V)=\varprojlim_{n}IH_{\bar{p}}^{*}(V\times_{G}\E G_{n})$.

\subsection{The generalized Poincar\'{e} duality and the decomposition theorem}

In this subsection, we state two important theorems. One is the generalized Poincar\'{e} duality and the other is the decomposition theorem.

\begin{theorem}[The generalized Poincar\'{e} duality]\label{GPD}
If $\bar{p}+\bar{q}=\bar{t}$, then there is a non-degenerate bilinear form
$$IH_{i}^{\bar{p}}(V)\times IH_{2\dim(V)-i}^{cl,\bar{q}}(V)\to\QQ.$$
\end{theorem}

\begin{theorem}[The decomposition theorem]\label{DT}

Suppose that $\varphi:W\to V$ is a projective morphism of quasi-projective varieties. Then there exist closed subvarieties $V_{\alpha}$ of $V$, local systems $L_{\alpha}$ on the non-singular parts $(V_{\alpha})_{nonsing}$ of $V_{\alpha}$ and integers $l_{\alpha}$ such that there is an isomorphism
$$R\varphi_{*}\ic^{\bullet}(W)\cong\bigoplus_{\alpha}\ic^{\bullet}(V_{\alpha},L_{\alpha})[l_{\alpha}]$$
in the derived category $D^{b}(V)$, where $\ic^{\bullet}(V_{\alpha},L_{\alpha})$ is the complex of sheaves of intersection chains with coefficients in $L_{\alpha}$.

\end{theorem}

There are three special important consequences of the decomposition theorem.

\begin{proposition}\label{3 consequences}
\begin{enumerate}
\item Suppose that $\varphi:W\to V$ is a resolution of singularities. Then $\ic^{\bullet}(V)$ (respectively, $IH^{*}(V)$) is a direct summand of $R\varphi_{*}\ic^{\bullet}(W)$ (respectively, $IH^{*}(W)$).

\item Suppose that $\varphi:W\to V$ is a projective morphism which is topologically a fibration whose fiber is a projective variety $F$. Then there is a Leray spectral sequence $E_{r}^{ij}$ converging to $IH^{i+j}(W)$ with $E_2$ term $E_{2}^{ij}=IH^{i}(V,IH^{j}(F))$, where $IH^{j}(F)$ denotes the local system $\cL$ on $V$ with stalk $\cL_{x}=IH^{j}(\varphi^{-1}(x))\cong IH^{j}(F)$. The decomposition theorem for $\varphi$ is equivalent to the degeneration of $E_{r}^{ij}$ at the $E_{2}$ term.

\item Suppose that $\varphi:W\to V$ is a $G$-equivariant resolution of singularities. Then $\ic_{G}^{\bullet}(V)$ (respectively, $IH_{G}^{*}(V)$) is a direct summand of $R\varphi_{*}\ic_{G}^{\bullet}(W)$ (respectively, $IH_{G}^{*}(W)$).
\end{enumerate}
\end{proposition}
\begin{proof}
\begin{enumerate}
\item Applying Theorem \ref{DT} to $\varphi$ and to the shifted constant sheaf $\QQ_{W}[\dim W]$, we get the result. The details of the proof can be found in \cite[Corollary 5.4.11]{Dim04}.

\item The statement is from \cite[Proposition 8.4.5]{KW06}.

\item We know that $\ic_{G}^{\bullet}(V)=(\ic^{\bullet}(V),\ic^{\bullet}(V\times_{G}\E G),\alpha)$ and $\ic_{G}^{\bullet}(W)=(\ic^{\bullet}(W),\ic^{\bullet}(W\times_{G}\E G),\beta)$. It follows from item (1) that $\ic^{\bullet}(V)$ is a direct summand of $R\varphi_{*}\ic^{\bullet}(W)$ and that $\ic^{\bullet}(V\times_{G}\E G_{n})$ is a direct summand of $R\varphi_{*}\ic^{\bullet}(W\times_{G}\E G_{n})$ for all $n$. Since $\ic^{\bullet}(V\times_{G}\E G)=\dss\varprojlim_{n}\ic^{\bullet}(V\times_{G}\E G_{n})$ and $\ic^{\bullet}(W\times_{G}\E G)=\dss\varprojlim_{n}\ic^{\bullet}(W\times_{G}\E G_{n})$, $\ic^{\bullet}(V\times_{G}\E G)$ is a direct summand of $R\varphi_{*}\ic^{\bullet}(W\times_{G}\E G)$.

Let $i:\ic^{\bullet}(V)\hookrightarrow R\varphi_{*}\ic^{\bullet}(W)$ and $\bar{i}:\ic^{\bullet}(V\times_{G}\E G)\hookrightarrow R\varphi_{*}\ic^{\bullet}(W\times_{G}\E G)$ be the inclusions from the decomposition theorem. It is easy to see that the following diagram
$$\xymatrix@C-=0.5cm{&p_{V}^{*}\ic^{\bullet}(V)\ar[r]^(0.4){\alpha}\ar[d]_{p^{*}(i)}&q_{V}^{*}\ic^{\bullet}(V\times_{G}\E G)\ar[d]^{q^{*}(\bar{i})}&\\
R\varphi_{*}p_{W}^{*}\ic^{\bullet}(W)\ar@{=}[r]&p_{V}^{*}R\varphi_{*}\ic^{\bullet}(W)\ar[r]_{R\varphi_{*}(\beta)\quad\quad}&q_{V}^{*}R\varphi_{*}\ic^{\bullet}(W\times_{G}\E G)\ar@{=}[r]&R\varphi_{*}q_{W}^{*}\ic^{\bullet}(W\times_{G}\E G)}$$
commutes, where $p_{V}:V\times \E G\to V$ (respectively, $p_{W}:W\times \E G\to W$) is the projection onto $V$ (respectively, $W$) and $q_{V}:V\times \E G\to V\times_{G}\E G$ (respectively, $q_{W}:W\times \E G\to W\times_{G}\E G$) is the quotient.
\end{enumerate}
\end{proof}

\section{Kirwan's desingularization of $\bM$}\label{Kirwan desing}

In this section, we briefly explain how $\bM$ can be desingularized by three blowing-ups by the Kirwan's algorithm introduced in \cite{K85-2}. For details, see \cite{KY08} and \cite{O99}.

We first consider the loci of type (i) of $(L,0)\oplus(L,0)$ with $L\cong L^{-1}$ in $\bM\setminus\bM^{s}$ and in $\bR\setminus\bR^{s}$, where $\bR^{s}$ is the stable locus of $\bR$. The loci of type (i) in $\bM$ and in $\bR$ are both isomorphic to the set of $\ZZ_2$-fixed points $\ZZ_{2}^{2g}$ in $J:=\Pic^0(X)$ by the involution $L\mapsto L^{-1}$. The singularity of the locus $\ZZ_{2}^{2g}$ of type (i) in $\bM$ is the quotient
$$\Upsilon^{-1}(0)/\!/\SL(2)$$
where $\Upsilon:[H^0(K_{X})\oplus H^1(\cO_{X})]\otimes sl(2)\to H^1(K_{X})\otimes sl(2)$ is the quadratic map given by the Lie bracket of $sl(2)$ coupled with the perfect pairing $H^0(K_{X})\oplus H^1(\cO_{X})\to H^1(K_{X})$ and the $\SL(2)$-action on $\Upsilon^{-1}(0)$ is induced from the adjoint representation $\SL(2)\to\Aut(sl(2))$.

Next we consider the loci of type (iii) of $(L,\psi)\oplus (L^{-1},-\psi)$ with $(L,\psi)\ncong (L^{-1},-\psi)$ in $\bM\setminus\bM^{s}$ and in $\bR\setminus\bR^s$. It is clear that the
locus of type (iii) in $\bM$ is isomorphic to
$$J\times_{\ZZ_2}H^0(K_X)-\ZZ_2^{2g}\cong T^*J/\ZZ_2-\ZZ_2^{2g}$$
where $\ZZ_2$ acts on $J$ by $L\mapsto L^{-1}$ and on $H^0(K_X)$
by $\psi\mapsto -\psi$. The locus of type (iii) in $\bR$ is a $\PP
\SL(2)/\CC^*$-bundle over $T^*J/\ZZ_2-\ZZ_2^{2g}$ and in particular
it is smooth.  The singularity along the
locus of type (iii) in $\bM$ is the quotient
\[ \Psi^{-1}(0)/\!/\CC^*,\]
where $\Psi:[H^0(L^{-2}K_X)\oplus H^1(L^2)]\oplus[H^0(L^{2}K_X)\oplus H^1(L^{-2})]\to H^1(K_X)$ is the quadratic map given by the sum of perfect pairings $H^0(L^{-2}K_X)\oplus H^1(L^2)\to H^1(K_X)$ and $H^0(L^{2}K_X)\oplus H^1(L^{-2})\to H^1(K_X)$ over $(L,\psi)\oplus (L^{-1},-\psi)\in T^*J/\ZZ_2-\ZZ_2^{2g}$ and the $\CC^{*}$-action on $\Psi^{-1}(0)$ is induced from the $\CC^{*}$-action on $[H^0(L^{-2}K_X)\oplus H^1(L^2)]\oplus[H^0(L^{2}K_X)\oplus H^1(L^{-2})]$ given by
$$\lambda\cdot(a,b,c,d)=(\lambda^{-2}a,\lambda^{2}b,\lambda^{2}c,\lambda^{-2}d).$$

Since we have identical singularities as in \cite{O99}, we can follow K.G. O'Grady's arguments to construct the Kirwan's desingularization $\bK$ of $\bM$. Let $\bR_{1}$ be the blowing-up of $\bR$ along the locus $\ZZ_{2}^{2g}$ of type (i). Let $\bR_{2}$ be the blowing-up of $\bR_{1}^{ss}$ along the strict transform $\Sigma$ of the locus of type (iii), where $\bR_{1}^{ss}$ is the locus of semistable points in $\bR_{1}$. Let $\bR_{2}^{ss}$ (respectively, $\bR_{2}^{s}$) be the locus of semistable (respectively, stable) points in $\bR_{2}$. Then it follows from the same argument as in \cite[Claim 1.8.10]{O99} that

\begin{enumerate}
\item[(a)] $\bR_{2}^{ss}=\bR_{2}^{s}$,
\item[(b)] $\bR_{2}^{s}$ is smooth.
\end{enumerate}

In particular, $\bR_{2}^{s}/\SL(2)$ has at worst orbifold singularities. When $g=2$, this is smooth. When $g\ge3$, we blow up $\bR_{2}^{s}$ along the locus of points with stabilizers larger than the center $\ZZ_{2}$ of $\SL(2)$ to obtain a variety $\bR_{3}$ such that the orbit space $\bK:=\bR_{3}^{s}/\SL(2)$ is a smooth variety obtained by blowing up $\bM$ along $\ZZ_{2}^{2g}$, along the strict transform of $T^*J/\ZZ_2-\ZZ_2^{2g}$ and along a nonsingular subvariety contained in the strict transform of the exceptional divisor of the first blowing-up. $\bK$ is called the \textbf{Kirwan's desingularization} of $\bM$.

Throughout this paper, $\pi_{\bR_{1}}:\bR_{1}\to\bR$ (respectively, $\pi_{\bR_{2}}:\bR_{2}\to\bR_{1}^{ss}$) denotes the first blowing-up map (respectively, the second blowing-up map). $\overline{\pi}_{\bR_{1}}:\bR_{1}^{ss}/\!/\SL(2)\to\bR/\!/\SL(2)$ and $\overline{\pi}_{\bR_{2}}:\bR_{2}^{ss}/\!/\SL(2)\to\bR_{1}^{ss}/\!/\SL(2)$ denote maps induced from $\pi_{\bR_{1}}$ and $\pi_{\bR_{2}}$ respectively.

\section{Local pictures in Kirwan's algorithm on $\bR$}\label{local picture}

In this section, we list local pictures that appear in Kirwan's algorithm on $\bR$ for later use. For details, see \cite[1.6 and 1.7]{O99}.

We first observe that $\pi_{\bR_{1}}^{-1}(x)=\PP\Upsilon^{-1}(0)$ for any $x\in\ZZ_{2}^{2g}$. We identify $\HH^{g}$ with $T_{x}(T^{*}J)=H^{1}(\cO_X)\oplus H^{0}(K_X)$ for any $x\in T^{*}J$, where $\HH$ is the division algebra of quaternions. Since the adjoint representation gives an identification $\PGL(2)\cong \SO(sl(2))$, $\PGL(2)$ acts on both $\Upsilon^{-1}(0)$ and $\PP\Upsilon^{-1}(0)$. Since $\PGL(2)=\SL(2)/\{\pm\id\}$ and the action of $\{\pm\id\}$ on $\Upsilon^{-1}(0)$ and $\PP\Upsilon^{-1}(0)$ are trivial,
$$\Upsilon^{-1}(0)/\!/\SL(2)=\Upsilon^{-1}(0)/\!/\PGL(2)\text{ and }\PP\Upsilon^{-1}(0)^{ss}/\!/\SL(2)=\PP\Upsilon^{-1}(0)^{ss}/\!/\PGL(2).$$

We have an explicit description of semistable points of $\PP\Upsilon^{-1}(0)$ with respect to the $\PGL(2)$-action as following.

\begin{proposition}[Proposition 1.6.2 of \cite{O99}]\label{ss-local-first-blowup}
A point $[\varphi]\in\PP\Upsilon^{-1}(0)$ is $\PGL(2)$-semistable if and only
if:
$$\rk\varphi\begin{cases}\geq 2,&\text{or}\\
=1&\text{and }[\varphi]\in\PGL(2)\cdot\PP\{\left(\begin{array}{cc}\lambda&0\\0&-\lambda\end{array}\right)|\lambda\in\HH^{g}\setminus\{O\}\}.
\end{cases}$$
\end{proposition}

Let $\Hom^{\omega}(sl(2),\HH^{g}):=\{\varphi:sl(2)\rightarrow\HH^{g}|\varphi^{*}\omega=0\}$, where $\omega$ is the Serre duality pairing on $\HH^{g}$. Let $(m,n)=4\Tr(mn)$ be the killing form on $sl(2)$. The killing form gives isomorphisms
$$\HH^{g}\otimes sl(2)\cong\Hom(sl(2),\HH^{g})\text{ and }sl(2)\cong\wedge^{2}sl(2)^{\vee}.$$
By the above identification, $\Upsilon:\Hom(sl(2),\HH^{g})\to\wedge^{2}sl(2)^{\vee}$ is given by $\varphi\mapsto\varphi^{*}\omega$. Then we have
$$\Upsilon^{-1}(0)=\Hom^{\omega}(sl(2),\HH^{g}).$$
Let
$$\Hom_{k}(sl(2),\HH^{g}):=\{\varphi\in\Hom(sl(2),\HH^{g})|\rk\varphi\leq k\}$$
and
$$\Hom_{k}^{\omega}(sl(2),\HH^{g}):=\Hom_{k}(sl(2),\HH^{g})\cap\Hom^{\omega}(sl(2),\HH^{g}).$$

We have a description of
points of $E_{1}\cap\Sigma$ as following.

\begin{proposition}[Lemma 1.7.5 of \cite{O99}]\label{intersection-of-1st-exc-and-2nd-center}
Let $x\in\ZZ_2^{2g}$. Then
$$\pi_{\bR_{1}}^{-1}(x)\cap\Sigma=\PP\Hom_{1}(sl(2),\HH^{g})^{ss},$$
where $\PP\Hom_{1}(sl(2),\HH^{g})^{ss}$ denotes the set of semistable points of $\PP\Hom_{1}(sl(2),\HH^{g})$ with respect to the $\PGL(2)$-action.
\end{proposition}

Assume that $\varphi\in\Hom_{1}(sl(2),\HH^{g})$. Since the Serre duality pairing is skew-symmetric, we can choose bases $\{e_{1},\cdots,e_{2g}\}$ of $\HH^{g}$ and $\{v_{1},v_{2},v_{3}\}$ of $sl(2)$ such that $\varphi=e_{1}\otimes v_{1}$ and so that
$$<e_{i},e_{j}>=\begin{cases}1 &\text{if }i=2q-1,j=2q,q=1,\cdots,g,\\
-1&\text{if }i=2q,j=2q-1,q=1,\cdots,g,\\
0&\text{otherwise}.\end{cases}$$
Every element in $\Hom(sl(2),\HH^{g})$ can be written as
$\sum_{i,j}Z_{ij}e_{i}\otimes  v_{j}$. Then we have a description of the normal cone $C_{\PP\Hom_1(sl(2),\HH^{g})}\PP\Upsilon^{-1}(0)$ to $\PP\Hom_{1}(sl(2),\HH^{g})$ in $\PP\Upsilon^{-1}(0)$.

\begin{proposition}\label{normalcone-localmodel}
Let $[\varphi]\in\PP\Hom_1(sl(2),\HH^{g})$ and let
$\omega^{\varphi}$ be the bilinear form induced by $\omega$ on
$\im\varphi^{\perp}/\im\varphi$. There is a
$\st([\varphi])$-equivariant isomorphism
$$(C_{\PP\Hom_1(sl(2),\HH^{g})}\PP\Upsilon^{-1}(0))|_{[\varphi]}\cong\Hom^{\omega_{\varphi}}(\ker\varphi,\im\varphi^{\perp}/\im\varphi)$$
where
$$\Hom^{\omega_{\varphi}}(\ker\varphi,\im\varphi^{\perp}/\im\varphi)=\{\chi\in\Hom(\ker\varphi,\im\varphi^{\perp}/\im\varphi)|\chi^*\omega^{\varphi}=0\}$$
\end{proposition}
\begin{proof}
Following the idea of proof of \cite[Lemma 1.7.13]{O99}, both sides are defined by the equation
$$\sum_{2\leq q\leq g}(Z_{2q-1,2}Z_{2q,3}-Z_{2q,2}Z_{2q-1,3})=0.$$
under the choice of basis as above.
\end{proof}

We now explain how $\st([\varphi])$ acts on $(C_{\PP\Hom_{1}(sl(2),\HH^{g})}\PP\Upsilon^{-1}(0))_{[\varphi]}$. If we add the condition that
$$\left.\begin{array}{ccc}(v_{1},v_{i})=-\delta_{1i}\\(v_{j},v_{j})=0,&j=2,3 \\(v_{2},v_{3})=1, \end{array}\right.$$
and $v_{1}\wedge v_{2}\wedge v_{3}$ is the volume form, where
$\wedge$ corresponds to the Lie bracket in $sl(2)$, then $\st([\varphi])=\O(\ker\varphi)=\O(2)$ is generated by
$$\{\theta_{\lambda}:=\left(\begin{array}{ccc}1&0&0\\0&\lambda&0\\0&0&\lambda^{-1}\end{array}\right)|\lambda\in\CC^{*}\}\text{ and }\tau:=\left(\begin{array}{ccc}-1&0&0\\0&0&1\\0&1&0\end{array}\right)$$
as a subgroup of $\SO(sl(2))$. $\O(2)$ can be also realized as the
subgroup of $\PGL(2)$ generated by
$$\SO(2)=\big\{\theta_{\lambda}:=\left(\begin{array}{cc}\lambda&0\\0&\lambda^{-1}\end{array}\right)|\lambda\in\CC^{*}\big\}/\{\pm\id\},\quad \tau=\left(\begin{array}{cc}0&1\\1&0\end{array}\right).$$
The action of $\st([\varphi])$ on
$(C_{\PP\Hom_{1}(sl(2),\HH^{g})}\PP\Upsilon^{-1}(0))_{[\varphi]}$
is given by
$$\theta_{\lambda}(\sum_{i=3}^{2g}(Z_{i,2}e_{i}\otimes v_{2}+Z_{i,3}e_{i}\otimes v_{3}))=\sum_{i=3}^{2g}(\lambda Z_{i,2}e_{i}\otimes v_{2}+\lambda^{-1}Z_{i,3}e_{i}\otimes v_{3}),$$
$$\tau(\sum_{i=3}^{2g}(Z_{i,2}e_{i}\otimes v_{2}+Z_{i,3}e_{i}\otimes v_{3}))=\sum_{i=3}^{2g}(-Z_{i,3}e_{i}\otimes v_{2}-Z_{i,2}e_{i}\otimes v_{3}).$$

Let us consider the blowing-up
$$\pi:Bl_{\PP\Hom_{1}}\PP\Upsilon^{-1}(0)^{ss}\to\PP\Upsilon^{-1}(0)^{ss}$$
of $\PP\Upsilon^{-1}(0)^{ss}$ along $\PP\Hom_{1}(sl(2),\HH^{g})^{ss}$ with the exceptional divisor $E_{\pi}$, where $\PP\Upsilon^{-1}(0)^{ss}$ is the locus of semistable points of $\PP\Upsilon^{-1}(0)$ with respect to the $\PGL(2)$-action. It is obvious that $(\pi_{\bR_{1}}\circ\pi_{\bR_{2}})^{-1}(x)=Bl_{\PP\Hom_{1}}\PP\Upsilon^{-1}(0)^{ss}$ for any $x\in\ZZ_{2}^{2g}$.

\begin{proposition}[Lemma 1.8.5 of \cite{O99}]\label{smoothness of 2nd local blowing-up}
$Bl_{\PP\Hom_{1}}\PP\Upsilon^{-1}(0)^{ss}$ is smooth.
\end{proposition}

\begin{proposition}[Lemma 1.8.6 of \cite{O99}]\label{ss of 2nd local blowing-up}
All semistable points of $Bl_{\PP\Hom_{1}}\PP\Upsilon^{-1}(0)^{ss}$ is stable. Explicitely:
\begin{enumerate}
\item Semistable points in $E_{\pi}$ are given by
$$\{([\varphi],[\alpha])\,|\,[\varphi]\in\PP\Hom_{1}(sl(2),\HH^{g})^{ss},[\alpha]\in\PP\Hom^{\omega_{\varphi}}(\ker\varphi,\im\varphi^{\perp}/\im\varphi),\alpha(v_{2})\ne0\ne\alpha(v_{3})\}.$$
\item Semistable points not in $E_{\pi}$ are given by
$$\{[\varphi]\in\PP\Hom^{\om}(sl(2),\HH^{g})\,|\,\rk\varphi=3\text{ or }\rk\varphi=2\text{ and }\ker\varphi\text{ non-isotropic}\}.$$
\end{enumerate}
\end{proposition}

\section{Blowing-up formula for intersection cohomology}\label{blowing-up formula for intersection cohomology}

In this section, we prove that a blowing-up formula for the intersection cohomology holds in Kirwan's algorithm introduced in Section \ref{Kirwan desing}.

Let $E_{1}$ (respectively, $E_{2}$) be the exceptional divisor of $\pi_{\bR_{1}}$ (respectively, $\pi_{\bR_{2}}$). Let $\cC_{1}$ be the normal cone to $\ZZ_{2}^{2g}$ in $\bR$, $\tcC_{1}$ the normal cone to $E_{1}^{ss}:=E_{1}\cap\bR_{1}^{ss}$ in $\bR_{1}^{ss}$, $\cC_{2}$ the normal cone to $\Sigma$ in $\bR_{1}$, $\tcC_{2}$ the normal cone to $E_{2}^{ss}:=E_{2}\cap\bR_{2}^{ss}$ in $\bR_{2}^{ss}$, $\cC$ the normal cone to $\PP\Hom_{1}(sl(2),\HH^{g})^{ss}$ in $\PP\Upsilon^{-1}(0)^{ss}$ and $\tcC$ the normal cone to $E_{\pi}^{ss}:=E_{\pi}\cap(Bl_{\PP\Hom_{1}}\PP\Upsilon^{-1}(0)^{ss})^{ss}$ in $(Bl_{\PP\Hom_{1}}\PP\Upsilon^{-1}(0)^{ss})^{ss}$, where $(Bl_{\PP\Hom_{1}}\PP\Upsilon^{-1}(0)^{ss})^{ss}$ is the locus of semistable points of $Bl_{\PP\Hom_{1}}\PP\Upsilon^{-1}(0)^{ss}$ with respect to the lifted $\PGL(2)$-action. Note that
$$(Bl_{\PP\Hom_{1}}\PP\Upsilon^{-1}(0)^{ss})^{ss}=(Bl_{\PP\Hom_{1}}\PP\Upsilon^{-1}(0)^{ss})^{s},$$
where $(Bl_{\PP\Hom_{1}}\PP\Upsilon^{-1}(0)^{ss})^{s}$ is the locus of stable points of $Bl_{\PP\Hom_{1}}\PP\Upsilon^{-1}(0)^{ss}$ with respect to the $\PGL(2)$-action (\cite[Lemma 1.8.6]{O99}). Then we have the following formulas.

\begin{lemma}\label{intersection blowing-up formula}
(1) $\dim IH^{i}(\bR_{1}^{ss}/\!/\SL(2))=\dim IH^{i}(\bR/\!/\SL(2))$
$$+\dim IH^{i}(\tcC_{1}/\!/\SL(2))-\dim IH^{i}(\cC_1/\!/\SL(2))$$
$$=\dim IH^{i}(\bR/\!/\SL(2))+2^{2g}\dim IH^{i}(Bl_{0}\Upsilon^{-1}(0)/\!/\PGL(2))-2^{2g}\dim IH^{i}(\Upsilon^{-1}(0)/\!/\PGL(2))$$
for all $i\ge0$, where $Bl_{0}\Upsilon^{-1}(0)$ is the blowing-up of $\Upsilon^{-1}(0)$ at the vertex.

(2) $\dim IH^{i}(\bR_{2}^{s}/\SL(2))=\dim IH^{i}(\bR_{1}^{ss}/\!/\SL(2))$
$$+\dim IH^{i}(\tcC_{2}/\!/\SL(2))-\dim IH^{i}(\cC_2/\!/\SL(2))$$
for all $i\ge0$.

(3)\label{local second intersection blowing-up formula} $\dim IH^{i}(Bl_{\PP\Hom_{1}}\PP\Upsilon^{-1}(0)^{ss}/\!/\SL(2))=\dim IH^{i}(\PP\Upsilon^{-1}(0)^{ss}/\!/\SL(2))$
$$+\dim IH^{i}(\tcC/\!/\SL(2))-\dim IH^{i}(\cC/\!/\SL(2))$$
for all $i\ge0$.
\end{lemma}

For the proof, we need to review a useful result by C.T. Simpson. Let $A^i$ (respectively, $A^{i,j}$) be the sheaf of smooth $i$-forms
(respectively, $(i,j)$-forms) on $X$. For a polystable Higgs bundle
$(E,\phi)$, consider the complex
\begin{equation}\label{e4.1}
0\to \End_{0}(E)\otimes A^0\to \End_{0}(E)\otimes A^1\to \End_{0}(E)\otimes
A^2\to 0\end{equation} whose differential is given by
$D''=\overline{\partial}+\phi$. Because $A^1=A^{1,0}\oplus
A^{0,1}$ and $\phi$ is of type $(1,0)$, we have an exact sequence
of complexes with \eqref{e4.1} in the middle
$$\xymatrix{
& 0\ar[d] & 0\ar[d] &0\ar[d] &    \\
0\ar[r] & 0\ar[r]\ar[d] &\End_{0}(E)\otimes
A^{1,0}\ar[r]^{\overline{\partial}}\ar[d] &\End_{0}(E)\otimes
A^{1,1}\ar[r]\ar[d]^= &0\\
0\ar[r]& \End_{0}(E)\otimes A^0\ar[r]^{D''}\ar[d]_= & \End_{0}(E)\otimes
A^1\ar[r]^{D''}\ar[d] & \End_{0}(E)\otimes A^2\ar[r]\ar[d] & 0\\
0\ar[r] & \End_{0}(E)\otimes A^{0,0}\ar[r]^{\overline{\partial}}\ar[d]
&
\End_{0}(E)\otimes A^{0,1}\ar[r]\ar[d] & 0\ar[r]\ar[d] & 0 \\
& 0 & 0  &0 }$$
This gives us a long exact sequence
\[\xymatrix{ 0\ar[r] & T^0\ar[r] &H^0(\End_{0}(E))\ar[r]^(.42){[\phi,-]}
& H^0(\End_{0}(E)\otimes K_X)\ar[r]& }\]
\[\xymatrix{ \ar[r]& T^1\ar[r] &H^1(\End_{0}(E))\ar[r]^(.42){[\phi,-]}
& H^1(\End_{0}(E)\otimes K_X)\ar[r]  &T^2\ar[r] &0 }
\]
where $T^i$ is the $i$-th cohomology of \eqref{e4.1}. The Zariski
tangent space of $\bM$ at polystable $(E,\phi)$ is isomorphic to
$T^1$.

\begin{proposition}[Theorem 10.4 and 10.5 of \cite{Simp94II}]\label{normal slice is the quadratic cone}
Let $C$ be the quadratic cone in $T^1$
defined by the map $T^1\to T^2$ which sends an $\End_{0}(E)$-valued 1-form $\eta$ to
$[\eta,\eta]$. Let $y=(E,\phi,\beta:E|_{x}\to\CC^{2})\in \bR$ be a point with
closed orbit and $\bar{y}\in\bM$ the image of $y$. Then the formal completion ${(\bR,y)}^\wedge$ is
isomorphic to the formal completion ${(C\times
\mathfrak{h}^\perp,0)}^\wedge$ where $\mathfrak{h}^\perp$ is the
perpendicular space to the image of $T^0\to H^0(\End_{0}(E))\to
sl(2)$. Furthermore, if we let $Y$ be the \'etale slice at $y$ of
the $\SL(2)$-orbit in $\bR$, then
$(Y,y)^\wedge \cong (C,0)^\wedge$ and $(\bM,\bar{y})^\wedge=(Y/\!/\st(y),\bar{y})^\wedge\cong(C/\!/\st(y),v)^\wedge$ where $\st(y)$ is the stabilizer of $y$ and $v$ is the cone point of $C$.
\end{proposition}

\begin{proof}[Proof of Lemma \ref{intersection blowing-up formula}]
\begin{enumerate}
\item Let $U_{x}$ be a sufficiently small open neighborhood of $x\in\ZZ_{2}^{2g}$ in $\bR/\!/\SL(2)$, let $\displaystyle U_{1}=\sqcup_{x\in\ZZ_{2}^{2g}}U_{x}$ and $\tU_{1}=\overline{\pi}_{\bR_{1}}^{-1}(U_{1})$. By the same argument as in the proof of \cite[Lemma 2.8]{K86}, we have
    $$\dim IH^{i}(\bR_{1}^{ss}/\!/\SL(2))=\dim IH^{i}(\bR/\!/\SL(2))+\dim IH^{i}(\tU_{1})-\dim IH^{i}(U_{1})$$
    for all $i\ge0$. By \cite[Theorem 3.1]{GM88} and Proposition \ref{normal slice is the quadratic cone}, there is an analytic isomorphism $U_{1}\cong\cC_1/\!/\SL(2)$. Since $\tcC_1/\!/\SL(2)$ is naturally isomorphic to the blowing-up of $\cC_1/\!/\SL(2)$ along $\ZZ_{2}^{2g}$, we also have an analytic isomorphism $\tU_{1}\cong\tcC_1/\!/\SL(2)$. Since $\cC_1/\!/\SL(2)$ (respectively, $\tcC_1/\!/\SL(2)$) is the $2^{2g}$ copy of
    \begin{center}$\Upsilon^{-1}(0)/\!/\PGL(2)$ (respectively, of $Bl_{0}\Upsilon^{-1}(0)/\!/\PGL(2)$),\end{center}
    we get the formula.

\item Let $U_{2}$ be a sufficiently small open neighborhood of the strict transform of $T^*J/\ZZ_2$ in $\bR_{1}^{ss}/\!/\SL(2)$ and let $\tU_{2}=\overline{\pi}_{\bR_{2}}^{-1}(U_{2})$. By the same argument as in the proof of \cite[Lemma 2.8]{K86}, we have
    $$\dim IH^{i}(\bR_{2}^{s}/\SL(2))=\dim IH^{i}(\bR_{1}^{ss}/\!/\SL(2))+\dim IH^{i}(\tU_{2})-\dim IH^{i}(U_{2})$$
    for all $i\ge0$. By \cite[Theorem 3.1]{GM88} and Proposition \ref{normal slice is the quadratic cone} and Luna's \'{e}tale slice theorem, there is an analytic isomorphism $U_{2}\cong\cC_2/\!/\SL(2)$. Since $\tcC_2/\!/\SL(2)$ is naturally isomorphic to the blowing-up of $\cC_2/\!/\SL(2)$ along the strict transform of $T^*J/\ZZ_2$ in $\bR_{1}^{ss}/\!/\SL(2)$, we also have an analytic isomorphism $\tU_{2}\cong\tcC_2/\!/\SL(2)$. Hence we get the formula.

\item Let $\overline{\pi}:Bl_{\PP\Hom_{1}}\PP\Upsilon^{-1}(0)^{ss}/\!/\PGL(2)\to\PP\Upsilon^{-1}(0)^{ss}/\!/\PGL(2)$ be the map induced from $\pi$. Since $\cC=\cC_{2}|_{\Sigma\cap\PP\Upsilon^{-1}(0)^{ss}}$ and $\tcC=\tcC_{2}|_{E_{2}\cap Bl_{\PP\Hom_{1}}\PP\Upsilon^{-1}(0)^{ss}}$, it follow from the argument of the proof of item (2) that $\cC/\!/\SL(2)$ (respectively, $\tcC/\!/\SL(2)$) can be identified with an open neighborhood $U$ of $\PP\Hom_{1}(sl(2),\HH^{g})^{ss}/\!/\PGL(2)$ (respectively, with $\overline{\pi}^{-1}(U)$). Again by the same argument as in the proof of \cite[Lemma 2.8]{K86}, we get the formula.
\end{enumerate}
\end{proof}

We give a computable fomula from Lemma \ref{intersection blowing-up formula} by more analysis on $Bl_{0}\Upsilon^{-1}(0)/\!/\PGL(2)$, $\cC_{2}/\!/\SL(2)$, $\tcC_{2}/\!/\SL(2)$, $\cC/\!/\SL(2)$ and $\tcC/\!/\SL(2)$.

We first give explicit geometric descriptions for $\cC_{2}/\!/\SL(2)$, $\tcC_{2}/\!/\SL(2)$, $\cC/\!/\SL(2)$ and $\tcC/\!/\SL(2)$. Let $\alpha:\widetilde{T^{*}J}\to T^{*}J$ be the blowing-up along $\ZZ_{2}^{2g}$. Let $(\cL,\psi_{\cL})$ be the pull-back to $\widetilde{T^{*}J}\times X$ of the universal pair on $T^{*}J\times X$ by $\alpha\times 1$ and let $p:\widetilde{T^{*}J}\times X\to\widetilde{T^{*}J}$ the projection onto the first factor.

\begin{lemma}\label{geometric descriptions of second cones}
(1) $\cC_{2}|_{\Sigma\setminus E_{1}}/\!/\SL(2)$ is a free $\ZZ_{2}$-quotient of $\Psi^{-1}(0)/\!/\CC^{*}$-bundle over $\widetilde{T^{*}J}\setminus\alpha^{-1}(\ZZ_{2}^{2g})$.

(2) $\tcC_{2}|_{\Sigma\setminus E_{1}}/\!/\SL(2)$ is a free $\ZZ_{2}$-quotient of $Bl_{0}\Psi^{-1}(0)/\!/\CC^{*}$-bundle over $\widetilde{T^{*}J}\setminus\alpha^{-1}(\ZZ_{2}^{2g})$, where $Bl_{0}\Psi^{-1}(0)$ is the blowing-up of $\Psi^{-1}(0)$ at the vertex.

(3) $\cC_{2}|_{\Sigma\cap E_{1}}/\!/\SL(2)$ is a free $\ZZ_{2}$-quotient of $\Hom^{\om_{\varphi}}(\ker\varphi,\im\varphi^{\perp}/\im\varphi)/\!/\CC^{*}$-bundle over $\alpha^{-1}(\ZZ_{2}^{2g})$, where $[\varphi]\in\Sigma\cap E_{1}$.

(4) $\tcC_{2}|_{\Sigma\cap E_{1}}/\!/\SL(2)$ is a free $\ZZ_{2}$-quotient of $Bl_{0}\Hom^{\om_{\varphi}}(\ker\varphi,\im\varphi^{\perp}/\im\varphi)/\!/\CC^{*}$-bundle over $\alpha^{-1}(\ZZ_{2}^{2g})$, where $[\varphi]\in\Sigma\cap E_{1}$ and $Bl_{0}\Hom^{\om_{\varphi}}(\ker\varphi,\im\varphi^{\perp}/\im\varphi)$ is the blowing-up of $\Hom^{\om_{\varphi}}(\ker\varphi,\im\varphi^{\perp}/\im\varphi)$ at the vertex.

(5) $\cC/\!/\SL(2)$ is a free $\ZZ_{2}$-quotient of $\Hom^{\om_{\varphi}}(\ker\varphi,\im\varphi^{\perp}/\im\varphi)/\!/\CC^{*}$-bundle over $\PP^{2g-1}$, where $[\varphi]\in\PP\Hom_{1}(sl(2),\HH^{g})^{ss}$.

(6) $\tcC/\!/\SL(2)$ is a free $\ZZ_{2}$-quotient of $Bl_{0}\Hom^{\om_{\varphi}}(\ker\varphi,\im\varphi^{\perp}/\im\varphi)/\!/\CC^{*}$-bundle over $\PP^{2g-1}$, where $[\varphi]\in\PP\Hom_{1}(sl(2),\HH^{g})^{ss}$.
\end{lemma}
\begin{proof}
Let $x$ be a point of $X$.
\begin{enumerate}
\item Consider the principal $\PGL(2)$-bundle
$$q:\PP\Isom(\cO_{\widetilde{T^{*}J}}^{2},\cL|_{x}\oplus\cL^{-1}|_{x})\to\widetilde{T^{*}J}.$$
$\PGL(2)$ acts on $\cO_{\widetilde{T^{*}J}}^{2}$ and $\O(2)$ acts on $\cL|_{x}\oplus\cL^{-1}|_{x}$. By the same argument as in the proof of \cite[Proposition 1.7.10]{O99},
$$\Sigma\cong\PP\Isom(\cO_{\widetilde{T^{*}J}}^{2},\cL|_{x}\oplus\cL^{-1}|_{x})/\!/\O(2).$$
$\cC_{2}|_{\Sigma\setminus E_{1}}/\!/\SL(2)$ is the quotient of $q^{*}\Psi_{\cL_{III}}^{-1}(0)/\!/\O(2)$ by the $\PGL(2)$-action, where $\cL_{III}=\cL|_{\widetilde{T^{*}J}\setminus\alpha^{-1}(\ZZ_{2}^{2g})}$ and
$$\Psi_{\cL_{III}}:[p_{*}(\cL_{III}^{-2}K_X)\oplus R^{1}p_{*}(\cL_{III}^2)]\oplus[p_{*}(\cL_{III}^{2}K_X)\oplus R^{1}p_{*}(\cL_{III}^{-2})]\to R^{1}p_{*}(K_X)$$
is the sum of perfect pairings $p_{*}(\cL_{III}^{-2}K_X)\oplus R^{1}p_{*}(\cL_{III}^2)\to R^{1}p_{*}(K_X)$ and $p_{*}(\cL_{III}^{2}K_X)\oplus R^{1}p_{*}(\cL_{III}^{-2})\to R^{1}p_{*}(K_X)$. Since the actions of $\PGL(2)$ and $\O(2)$ commute and $q$ is the principal $\PGL(2)$-bundle, $$\cC_{2}|_{\Sigma\setminus E_{1}}/\!/\SL(2)=\Psi_{\cL_{III}}^{-1}(0)/\!/\O(2)=\frac{\Psi_{\cL_{III}}^{-1}(0)/\!/\SO(2)}{\O(2)/\SO(2)}=\frac{\Psi_{\cL_{III}}^{-1}(0)/\!/\CC^{*}}{\ZZ_{2}}.$$
Hence we get the description.

\item Since $\tcC_{2}|_{\Sigma\setminus E_{1}}/\!/\SL(2)$ is isomorphic to the blowing-up of $\cC_{2}|_{\Sigma\setminus E_{1}}/\!/\SL(2)$ along $T^{*}J/\ZZ_{2}\setminus\ZZ_{2}^{2g}$, it is isomorphic to $\displaystyle\frac{\widetilde{\Psi_{\cL_{III}}^{-1}(0)}/\!/\CC^{*}}{\ZZ_{2}}$, where $\widetilde{\Psi_{\cL_{III}}^{-1}(0)}$ is the blowing-up of $\Psi_{\cL_{III}}^{-1}(0)$ along $\widetilde{T^{*}J}\setminus\alpha^{-1}(\ZZ_{2}^{2g})\cong T^{*}J\setminus\ZZ_{2}^{2g}$.

\item Note that $E_{1}$ is a $2^{2g}$ disjoint union of $\PP\Upsilon^{-1}(0)$. It follows from Proposition \ref{intersection-of-1st-exc-and-2nd-center} that $\Sigma\cap E_{1}$ is a $2^{2g}$ disjoint union of $\PP\Hom_{1}(sl(2),\HH^{g})^{ss}$. By Proposition \ref{ss-local-first-blowup}, we have
$$\PP\Hom_{1}(sl(2),\HH^{g})^{ss}=\PGL(2)Z^{ss}\cong\PGL(2)\times_{\O(2)}Z^{ss}$$
and
$$\PP\Hom_{1}(sl(2),\HH^{g})^{ss}/\!/\PGL(2)\cong Z/\!/\O(2)=Z_{1}=\PP^{2g-1},$$
where $Z=Z_{1}\cup Z_{2}\cup Z_{3}$, $Z^{ss}$ is the set of
semistable points of $Z$ for the action of $\O(2)$,
$Z_{1}=\PP\{v_{1}\otimes\HH^{g}\}=Z^{ss}$,
$Z_{2}=\PP\{v_{2}\otimes\HH^{g}\}$ and
$Z_{3}=\PP\{v_{3}\otimes\HH^{g}\}$ for the basis $\{v_{1},v_{2},v_{3}\}$ of $sl(2)$ chosen in Section \ref{local picture}. Then we have
$$\cC_{2}|_{\Sigma\cap\PP\Upsilon^{-1}(0)}=\PGL(2)\times_{\O(2)}\cC_{2}|_{Z^{ss}}$$
and
$$\cC_{2}|_{\Sigma\cap\PP\Upsilon^{-1}(0)}/\!/\SL(2)=\cC_{2}|_{\Sigma\cap\PP\Upsilon^{-1}(0)}/\!/\PGL(2)=\cC_{2}|_{Z^{ss}}/\!/\O(2)=\frac{\cC_{2}|_{Z^{ss}}/\!/\SO(2)}{\ZZ_{2}}.$$
Since $\cC_{2}|_{Z^{ss}}$ is a $\Hom^{\om_{\varphi}}(\ker\varphi,\im\varphi^{\perp}/\im\varphi)$-bundle over $Z^{ss}$ by Proposition \ref{normalcone-localmodel},
$$\cC_{2}|_{Z^{ss}}/\!/\SO(2)$$
is a $\Hom^{\om_{\varphi}}(\ker\varphi,\im\varphi^{\perp}/\im\varphi)/\!/\CC^{*}$-bundle over $Z/\!/\SO(2)=Z_{1}=\PP^{2g-1}$. Since $\alpha^{-1}(\ZZ_{2}^{2g})$ is a $2^{2g}$ disjoint union of $\PP^{2g-1}$, we get the description.

\item Since $\tcC_{2}|_{\Sigma\cap E_{1}}/\!/\SL(2)$ is isomorphic to the blowing-up of $\cC_{2}|_{\Sigma\cap E_{1}}/\!/\SL(2)$ along $2^{2g}$ disjoint union of $\PP\Hom_{1}(sl(2),\HH^{g})^{ss}/\!/\PGL(2)\cong\PP^{2g-1}$, it is isomorphic to $2^{2g}$ disjoint union of $\displaystyle\frac{\widetilde{\cC_{2}|_{Z^{ss}}}/\!/\SO(2)}{\ZZ_{2}}$, where $\widetilde{\cC_{2}|_{Z^{ss}}}$ is the blowing-up of $\cC_{2}|_{Z^{ss}}$ along $Z^{ss}$.

\item Since $\cC=\cC_{2}|_{\Sigma\cap\PP\Upsilon^{-1}(0)^{ss}}$, we get the description from (3).

\item Since $\tcC=\tcC_{2}|_{E_{2}\cap Bl_{\PP\Hom_{1}}\PP\Upsilon^{-1}(0)^{ss}}$, we get the description from (4).
\end{enumerate}
\end{proof}

We next explain how to compute the terms
$$\dim IH^{i}(Bl_{0}\Upsilon^{-1}(0)/\!/\PGL(2))-\dim IH^{i}(\Upsilon^{-1}(0)/\!/\PGL(2)),$$
$$\dim IH^{i}(\tcC_{2}/\!/\SL(2))-\dim IH^{i}(\cC_2/\!/\SL(2))$$
and
$$\dim IH^{i}(\tcC/\!/\SL(2))-\dim IH^{i}(\cC/\!/\SL(2))$$
that appear in Lemma \ref{intersection blowing-up formula}. We start with the following technical lemma.

\begin{lemma}[Lemma 2.12 in \cite{K86}]\label{int coh of quotient by finite group}
Let $V$ be a complex variety on which a finite group $F$ acts. Then
$$IH^{*}(V/F)\cong IH^{*}(V)^{F}$$ where $IH^{*}(V)^{F}$ denotes the
invariant part of $IH^{*}(V)$ under the action of $F$.
\end{lemma}

Now recall that $\Upsilon^{-1}(0)/\!/\SL(2)=\Upsilon^{-1}(0)/\!/\PGL(2)$ and $\PP\Upsilon^{-1}(0)^{ss}/\!/\SL(2)=\PP\Upsilon^{-1}(0)^{ss}/\!/\PGL(2)$ from section \ref{local picture}.

We can compute $IH^{*}(\Upsilon^{-1}(0)/\!/\SL(2))$ (respectively, $IH^{*}(\Psi^{-1}(0)/\!/\CC^*)$ and $$IH^{*}(\Hom^{\om_{\varphi}}(\ker\varphi,\im\varphi^{\perp}/\im\varphi)/\!/\CC^*)\text{)}$$
in terms of $IH^{*}(\PP\Upsilon^{-1}(0)^{ss}/\!/\SL(2))$ (respectively, $IH^{*}(\PP\Psi^{-1}(0)^{ss}/\!/\CC^*)$ and $$IH^{*}(\PP\Hom^{\om_{\varphi}}(\ker\varphi,\im\varphi^{\perp}/\im\varphi)^{ss}/\!/\CC^*)\text{)}.$$
In order to explain this, we need the following lemmas. The first lemma shows the surjectivities of the Kirwan maps on the fibers of normal cones and exceptional divisors.

\begin{lemma}\label{Kir-Map}

(1) The Kirwan map
$$IH_{\SL(2)}^{*}(\PP\Upsilon^{-1}(0)^{ss})\rightarrow IH^{*}(\PP\Upsilon^{-1}(0)^{ss}/\!/\SL(2))$$
is surjective.

(2) The Kirwan map
$$IH_{\SL(2)}^{*}(\Upsilon^{-1}(0))\rightarrow IH^{*}(\Upsilon^{-1}(0)/\!/\SL(2))$$
is surjective.

(3) The Kirwan map
$$H_{\CC^*}^{*}(\PP\Psi^{-1}(0)^{ss})\rightarrow IH^{*}(\PP\Psi^{-1}(0)^{ss}/\!/\CC^*)$$
is surjective.

(4) The Kirwan map
$$IH_{\CC^*}^{*}(\Psi^{-1}(0))\rightarrow IH^{*}(\Psi^{-1}(0)/\!/\CC^*)$$
is surjective.

(5) The Kirwan map
$$H_{\CC^*}^{*}(\PP\Hom^{\om_{\varphi}}(\ker\varphi,\im\varphi^{\perp}/\im\varphi)^{ss})\rightarrow IH^{*}(\PP\Hom^{\om_{\varphi}}(\ker\varphi,\im\varphi^{\perp}/\im\varphi)^{ss}/\!/\CC^*)$$
is surjective.

(6) The Kirwan map
$$IH_{\CC^*}^{*}(\Hom^{\om_{\varphi}}(\ker\varphi,\im\varphi^{\perp}/\im\varphi))\rightarrow IH^{*}(\Hom^{\om_{\varphi}}(\ker\varphi,\im\varphi^{\perp}/\im\varphi)/\!/\CC^*)$$
is surjective.
\end{lemma}
\begin{proof}
\begin{enumerate}
\item Consider the quotient map
$$f:\PP\Upsilon^{-1}(0)^{ss}\to\PP\Upsilon^{-1}(0)^{ss}/\!/\SL(2).$$
In \cite[section 6]{BL94}, Bernstein and Lunts define a functor
$$Qf_{*}:\bD_{\SL(2)}(\PP\Upsilon^{-1}(0)^{ss})\to\bD(\PP\Upsilon^{-1}(0)^{ss}/\!/\SL(2))$$
that extends the pushforward of sheaves $f_{*}$.

By the same arguments as those of \cite[$\S2$ and
$\S3$]{Woo03}, we can obtain morphisms
$$\lambda_{\PP\Upsilon^{-1}(0)}:\ic^{\bullet}(\PP\Upsilon^{-1}(0)^{ss}/\!/\SL(2))[3]\to Qf_{*}\ic^{\bullet}_{\SL(2)}(\PP\Upsilon^{-1}(0)^{ss})$$
and
$$\kappa_{\PP\Upsilon^{-1}(0)}:Qf_{*}\ic^{\bullet}_{\SL(2)}(\PP\Upsilon^{-1}(0)^{ss})\to\ic^{\bullet}(\PP\Upsilon^{-1}(0)^{ss}/\!/\SL(2))[3]$$
such that $\kappa_{\PP\Upsilon^{-1}(0)}\circ\lambda_{\PP\Upsilon^{-1}(0)}=\id$. $\lambda_{\PP\Upsilon^{-1}(0)}$ induces a map
$$IH^{*}(\PP\Upsilon^{-1}(0)^{ss}/\!/\SL(2))\to IH_{\SL(2)}^{*}(\PP\Upsilon^{-1}(0)^{ss})$$
which is an inclusion.

Hence $\kappa_{\PP\Upsilon^{-1}(0)}$ induces a map
$$\tilde{\kappa}_{\PP\Upsilon^{-1}(0)}:IH_{\SL(2)}^{*}(\PP\Upsilon^{-1}(0)^{ss})\to IH^{*}(\PP\Upsilon^{-1}(0)^{ss}/\!/\SL(2))$$
which is split by the inclusion
$$IH^{*}(\PP\Upsilon^{-1}(0)^{ss}/\!/\SL(2))\to IH_{\SL(2)}^{*}(\PP\Upsilon^{-1}(0)^{ss}).$$

\item Let $R:=\CC[T_0,T_1,\cdots,T_{6g-1}]$. For an
$\SL(2)$-invariant ideal $I\subset R$ generated by three quadratic
homogeneous polynomials in $R$ defining $\Upsilon^{-1}(0)$, we can
write
$$\Upsilon^{-1}(0)=\spec(R/I).$$
Let $\overline{\Upsilon^{-1}(0)}$ be the Zariski closure of
$\Upsilon^{-1}(0)$ in $\PP^{6g}$. Since the homogenization of
$I$ equals to $I$, we can write
$$\overline{\Upsilon^{-1}(0)}=\proj(R[T]/I\cdot R[T])$$
where $\SL(2)$ acts trivially on the variable $T$. Thus
$$\Upsilon^{-1}(0)/\!/\SL(2)=\spec(R^{\SL(2)}/I\cap R^{\SL(2)})$$
and
$$\overline{\Upsilon^{-1}(0)}^{ss}/\!/\SL(2)=\proj(R[T]/I\cdot R[T])^{\SL(2)}.$$
Since $\SL(2)$ acts trivially on the variable $T$,
$$\overline{\Upsilon^{-1}(0)}^{ss}/\!/\SL(2)=\proj(R^{\SL(2)}[T]/(I\cap R^{\SL(2)})\cdot R^{\SL(2)}[T]).$$
Hence we have an open immersion
$\Upsilon^{-1}(0)/\!/\SL(2)\hookrightarrow\overline{\Upsilon^{-1}(0)}^{ss}/\!/\SL(2)$
given by $\mathfrak{p}\mapsto\mathfrak{p}^{hom}$ where
$\mathfrak{p}^{hom}$ is the homogenization of $\mathfrak{p}$.

Note that
$$\overline{\Upsilon^{-1}(0)}\setminus\Upsilon^{-1}(0)=\proj(R[T]/I\cdot
R[T]+(T))\cong\proj(R/I)=\PP\Upsilon^{-1}(0)$$ and
$$[\overline{\Upsilon^{-1}(0)}^{ss}/\!/\SL(2)]\setminus[\Upsilon^{-1}(0)/\!/\SL(2)]$$
$$=\proj(R^{\SL(2)}[T]/(I\cap R^{\SL(2)})\cdot R^{\SL(2)}[T]+((T)\cap R^{\SL(2)}[T]))$$
$$\cong\proj(R^{\SL(2)}/I\cap R^{\SL(2)})=\PP\Upsilon^{-1}(0)^{ss}/\!/\SL(2)$$
where $(T)$ is the ideal of $R[T]$ generated by $T$.

Consider the quotient map
$$f:\overline{\Upsilon^{-1}(0)}^{ss}\to\overline{\Upsilon^{-1}(0)}^{ss}/\!/\SL(2).$$
In \cite[section 6]{BL94}, Bernstein and Lunts define a functor
$$Qf_{*}:\bD_{\SL(2)}(\overline{\Upsilon^{-1}(0)}^{ss})\to\bD(\overline{\Upsilon^{-1}(0)}^{ss}/\!/\SL(2))$$
that extends the pushforward of sheaves $f_{*}$.

By the same arguments as those of \cite[$\S2$ and
$\S3$]{Woo03}, we can obtain morphisms
$$\lambda_{\overline{\Upsilon^{-1}(0)}}:\ic^{\bullet}(\overline{\Upsilon^{-1}(0)}^{ss}/\!/\SL(2))[3]\to Qf_{*}\ic^{\bullet}_{\SL(2)}(\overline{\Upsilon^{-1}(0)}^{ss})$$
and
$$\kappa_{\overline{\Upsilon^{-1}(0)}}:Qf_{*}\ic^{\bullet}_{\SL(2)}(\overline{\Upsilon^{-1}(0)}^{ss})\to\lambda_{\overline{\Upsilon^{-1}(0)}}:\ic^{\bullet}(\overline{\Upsilon^{-1}(0)}^{ss}/\!/\SL(2))[3]$$
such that $\kappa_{\overline{\Upsilon^{-1}(0)}}\circ\lambda_{\overline{\Upsilon^{-1}(0)}}=\id$. $\lambda_{\overline{\Upsilon^{-1}(0)}}$ induces a map
$$IH^{*}(\overline{\Upsilon^{-1}(0)}^{ss}/\!/\SL(2))\to IH_{\SL(2)}^{*}(\overline{\Upsilon^{-1}(0)}^{ss})$$
which is an inclusion.

Hence $\kappa_{\overline{\Upsilon^{-1}(0)}}$ induces a map
$$\tilde{\kappa}_{\overline{\Upsilon^{-1}(0)}}:IH_{\SL(2)}^{*}(\overline{\Upsilon^{-1}(0)}^{ss})\to IH^{*}(\overline{\Upsilon^{-1}(0)}^{ss}/\!/\SL(2))$$
which is split by the inclusion $IH^{*}(\overline{\Upsilon^{-1}(0)}^{ss}/\!/\SL(2))\to IH_{\SL(2)}^{*}(\overline{\Upsilon^{-1}(0)}^{ss})$.

Consider the following commutative diagram:

\begin{equation}\label{Seq-Cpx}\xymatrix{\vdots\ar[d]&\vdots\ar[d]\\IH_{\SL(2)}^{i-2}(\PP\Upsilon^{-1}(0)^{ss})\ar[r]^{\tilde{\kappa}_{\PP\Upsilon^{-1}(0)}}\ar[d]&IH^{i-2}(\PP\Upsilon^{-1}(0)^{ss}/\!/\SL(2))\ar[d]\\
IH_{\SL(2)}^i(\overline{\Upsilon^{-1}(0)}^{ss})\ar[r]^{\tilde{\kappa}_{\overline{\Upsilon^{-1}(0)}}}\ar[d]&IH^{i}(\overline{\Upsilon^{-1}(0)}^{ss}/\!/\SL(2))\ar[d]\\
IH_{\SL(2)}^i(\Upsilon^{-1}(0))\ar[r]^{\tilde{\kappa}_{\Upsilon^{-1}(0)}}\ar[d]&IH^{i}(\Upsilon^{-1}(0)/\!/\SL(2))\ar[d]\\
\vdots&\vdots}\end{equation}\\
Vertical sequences are Gysin sequences and $\tilde{\kappa}_{\Upsilon^{-1}(0)}$ is induced from $\tilde{\kappa}_{\PP\Upsilon^{-1}(0)}$ and $\tilde{\kappa}_{\overline{\Upsilon^{-1}(0)}}$. Since
$\tilde{\kappa}_{\PP\Upsilon^{-1}(0)}$ and
$\tilde{\kappa}_{\overline{\Upsilon^{-1}(0)}}$ are surjective, $\tilde{\kappa}_{\Upsilon^{-1}(0)}$ is surjective.

\item Following the idea of the proof of item (1), we get the result.

\item Following the idea of the proof of item (2), we get the result.

\item Following the idea of the proof of item (1), we get the result.

\item Following the idea of the proof of item (2), we get the result.
\end{enumerate}
\end{proof}

The second lemma shows how to compute the intersection cohomologies of the fibers of the normal cones of the singularities of $\bM$ via those of the projectivizations of the fibers.

It is well known that there is a very ample line
bundle $\cL$ (respectively, $\cM_{1}$ and $\cM_{2}$) on
\begin{center}$\PP\Upsilon^{-1}(0)^{ss}/\!/\SL(2)$ (respectively, $\PP\Psi^{-1}(0)^{ss}/\!/\CC^*$ and $\PP\Hom^{\om_{\varphi}}(\ker\varphi,\im\varphi^{\perp}/\im\varphi)^{ss}/\!/\CC^*$),\end{center}
whose
pullback to $\PP\Upsilon^{-1}(0)^{ss}$ (respectively, $\PP\Psi^{-1}(0)^{ss}$ and $\PP\Hom^{\om_{\varphi}}(\ker\varphi,\im\varphi^{\perp}/\im\varphi)^{ss}$) is the $M$th (respectively, $N_{1}$th and $N_{2}$th) tensor
power of the hyperplane line bundle on $\PP\Upsilon^{-1}(0)$ (respectively, $\PP\Psi^{-1}(0)$ and $\PP\Hom^{\om_{\varphi}}(\ker\varphi,\im\varphi^{\perp}/\im\varphi)$) for some $M$ (respectively, $N_{1}$ and $N_{2}$).

Let $C_{\cL}(\PP\Upsilon^{-1}(0)^{ss}/\!/\SL(2))$ (respectively, $C_{\cM_{1}}(\PP\Psi^{-1}(0)^{ss}/\!/\CC^*)$ and
$$C_{\cM_{2}}(\PP\Hom^{\om_{\varphi}}(\ker\varphi,\im\varphi^{\perp}/\im\varphi)^{ss}/\!/\CC^*)\text{)}$$
be the affine cone on $\PP\Upsilon^{-1}(0)^{ss}/\!/\SL(2)$ (respectively, $\PP\Psi^{-1}(0)^{ss}/\!/\CC^*$ and
$$\PP\Hom^{\om_{\varphi}}(\ker\varphi,\im\varphi^{\perp}/\im\varphi)^{ss}/\!/\CC^*\text{)}$$
with respect to the projective embedding induced by the sections of $\cL$ (respectively, $\cM_{1}$ and $\cM_{2}$).

\begin{lemma}\label{int coh of affine cone of git quotient}
(1) $IH^{*}(\Upsilon^{-1}(0)/\!/\SL(2))=IH^{*}(C_{\cL}(\PP\Upsilon^{-1}(0)^{ss}/\!/\SL(2)))$ and $$IH^{*}(Bl_{0}\Upsilon^{-1}(0)/\!/\SL(2))=IH^{*}(\PP\Upsilon^{-1}(0)^{ss}/\!/\SL(2)),$$

(2) $IH^{*}(\Psi^{-1}(0)/\!/\CC^{*})=IH^{*}(C_{\cM_{1}}(\PP\Psi^{-1}(0)^{ss}/\!/\CC^{*}))$ and $$IH^{*}(Bl_{0}\Psi^{-1}(0)/\!/\CC^{*})=IH^{*}(\PP\Psi^{-1}(0)^{ss}/\!/\CC^{*}).$$

(3) $IH^{*}(\Hom^{\om_{\varphi}}(\ker\varphi,\im\varphi^{\perp}/\im\varphi)/\!/\CC^{*})=IH^{*}(C_{\cM_{2}}(\PP\Hom^{\om_{\varphi}}(\ker\varphi,\im\varphi^{\perp}/\im\varphi)^{ss}/\!/\CC^{*}))$ and $$IH^{*}(Bl_{0}\Hom^{\om_{\varphi}}(\ker\varphi,\im\varphi^{\perp}/\im\varphi)/\!/\CC^{*})=IH^{*}(\PP\Hom^{\om_{\varphi}}(\ker\varphi,\im\varphi^{\perp}/\im\varphi)^{ss}/\!/\CC^{*}).$$
\end{lemma}
\begin{proof}
\begin{enumerate}
\item We first follow the idea of the proof of \cite[Lemma 2.15]{K86} to see that
$$C_{\cL}(\PP\Upsilon^{-1}(0)^{ss}/\!/\SL(2))\cong\Upsilon^{-1}(0)/\!/(\SL(2)\times
F)\cong(\Upsilon^{-1}(0)/\!/\SL(2))/F,$$
where $F$ is the finite subgroup of $\GL(6g)$ consisting of all diagonal matrices
$diag(\eta,\cdots,\eta)$ such that $\eta$ is an $M$th root of unity.

The coordinate ring of $C_{\mathcal{L}}(\PP\Upsilon^{-1}(0)^{ss}/\!/\SL(2))$ is the subring $(\CC[Y_{0},\cdots,Y_{6g-1}]/I)_{M}^{\SL(2)}$ of the coordinate ring $\CC[Y_{0},\cdots,Y_{6g-1}]/I$ of $\Upsilon^{-1}(0)$ which is generated by homogeneous polynomials fixed by the natural action of $\SL(2)$ and of degree $M$. Since
$$(\CC[Y_{0},\cdots,Y_{6g-1}]/I)_{M}=\CC[Y_{0},\cdots,Y_{6g-1}]_{M}/I\cap\CC[Y_{0},\cdots,Y_{6g-1}]_{M},$$
we have
$$(\CC[Y_{0},\cdots,Y_{6g-1}]/I)_{M}^{\SL(2)}=\CC[Y_{0},\cdots,Y_{6g-1}]_{M}^{\SL(2)}/I\cap\CC[Y_{0},\cdots,Y_{6g-1}]_{M}^{\SL(2)}$$
$$=\CC[Y_{0},\cdots,Y_{6g-1}]^{\SL(2)\times F}/I\cap\CC[Y_{0},\cdots,Y_{6g-1}]^{\SL(2)\times F}=(\CC[Y_{0},\cdots,Y_{6g-1}]/I)^{\SL(2)\times F}.$$
Thus we get
$$C_{\cL}(\PP\Upsilon^{-1}(0)^{ss}/\!/\SL(2))\cong\Upsilon^{-1}(0)/\!/(\SL(2)\times
F)\cong(\Upsilon^{-1}(0)/\!/\SL(2))/F$$
and then
$$IH^{*}(\Upsilon^{-1}(0)/\!/\SL(2))^{F}=IH^{*}(C_{\cL}(\PP\Upsilon^{-1}(0)^{ss}/\!/\SL(2)))$$
by Lemma \ref{int coh of quotient by finite group}.

It remains to show that the action of $F$ on
$IH^{*}(\Upsilon^{-1}(0)/\!/\SL(2))$ is trivial. Since the Kirwan map
$$IH_{\SL(2)}^{*}(\Upsilon^{-1}(0))\rightarrow IH^{*}(\Upsilon^{-1}(0)/\!/\SL(2))$$
is surjective by Lemma \ref{Kir-Map}-(2), it suffices to show that
the action of $F$ on $IH_{\SL(2)}^{*}(\Upsilon^{-1}(0))$ is trivial.
Let
$$\pi_1:Bl_0\Upsilon^{-1}(0)\to \Upsilon^{-1}(0)$$
be the blowing-up of $\Upsilon^{-1}(0)$ at the vertex and let
$$\pi_2:Bl_{\Hom_1}Bl_0\Upsilon^{-1}(0)\to Bl_0\Upsilon^{-1}(0)$$
be the blowing-up of $Bl_0\Upsilon^{-1}(0)$ along
$\widetilde{\Hom_1(sl(2),\HH^g)}$, where $\widetilde{\Hom_1(sl(2),\HH^g)}$ is the strict transform of $\Hom_1(sl(2),\HH^g)$. Since the centers of the blowing-ups are $F$-invariant, the action of $F$ on $\Upsilon^{-1}(0)$ lifts to an action of $F$ on $Bl_{\Hom_1}Bl_0\Upsilon^{-1}(0)$. Since $\pi_1\circ\pi_2$ is proper and $Bl_{\Hom_1}Bl_0\Upsilon^{-1}(0)$ is smooth (See the proof of Lemma 1.8.5 in \cite{O99}), by Proposition \ref{3 consequences}-(3), $IH_{\SL(2)}^{*}(\Upsilon^{-1}(0))$ is a direct summand
of
$$IH_{\SL(2)}^{*}(Bl_{\Hom_1}Bl_0\Upsilon^{-1}(0))=H_{\SL(2)}^{*}(Bl_{\Hom_1}Bl_0\Upsilon^{-1}(0)).$$
Since $Bl_{\Hom_1}Bl_0\Upsilon^{-1}(0)$ is homotopically equivalent
to $Bl_{\PP\Hom_1}\PP\Upsilon^{-1}(0)$, it suffices to
show that the action of $F$ on
$H_{\SL(2)}^{*}(Bl_{\PP\Hom_1}\PP\Upsilon^{-1}(0))$ is
trivial. But this is true because the action of $F$ on
$\PP\Upsilon^{-1}(0)$ is trivial and it lifts to the trivial action of $F$ on $Bl_{\PP\Hom_1}\PP\Upsilon^{-1}(0)$. Hence $F$ acts trivially on $IH^{*}(\Upsilon^{-1}(0)/\!/\SL(2))$.

Similarly, we next see that $Bl_{v}(C_{\cL}(\PP\Upsilon^{-1}(0)^{ss}/\!/\SL(2)))$ is naturally isomorphic to
$$(Bl_{0}\Upsilon^{-1}(0)/\!/\SL(2))/F,$$
where $v$ is the vertex of $C_{\cL}(\PP\Upsilon^{-1}(0)^{ss}/\!/\SL(2))$.

Let $J$ be the ideal of $\CC[Y_{0},\cdots,Y_{6g-1}]/I$ corresponding to the vertex $O$ of $\Upsilon^{-1}(0)$. Then we have $\dss Bl_{0}\Upsilon^{-1}(0)=\bproj\Bigg(\bigoplus_{m\ge 0}J^{m}\Bigg)$. Then
$$(Bl_{0}\Upsilon^{-1}(0)/\!/\SL(2))/F=Bl_{0}\Upsilon^{-1}(0)/\!/(\SL(2)\times F)=\bproj\Bigg(\bigoplus_{m\ge 0}(J^{m})^{\SL(2)\times F}\Bigg).$$
Since $(J^{m})^{\SL(2)\times F}=J^{m}\cap(\CC[Y_{0},\cdots,Y_{6g-1}]/I)^{\SL(2)\times F}=(J\cap(\CC[Y_{0},\cdots,Y_{6g-1}]/I)^{\SL(2)\times F})^{m}$
$$=(J^{\SL(2)\times F})^{m}$$
and $J^{\SL(2)\times F}$ is the ideal corresponding to $v=O/\!/(\SL(2)\times F)$,
we have
$$\bproj\Bigg(\bigoplus_{m\ge 0}(J^{m})^{\SL(2)\times F}\Bigg)=\bproj\Bigg(\bigoplus_{m\ge 0}(J^{\SL(2)\times F})^{m}\Bigg)=Bl_{v}(\Upsilon^{-1}(0)/\!/(\SL(2)\times
F))$$
$$=Bl_{v}(C_{\cL}(\PP\Upsilon^{-1}(0)^{ss}/\!/\SL(2))).$$

Thus
$$IH^{*}(Bl_{0}\Upsilon^{-1}(0)/\!/\SL(2))^{F}=IH^{*}(Bl_{v}(C_{\cL}(\PP\Upsilon^{-1}(0)^{ss}/\!/\SL(2)))).$$
By the same idea of the proof of the first statement, $F$ acts trivially on $IH^{*}(Bl_{0}\Upsilon^{-1}(0)/\!/\SL(2))$ and then
$$IH^{*}(Bl_{0}\Upsilon^{-1}(0)/\!/\SL(2))^{F}=IH^{*}(Bl_{0}\Upsilon^{-1}(0)/\!/\SL(2)).$$
Since $Bl_{v}(C_{\cL}(\PP\Upsilon^{-1}(0)^{ss}/\!/\SL(2)))$ is homeomorphic to the line bundle $\cL^{\vee}$ over $\PP\Upsilon^{-1}(0)^{ss}/\!/\SL(2)$, there is a Leray spectral sequence $E_{r}^{pq}$ converging to
$$IH^{*}(Bl_{v}(C_{\cL}(\PP\Upsilon^{-1}(0)^{ss}/\!/\SL(2))))$$
with
$$E_{2}^{pq}=IH^{p}(\PP\Upsilon^{-1}(0)^{ss}/\!/\SL(2),IH^{q}(\CC))=\begin{cases}IH^{p}(\PP\Upsilon^{-1}(0)^{ss}/\!/\SL(2))&\text{if }q=0\\
0&\text{otherwise}.\end{cases}$$
Hence we get
$$IH^{*}(Bl_{0}\Upsilon^{-1}(0)/\!/\SL(2))=IH^{*}(\PP\Upsilon^{-1}(0)^{ss}/\!/\SL(2)).$$

\item Following the idea of the proof of item (1) and using Lemma \ref{Kir-Map}-(4), we get the result.

\item Following the idea of the proof of item (1) and using Lemma \ref{Kir-Map}-(6), we get the result.
\end{enumerate}
\end{proof}

By the standard argument of \cite[Proposition 4.7.2]{KW06}, we get the third lemma as follows. It gives a way to compute the intersection cohomology of affine cones of projective GIT quotients.

\begin{lemma}\label{int coh of affine cone of git quotient II}

(1) Let $n=\dim_{\CC}C_{\cL}(\PP\Upsilon^{-1}(0)^{ss}/\!/\SL(2))$. Then
$$IH^{i}(C_{\cL}(\PP\Upsilon^{-1}(0)^{ss}/\!/\SL(2)))\cong\begin{cases}0&\text{for }i\ge n\\
IH^{i}(C_{\cL}(\PP\Upsilon^{-1}(0)^{ss}/\!/\SL(2))-\{0\})&\text{for }i<n.\end{cases}$$

(2) Let $n=\dim_{\CC}C_{\cM_{1}}(\PP\Psi^{-1}(0)^{ss}/\!/\CC^{*})$. Then
$$IH^{i}(C_{\cM_{1}}(\PP\Psi^{-1}(0)^{ss}/\!/\CC^{*}))\cong\begin{cases}0&\text{for }i\ge n\\
IH^{i}(C_{\cM_{1}}(\PP\Psi^{-1}(0)^{ss}/\!/\CC^{*})-\{0\})&\text{for }i<n.\end{cases}$$

(3) Let $n=\dim_{\CC}C_{\cM_{2}}(\PP\Hom^{\om_{\varphi}}(\ker\varphi,\im\varphi^{\perp}/\im\varphi)^{ss}/\!/\CC^{*})$. Then
$$IH^{i}(C_{\cM_{2}}(\PP\Hom^{\om_{\varphi}}(\ker\varphi,\im\varphi^{\perp}/\im\varphi)^{ss}/\!/\CC^{*}))$$
$$\cong\begin{cases}0&\text{for }i\ge n\\
IH^{i}(C_{\cM_{2}}(\PP\Hom^{\om_{\varphi}}(\ker\varphi,\im\varphi^{\perp}/\im\varphi)^{ss}/\!/\CC^{*})-\{0\})&\text{for }i<n.\end{cases}$$
\end{lemma}

The following lemma explains how $IH^{*}(\Upsilon^{-1}(0)/\!/\SL(2))$ (respectively, $IH^{*}(\Psi^{-1}(0)/\!/\CC^*)$ and $IH^{*}(\Hom^{\om_{\varphi}}(\ker\varphi,\im\varphi^{\perp}/\im\varphi)/\!/\CC^*)$) can be computed in terms of $IH^{*}(\PP\Upsilon^{-1}(0)^{ss}/\!/\SL(2))$ (respectively, $IH^{*}(\PP\Psi^{-1}(0)^{ss}/\!/\CC^*)$ and $IH^{*}(\PP\Hom^{\om_{\varphi}}(\ker\varphi,\im\varphi^{\perp}/\im\varphi)^{ss}/\!/\CC^*)$) as desired.

\begin{lemma}\label{suffices-to-show-on-projectivized-git}
(1) $\begin{cases}
IH^{i}(\Upsilon^{-1}(0)/\!/\SL(2))=0&\text{for }i\geq\dim \Upsilon^{-1}(0)/\!/\SL(2)\\
IH^{i}(\Upsilon^{-1}(0)/\!/\SL(2))\cong\coker\lambda&\text{for }i<\dim \Upsilon^{-1}(0)/\!/\SL(2),
\end{cases}$

where $\lambda:IH^{i-2}(\PP\Upsilon^{-1}(0)^{ss}/\!/\SL(2))\rightarrow IH^{i}(\PP\Upsilon^{-1}(0)^{ss}/\!/\SL(2))$ is an injection.

(2) $\begin{cases}
IH^{i}(\Psi^{-1}(0)/\!/\CC^{*})=0&\text{for }i\geq\dim \Psi^{-1}(0)/\!/\CC^{*}\\
IH^{i}(\Psi^{-1}(0)/\!/\CC^{*})\cong\coker\lambda&\text{for }i<\dim \Psi^{-1}(0)/\!/\CC^{*},
\end{cases}$

where $\lambda:IH^{i-2}(\PP\Psi^{-1}(0)^{ss}/\!/\CC^{*})\rightarrow IH^{i}(\PP\Psi^{-1}(0)^{ss}/\!/\CC^{*})$ is an injection.

(3) $\begin{cases}
IH^{i}(\Hom^{\om_{\varphi}}(\ker\varphi,\im\varphi^{\perp}/\im\varphi)/\!/\CC^{*})=0&\text{for }i\geq\dim \Hom^{\om_{\varphi}}(\ker\varphi,\im\varphi^{\perp}/\im\varphi)/\!/\CC^{*}\\
IH^{i}(\Hom^{\om_{\varphi}}(\ker\varphi,\im\varphi^{\perp}/\im\varphi)/\!/\CC^{*})\cong\coker\lambda&\text{for }i<\dim \Hom^{\om_{\varphi}}(\ker\varphi,\im\varphi^{\perp}/\im\varphi)/\!/\CC^{*},
\end{cases}$

where
$$\lambda:IH^{i-2}(\PP\Hom^{\om_{\varphi}}(\ker\varphi,\im\varphi^{\perp}/\im\varphi)^{ss}/\!/\CC^{*})\rightarrow IH^{i}(\PP\Hom^{\om_{\varphi}}(\ker\varphi,\im\varphi^{\perp}/\im\varphi)^{ss}/\!/\CC^{*})$$
is an injection.
\end{lemma}
\begin{proof}
We follow the idea of the proof of \cite[Corollary 2.17]{K86}. We only prove item (1) because the proofs of item (2) and item (3) are similar to that of item (1).

By Lemma \ref{int coh of affine cone of git quotient}-(1),
$$IH^{*}(\Upsilon^{-1}(0)/\!/\SL(2))=IH^{*}(C_{\cL}(\PP\Upsilon^{-1}(0)^{ss}/\!/\SL(2))).$$

Let $n=\dim_{\CC}\Upsilon^{-1}(0)/\!/\SL(2)$. By Lemma \ref{int coh of affine cone of git quotient II}-(1), $$IH^{i}(C_{\cL}(\PP\Upsilon^{-1}(0)^{ss}/\!/\SL(2)))\cong\begin{cases}0&\text{if }i\ge n\\
IH^{i}(C_{\cL}(\PP\Upsilon^{-1}(0)^{ss}/\!/\SL(2))-\{0\})&\text{if }i<n.\end{cases}$$
Since $C_{\cL}(\PP\Upsilon^{-1}(0)^{ss}/\!/\SL(2))-\{0\}$ fibers over
$\PP\Upsilon^{-1}(0)^{ss}/\!/\SL(2)$ with fiber $\CC^{*}$, there is a Leray spectral sequence
$E_{r}^{pq}$ converging to
$$IH^{*}(C_{\cL}(\PP\Upsilon^{-1}(0)^{ss}/\!/\SL(2))-\{0\})$$
with
$$E_{2}^{pq}=IH^{p}(\PP\Upsilon^{-1}(0)^{ss}/\!/\SL(2),IH^{q}(\CC^{*}))=\begin{cases}IH^{p}(\PP\Upsilon^{-1}(0)^{ss}/\!/\SL(2))&\text{if }q=0,1\\
0&\text{otherwise}.\end{cases}$$
It follows from \cite[5.1]{CGM82} and \cite[p.462--p.468]{GH78} that the differential
$$\lambda:IH^{i-2}(\PP\Upsilon^{-1}(0)^{ss}/\!/\SL(2))\rightarrow IH^{i}(\PP\Upsilon^{-1}(0)^{ss}/\!/\SL(2))$$
is given by the multiplication by $c_{1}(\cL)$. By the Hard
Lefschetz theorem for intersection cohomology, $\lambda$ is injective for $i<n$. Hence we get the result.
\end{proof}

The quotients $\PP\Psi^{-1}(0)^{ss}/\!/\CC^{*}$ and $\PP\Hom^{\om_{\varphi}}(\ker\varphi,\im\varphi^{\perp}/\im\varphi)^{ss}/\!/\CC^{*}$ can be identified with some incidence variety.

\begin{lemma}\label{isomorphic to incidence variety}
Let $I_{2g-3}$ be the incidence variety given by
$$I_{2g-3}=\{(p,H)\in\PP^{2g-3}\times\breve{\PP}^{2g-3}|p\in H\}.$$
\begin{enumerate}
\item\label{description of projectivized middle singularity} $\PP\Psi^{-1}(0)^{ss}/\!/\CC^{*}\cong I_{2g-3}$,

\item $\PP\Hom^{\om_{\varphi}}(\ker\varphi,\im\varphi^{\perp}/\im\varphi)^{ss}/\!/\CC^{*}\cong I_{2g-3}$.
\end{enumerate}
\end{lemma}
\begin{proof}
\begin{enumerate}
\item Consider the map $f:\PP\Psi^{-1}(0)\to I_{2g-3}$ given by
$$(a,b,c,d)\mapsto((b,c),(-a,d)).$$
Since $f$ is $\CC^{*}$-invariant, we have the induced map
$$\bar{f}:\PP\Psi^{-1}(0)^{ss}/\!/\CC^{*}\to I_{2g-3}.$$
We claim that $\bar{f}$ is injective. Assume that $\bar{f}([a_{1},b_{1},c_{1},d_{1}])=\bar{f}([a_{2},b_{2},c_{2},d_{2}])$ where $[a,b,c,d]$ denotes the closed orbit of $(a,b,c,d)$. Then there are nonzero complex numbers $\lambda$ and $\mu$ such that $(b_{1},c_{1})=\lambda(b_{2},c_{2})$ and $(-a_{1},d_{1})=\mu(-a_{2},d_{2})$. Then
$$[a_{1},b_{1},c_{1},d_{1}]=[\mu a_{2},\lambda b_{2},\lambda c_{2},\mu d_{2}]=[(\lambda\mu)^{1/2}a_{2},(\lambda\mu)^{1/2}b_{2},(\lambda\mu)^{1/2}c_{2},(\lambda\mu)^{1/2}d_{2}]$$
$$=[a_{2},b_{2},c_{2},d_{2}].$$
Thus $\bar{f}$ is injective.

Since the domain and the range of $\bar{f}$ are normal varieties with the same dimension and the range $I_{2g-3}$ is irreducible, $\bar{f}$ is an isomorphism.

\item Consider the map $g:\PP\Hom^{\om_{\varphi}}(\ker\varphi,\im\varphi^{\perp}/\im\varphi)\to I_{2g-3}$ given by $$(Z_{12},\cdots,Z_{2g,2},Z_{13},\cdots,Z_{2g,3})$$
$$\mapsto((Z_{12},Z_{22},\cdots,Z_{2g-1,2},Z_{2g,2}),(Z_{23},-Z_{13}\cdots,Z_{2g,3},-Z_{2g-1,3})).$$
Since $g$ is $\CC^{*}$-invariant, we have the induced map
$$\bar{g}:\PP\Hom^{\om_{\varphi}}(\ker\varphi,\im\varphi^{\perp}/\im\varphi)^{ss}/\!/\CC^{*}\to I_{2g-3}.$$
We can see that $\bar{g}$ is injective by the similar way as in the proof of (\ref{description of projectivized middle singularity}). Since the domain and the range of $\bar{g}$ are normal varieties with the same dimension and the range $I_{2g-3}$ is irreducible, $\bar{g}$ is an isomorphism.
\end{enumerate}
\end{proof}

By the proof of Lemma \ref{geometric descriptions of second cones},
\begin{center}$\cC_{2}/\!/\SL(2)=(Y/\!/\CC^{*})/\ZZ_{2}$ and $\tcC_{2}/\!/\SL(2)=(Bl_{\widetilde{T^{*}J}}Y/\!/\CC^{*})/\ZZ_{2},$\end{center}
where $Y$ is either a $\Psi^{-1}(0)$-bundle or a $\Hom^{\om_{\varphi}}(\ker\varphi,\im\varphi^{\perp}/\im\varphi)$-bundle over $\widetilde{T^{*}J}$.

To give computable fomulas from Lemma \ref{intersection blowing-up formula}, we need the following technical statements for $Y/\!/\CC^{*}$ and $Bl_{\widetilde{T^{*}J}}Y/\!/\CC^{*}$.

\begin{lemma}\label{second cone gives a constant sheaf}
Let $g:Y/\!/\CC^{*}\to\widetilde{T^{*}J}$ be the map induced by the projection $Y\to\widetilde{T^{*}J}$ and let $h:Bl_{\widetilde{T^{*}J}}Y/\!/\CC^{*}\to\widetilde{T^{*}J}$ be the map induced by the composition of maps $Bl_{\widetilde{T^{*}J}}Y\to Y\to\widetilde{T^{*}J}$. Then $R^{i}g_{*}\ic^{\bullet}(Y/\!/\CC^{*})$ and $R^{i}h_{*}\ic^{\bullet}(Bl_{\widetilde{T^{*}J}}Y/\!/\CC^{*})$ are constant sheaves for each $i\ge0$.
\end{lemma}
\begin{proof}
Following the idea of proof of \cite[Proposition 2.13]{K86} in the case that $G=\SL(2)$, $K=\SU(2)$, $R=S=\CC^{*}$, $N=N_{0}=G$ and $Z_{R}^{ss}/\!/N_{0}=\widetilde{T^{*}J}$ and using Lemma \ref{int coh of affine cone of git quotient}-(2), (3) and Lemma \ref{suffices-to-show-on-projectivized-git}-(2), (3), we can see that $R^{i}g_{*}\ic^{\bullet}(Y/\!/\CC^{*})$ and $R^{i}h_{*}\ic^{\bullet}(Bl_{\widetilde{T^{*}J}}Y/\!/\CC^{*})$ are locally constant sheaves for each $i\ge0$. Since $\widetilde{T^{*}J}$ is irreducible, we get the conclusion.
\end{proof}

Then we have the following computable blowing-up formula.

\begin{theorem}\label{computable intersection blowing-up formula}
(1) $\dim IH^{i}(\bR_{1}^{ss}/\!/\SL(2))=\dim IH^{i}(\bR/\!/\SL(2))$
$$+2^{2g}\dim IH^{i}(\PP\Upsilon^{-1}(0)^{ss}/\!/\PGL(2))-2^{2g}\dim IH^{i}(\Upsilon^{-1}(0)/\!/\PGL(2))$$
for all $i\ge0$.

(2) $\dim IH^{i}(\bR_{2}^{s}/\SL(2))=\dim IH^{i}(\bR_{1}^{ss}/\!/\SL(2))$
$$+\sum_{p+q=i}\dim[H^{p}(\widetilde{T^{*}J})\otimes H^{t(q)}(I_{2g-3})]^{\ZZ_{2}}$$
for all $i\ge0$, where $t(q)=q-2$ for $q\le\dim I_{2g-3}=4g-7$ and $t(q)=q$ otherwise.

(3) $\dim IH^{i}(Bl_{\PP\Hom_{1}}\PP\Upsilon^{-1}(0)^{ss}/\!/\SL(2))=\dim IH^{i}(\PP\Upsilon^{-1}(0)^{ss}/\!/\SL(2))$
$$+\sum_{p+q=i}\dim[H^{p}(\PP^{2g-1})\otimes H^{t(q)}(I_{2g-3})]^{\ZZ_{2}}$$
for all $i\ge0$, where $t(q)=q-2$ for $q\le\dim I_{2g-3}=4g-7$ and $t(q)=q$ otherwise.
\end{theorem}
\begin{proof}
\begin{enumerate}
\item Since it follows from Lemma \ref{int coh of affine cone of git quotient}-(1) that $$IH^{i}(Bl_{0}\Upsilon^{-1}(0)/\!/\PGL(2))=IH^{i}(\PP\Upsilon^{-1}(0)^{ss}/\!/\PGL(2)),$$
    we get the formula.

\item Let $g:Y/\!/\CC^{*}\to\widetilde{T^{*}J}$ and $h:Bl_{\widetilde{T^{*}J}}Y/\!/\CC^{*}\to\widetilde{T^{*}J}$ be the maps induced by the projections $Y\to\widetilde{T^{*}J}$ and $Bl_{\widetilde{T^{*}J}}Y\to\widetilde{T^{*}J}$. By Proposition \ref{3 consequences}-(2), the Leray spectral sequences of intersection cohomology associated to $g$ and $h$ have $E_2$ terms given by
    $$E_{2}^{pq}=H^{p}(\widetilde{T^{*}J})\otimes IH^{q}(\widehat{I_{2g-3}})$$
    and
    $$E_{2}^{pq}=H^{p}(\widetilde{T^{*}J})\otimes H^{q}(I_{2g-3})$$
    by Lemma \ref{int coh of affine cone of git quotient}, Lemma \ref{isomorphic to incidence variety} and Lemma \ref{second cone gives a constant sheaf}, where $\widehat{I_{2g-3}}$ is the affine cone of $I_{2g-3}$. Since the spectral sequence associated to $h$ is identical to that associated to the projection $E_{h}/\!/\CC^{*}\to\widetilde{T^{*}J}$, where $E_{h}$ is the exceptional divisor of the blowing-up map $Bl_{\widetilde{T^{*}J}}Y\to Y$, it follows from Proposition \ref{3 consequences}-(2) that the decomposition theorem for the projection $E_{h}/\!/\CC^{*}\to\widetilde{T^{*}J}$ implies that the Leray spectral sequence of intersection cohomology associated to $h$ degenerates at the $E_2$ term. Since $IH^{q}(\widehat{I_{2g-3}})$ embeds in $IH^{q}(I_{2g-3})$ by Lemma \ref{suffices-to-show-on-projectivized-git}-(2), (3) and Lemma \ref{isomorphic to incidence variety}, the Leray spectral sequence of intersection cohomology associated to $g$ also degenerates at the $E_2$ term. Since $\cC_{2}/\!/\SL(2)=(Y/\!/\CC^{*})/\ZZ_{2}$ and $\tcC_{2}/\!/\SL(2)=(Bl_{\widetilde{T^{*}J}}Y/\!/\CC^{*})/\ZZ_{2}$, we have
    $$IH^{i}(\cC_{2}/\!/\SL(2))=\bigoplus_{p+q=i}[H^{p}(\widetilde{T^{*}J})\otimes IH^{q}(\widehat{I_{2g-3}})]^{\ZZ_{2}}$$
    and
    $$H^{i}(\tcC_{2}/\!/\SL(2))=\bigoplus_{p+q=i}[H^{p}(\widetilde{T^{*}J})\otimes H^{q}(I_{2g-3})]^{\ZZ_{2}}$$
    by Lemma \ref{int coh of quotient by finite group}. Applying Lemma \ref{suffices-to-show-on-projectivized-git}-(2), (3) again, we get the formula.

\item Note that $\cC/\!/\SL(2)=(Y|_{\PP^{2g-1}}/\!/\CC^{*})/\ZZ_{2}$ and $\tcC_{2}/\!/\SL(2)=(Bl_{\PP^{2g-1}}Y|_{\PP^{2g-1}}/\!/\CC^{*})/\ZZ_{2}$. Let $g':Y|_{\PP^{2g-1}}/\!/\CC^{*}\to\PP^{2g-1}$ and $h':Bl_{\PP^{2g-1}}Y|_{\PP^{2g-1}}/\!/\CC^{*}\to\PP^{2g-1}$ be the maps induced by the projections $Y|_{\PP^{2g-1}}\to\PP^{2g-1}$ and $Bl_{\PP^{2g-1}}Y|_{\PP^{2g-1}}\to\PP^{2g-1}$. Since $\PP^{2g-1}$ is simply connected, $R^{i}g'_{*}\ic^{\bullet}(Y|_{\PP^{2g-1}}/\!/\CC^{*})$ and $R^{i}h'_{*}\ic^{\bullet}(Bl_{\PP^{2g-1}}Y|_{\PP^{2g-1}}/\!/\CC^{*})$ are constant sheaves for each $i\ge0$ by the same argument as in the proof of Lemma \ref{second cone gives a constant sheaf} and then the Leray spectral sequences of intersection cohomology associated to $g'$ and $h'$ have $E_2$ terms given by
    $$E_{2}^{pq}=H^{p}(\PP^{2g-1})\otimes IH^{q}(\widehat{I_{2g-3}})$$
    and
    $$E_{2}^{pq}=H^{p}(\PP^{2g-1})\otimes H^{q}(I_{2g-3})$$
by Lemma \ref{int coh of affine cone of git quotient} and Lemma \ref{isomorphic to incidence variety}. By the same argument as in the remaining part of the proof of item (2), we get the formula.
\end{enumerate}
\end{proof}

\section{A strategy to get a formula for the Poincar\'{e} polynomial of $IH^{*}(\bM)$}\label{strategy}

Since $\bR_{2}^{s}/\SL(2)$ has an orbifold singularity, we have $H^{i}(\bR_{2}^{s}/\SL(2))\cong H_{\SL(2)}^{i}(\bR_{2}^{s})$ for each $i\ge0$. If we have a blowing-up formula for the equivariant cohomology that can be applied to get $\dim H_{\SL(2)}^{i}(\bR_{2}^{s})$ from $\dim H_{\SL(2)}^{i}(\bR)$ for each $i\ge0$, Theorem \ref{computable intersection blowing-up formula} can be used to calculate $\dim IH^{i}(\bM)$ from $\dim H^{i}(\bR_{2}^{s}/\SL(2))$ for each $i$.

\subsection{Towards a blowing-up formula for the equivariant cohomology}

In this subsection, we give a strategy to get a blowing-up formula for the equivariant cohomology in Kirwan's algorithm starting with $\bR$ and prove that a blowing-up formula for the equivariant cohomology in the blowing-up $\pi:Bl_{\PP\Hom_{1}}\PP\Upsilon^{-1}(0)^{ss}\to\PP\Upsilon^{-1}(0)^{ss}$ holds.

%%%%%%%%%%%%%%%%%%%%%%
There are the $\CC^{*}$-actions on $\bM$ by $\lambda\cdot(E,\phi)=(E,\lambda\phi)$ and on $\bR$ by $\lambda\cdot(E,\phi,\beta)=(E,\lambda\phi,\beta)$. Then the $\SL(2)$-action on $\bR$ commutes with the $\CC^{*}$-action on $\bR$. Since $\ZZ_{2}^{2g}$ is invariant under the $\CC^{*}$-action, the $\CC^{*}$-action on $\bR$ lifts to $\bR_{1}$ and the $\CC^{*}$-action on $\bM$ lifts to $\bR_{1}^{ss}/\!/\SL(2)$. Since $\Sig$ is invariant under the $\CC^{*}$-action on $\bR_{1}^{ss}$, the $\CC^{*}$-action on $\bR_{1}^{ss}$ lifts to $\bR_{2}$. The fixed loci of the $\CC^{*}$-action on $\bM$ and $\bR_{1}^{ss}/\!/\SL(2)$ can be described as in the following two Propositions.

\begin{proposition}[(4.1) of \cite{HT03}]\label{fixed locus of bM}
The fixed locus of the $\CC^{*}$-action on $\bM$ is
$$\bM^{\CC^{*}}=\bN\sqcup\bigsqcup_{d=1}^{g-1}F_{d}$$
where
\begin{enumerate}
\item[1.] $\bN$ is the moduli space of semistable rank $2$ vector bundles with trivial determinant on $X$, which parametrizes $\SL(2)$-Higgs bundles of the form $(E,0)$;
\item[2.] $F_{d}=S^{2g-2-2d}X\times_{\Pic^{2d}(X)}\Pic^{d}(X)$, which parametrizes $\SL(2)$-Higgs bundles $(E,\phi)$ of the form $E=L\oplus L^{-1}$, $\phi=\left(\begin{matrix}0&0\\s&0\end{matrix}\right)$, where $S^{m}X$ is the $m$-fold symmetric product of $X$, $F_{d}=S^{2g-2-2d}X\times_{\Pic^{2d}(X)}\Pic^{d}(X)$, the map $S^{2g-2-2d}X\to\Pic^{2d}(X)$ is given by $D\mapsto K_{X}(-D)$, the map $\Pic^{d}(X)\to\Pic^{2d}(X)$ is given by $L\mapsto L^{2}$ and $K_{X}L^{-2}=\cO_{X}(\dv(s))$.
\end{enumerate}
\end{proposition}

\begin{proposition}\label{fixed locus of the quotient of bR1ss}
The fixed locus of the $\CC^{*}$-action on $\bR_{1}^{ss}/\!/\SL(2)$ is
$$(\bR_{1}^{ss}/\!/\SL(2))^{\CC^{*}}=\bN_{1}\sqcup\bigsqcup_{d=1}^{g-1}\overline{\pi}_{\bR_{1}}^{-1}(F_{d})\sqcup\bigsqcup_{j=1}^{2^{2g}}t_{j},$$
where $\bN_{1}$ is the strict transform of $\bN$ under $\overline{\pi}_{\bR_{1}}$ and $t_{j}=\PP(H^{0}(K_{X})\otimes sl(2))^{ss}/\!/\SL(2)$.
\end{proposition}
\begin{proof}
The component of the fixed locus $(\bR_{1}^{ss}/\!/\SL(2))^{\CC^{*}}$ not contained in $\dss\overline{\pi}_{\bR_{1}}^{-1}(\bN-\ZZ_{2}^{2g})\sqcup\bigsqcup_{d=1}^{g-1}\overline{\pi}_{\bR_{1}}^{-1}(F_{d})$ lies over $\ZZ_{2}^{2g}$.

The fiber of $\overline{\pi}_{\bR_{2}}$ over the points of $\ZZ_{2}^{2g}$ is $\PP\Upsilon^{-1}(0)^{ss}/\!/\SL(2)$. The $\CC^*$-action on $\PP\Upsilon^{-1}(0)^{ss}/\!/\SL(2)$ is given by
$$\lambda\cdot\left(\left(\begin{matrix}d&e\\f&-d\end{matrix}\right),\left(\begin{matrix}a&b\\c&-a\end{matrix}\right)\right)=\left(\left(\begin{matrix}\lambda d&\lambda e\\\lambda f&-\lambda d\end{matrix}\right),\left(\begin{matrix}a&b\\c&-a\end{matrix}\right)\right).$$
Thus
$$(\PP\Upsilon^{-1}(0)^{ss}/\!/\SL(2))^{\CC^{*}}=\{\left(\left(\begin{matrix}d&e\\f&-d\end{matrix}\right),\left(\begin{matrix}a&b\\c&-a\end{matrix}\right)\right)\in\PP\Upsilon^{-1}(0)^{ss}/\!/\SL(2)\,|\,d=e=f=0\}$$
$$\sqcup\{\left(\left(\begin{matrix}d&e\\f&-d\end{matrix}\right),\left(\begin{matrix}a&b\\c&-a\end{matrix}\right)\right)\in\PP\Upsilon^{-1}(0)^{ss}/\!/\SL(2)\,|\,a=b=c=0\}$$
$$\cong\PP(H^{1}(\cO_{X})\otimes sl(2))^{ss}/\!/\SL(2)\sqcup\PP(H^{0}(K_{X})\otimes sl(2))^{ss}/\!/\SL(2).$$

Since the disjoint union of $\overline{\pi}_{\bR_{1}}^{-1}(\bN-\ZZ_{2}^{2g})$ and $2^{2g}$ copies of $\PP(H^{1}(\cO_{X})\otimes sl(2))^{ss}/\!/\SL(2)$ is isomorphic to the strict transform $\bN_{1}$ of $\bN$, we get the conclusion.
\end{proof}

We consider special kinds of normal complex quasi-projective varieties. Let $R$ is a normal complex quasi-projective variety with a $\CC^{*}$-action. If $\dss\lim_{\lambda\to 0}\lambda\cdot x$ uniquely exists in $R^{\CC^{*}}$ for every $x\in R$, we can give a decomposition of Bia{\l}ynicki–Birula type on $R$.

\begin{proposition}[Bia{\l}ynicki–Birula decomposition, \cite{BB73}, \cite{BB76}, 4.4 of \cite{Will20}, Theorem 5.5 of \cite{FM20}]\label{BB decomposition on R}
Let $R$ be a normal complex quasi-projective variety with a $\CC^{*}$-action. $R^{\text{sm}}$ denotes the smooth locus of $R$. Assume that $\dss\lim_{\lambda\to 0}\lambda\cdot x$ uniquely exists in $R^{\CC^{*}}$ for every $x\in R$. Then
\begin{enumerate}
\item[1.] $R$ decomposes into $\CC^{*}$-invariant locally closed subsets
$$R=\bigsqcup_{C\in\pi_{0}(R^{\CC^{*}})}R_{C}^{+},$$
where $R_{C}^{+}$ is the attracting set $\dss\{x\in R\,|\,\lim_{\lambda\to0}\lambda\cdot x\in C\}$ for each connected component $C\subset R^{\CC^{*}}$.

\item[2.] The limit map
$$R_{C}^{+}\to C,\quad x\mapsto\lim_{\lambda\to0}\lambda\cdot x$$
is a morphism of algebraic varieties, and it is a bundle of affine spaces if $C\subset R^{\text{sm}}$.
\item[3.] Each connected component of $(R^{\text{sm}})^{\CC^{*}}$ is smooth.
\item[4.] There exists an ordering $C_{0},C_{1},\cdots$ of the components of $R^{\CC^{*}}$ such
that
\begin{itemize}
\item $i<j$ $\Leftrightarrow$ $\dim R_{C_{i}}^{+}\ge\dim R_{C_{j}}^{+}$;

\item the sets
$$R_{i}=\bigcup_{i\le j}R_{C_{j}}^{+}$$
yield a filtration of $R$ by closed subvarieties.
\end{itemize}
\end{enumerate}
\end{proposition}

Assume that the $i$-th connected component $C_{i}$ of $R^{\CC^{*}}$ is smooth. Applying the local-global spectral sequence (\cite[Proof of Proposition 3.4.4]{Soe01}, \cite[\S4.4]{Will20}) associated to the filtration $R_{i}\times_{\SL(2)}\E\SL(2)$ of $R\times_{\SL(2)}\E\SL(2)$
$$E_{1}^{ij}=H^{i+j}((R_{i}\setminus R_{i+1})\times_{\SL(2)}\E\SL(2),u_{i}^{!}\QQ_{R\times_{\SL(2)}\E\SL(2)})\Rightarrow H^{i+j}(R\times_{\SL(2)}\E\SL(2),\QQ)$$
where $u_{i}:R_{C_{i}}^{+}\times_{\SL(2)}\E\SL(2)\hookrightarrow R\times_{\SL(2)}\E\SL(2)$ is the inclusion.

If $R$ is smooth and $R_{C_{i}}^{+}$ is smooth of codimension $c_{i}$ in $R$, we have $$u_{i}^{!}\QQ_{R\times_{\SL(2)}\E\SL(2)}=\QQ_{R_{C_{i}}^{+}\times_{\SL(2)}\E\SL(2)}[2c_{i}]$$
and we can rewrite the terms of our spectral sequence as
$$E_{1}^{ij}=H^{i+j}(R_{C_{i}}^{+}\times_{\SL(2)}\E\SL(2),u_{i}^{!}\QQ_{R\times_{\SL(2)}\E\SL(2)})$$
$$=H^{i+j-2c_{i}}(R_{C_{i}}^{+}\times_{\SL(2)}\E\SL(2),\QQ)=H^{i+j-2c_{i}}(C_{i}\times_{\SL(2)}\E\SL(2),\QQ).$$

\begin{proposition}[Proof of Proposition 3.4.4 of \cite{Soe01}, \S4.4 of \cite{Will20}]\label{lgss}
Let $R$ be a normal complex quasi-projective variety with a $\CC^{*}$-action. $R^{\text{sm}}$ denotes the smooth locus of $R$. Assume that $\dss\lim_{\lambda\to 0}\lambda\cdot x$ uniquely exists in $R^{\CC^{*}}$ for every $x\in R$. Then
\begin{enumerate}
\item[1.] The decomposition of $R$ in Proposition \ref{BB decomposition on R} yields the spectral sequence

\begin{equation}\label{lgss1}E_{1}^{ij}=H^{i+j}(R_{C_{i}}^{+}\times_{\SL(2)}\E\SL(2),u_{i}^{!}\QQ_{R\times_{\SL(2)}\E\SL(2)})\Rightarrow H_{\SL(2)}^{i+j}(R,\QQ)\end{equation}

where $u_{i}:R_{C_{i}}^{+}\times_{\SL(2)}\E\SL(2)\hookrightarrow R\times_{\SL(2)}\E\SL(2)$ is the inclusion.
\item[2.] If $R$ is smooth and $R_{C_{i}}^{+}$ is smooth of codimension $c_{i}$ in $R$, then we can
rewrite the spectral sequence (\ref{lgss1}) as
\begin{equation}\label{lgss2}E_{1}^{ij}=H_{\SL(2)}^{i+j-2c_{i}}(C_{i},\QQ)\Rightarrow H_{\SL(2)}^{i+j}(R,\QQ)\end{equation}
\item[3.] The spectral sequence (\ref{lgss2}) degenerates at the first page $E_{1}$.
\end{enumerate}
\end{proposition}

Let $r_{\bR}:\bR\to\bR/\!/\SL(2)=\bM$, $r_{\bR_{1}}:\bR_{1}^{ss}\to\bR_{1}^{ss}/\!/\SL(2)$ and $r_{\bR_{2}}:\bR_{2}^{s}\to\bR_{2}^{s}/\!/\SL(2)$ be the good quotient maps which are $\CC^{*}$-equivariant. We have the following lemma useful later.

\begin{lemma}\label{1st assumption}
\begin{enumerate}
\item\label{1st assumption1} $\dss\bR^{\CC^{*}}=r_{\bR}^{-1}(\bN)\sqcup\bigsqcup_{d=1}^{g-1}r_{\bR}^{-1}(F_{d})$;

\item\label{1st assumption2} $\dss\bR_{1}^{\CC^{*}}=\widetilde{r_{\bR}^{-1}(\bN)}\sqcup\bigsqcup_{d=1}^{g-1}(\overline{\pi}_{\bR_{1}}\circ r_{\bR_{1}})^{-1}(F_{d})\sqcup\bigsqcup_{j=1}^{2^{2g}}\PP(H^{0}(K_{X})\otimes sl(2))$, where $\widetilde{r_{\bR}^{-1}(\bN)}$ is the strict transform of $r_{\bR}^{-1}(\bN)$ under $\pi_{\bR_{1}}$;

\item\label{1st assumption3} $\dss\bR_{2}^{\CC^{*}}=\widetilde{r_{\bR_{1}}^{-1}(\bN_{1})}\sqcup\bigsqcup_{d=1}^{g-1}(\overline{\pi}_{\bR_{1}}\circ\overline{\pi}_{\bR_{2}}\circ r_{\bR_{2}})^{-1}(F_{d})\sqcup\bigsqcup_{j=1}^{2^{2g}}\pi_{\bR_{2}}^{-1}(\PP(H^{0}(K_{X})\otimes sl(2))^{ss})\sqcup T^{+}$, where $\widetilde{r_{\bR_{1}}^{-1}(\bN_{1})}$ is the strict transform of $r_{\bR_{1}}^{-1}(\bN_{1})$ under $\pi_{\bR_{2}}$ and $T^{+}$ is a $\PP^{2g-3}$-bundle over $\bR_{1}^{\CC^{*}}\cap\Sig$.
\end{enumerate}
\end{lemma}
\begin{proof}
\begin{enumerate}
\item It is an immediate consequence from Proposition \ref{fixed locus of bM} the fact that $(E,\phi)$ is fixed under the $\CC^{*}$-action on $\bM$ if and only if $(E,\phi,\beta)$ is fixed under the $\CC^{*}$-action on $\bR$ for all $\beta$.

\item The component of the fixed locus $\bR_{1}^{\CC^{*}}$ not contained in
$$\pi_{\bR_{1}}^{-1}(r_{\bR}^{-1}(\bN)-\ZZ_{2}^{2g})\sqcup\bigsqcup_{d=1}^{g-1}\pi_{\bR_{1}}^{-1}(r_{\bR}^{-1}(F_{d}))$$
lies over $\ZZ_{2}^{2g}$.
 
The fiber of $\pi_{\bR_{1}}$ over the points of $\ZZ_{2}^{2g}$ is $\PP\Upsilon^{-1}(0)$. The $\CC^*$-action on $\PP\Upsilon^{-1}(0)$ is given by
$$\lambda\cdot\left(\left(\begin{matrix}d&e\\f&-d\end{matrix}\right),\left(\begin{matrix}a&b\\c&-a\end{matrix}\right)\right)=\left(\left(\begin{matrix}\lambda d&\lambda e\\\lambda f&-\lambda d\end{matrix}\right),\left(\begin{matrix}a&b\\c&-a\end{matrix}\right)\right).$$
Thus
$$(\PP\Upsilon^{-1}(0))^{\CC^{*}}=\{\left(\left(\begin{matrix}d&e\\f&-d\end{matrix}\right),\left(\begin{matrix}a&b\\c&-a\end{matrix}\right)\right)\in\PP\Upsilon^{-1}(0)\,|\,d=e=f=0\}$$
$$\sqcup\{\left(\left(\begin{matrix}d&e\\f&-d\end{matrix}\right),\left(\begin{matrix}a&b\\c&-a\end{matrix}\right)\right)\in\PP\Upsilon^{-1}(0)\,|\,a=b=c=0\}$$
$$\cong\PP(H^{1}(\cO_{X})\otimes sl(2))\sqcup\PP(H^{0}(K_{X})\otimes sl(2)).$$

Since the disjoint union of $\pi_{\bR_{1}}^{-1}(r_{\bR}^{-1}(\bN)-\ZZ_{2}^{2g})$ and $2^{2g}$ copies of $\PP(H^{1}(\cO_{X})\otimes sl(2))$ is isomorphic to the strict transform $\widetilde{r_{\bR}^{-1}(\bN)}$, we get the conclusion.

\item Let $T:=\bR_{1}^{\CC^{*}}\cap\Sig$. The component of the fixed locus $\bR_{2}^{\CC^{*}}$ not contained in
$$\pi_{\bR_{2}}^{-1}(\widetilde{r_{\bR}^{-1}(\bN)}-T)\sqcup\bigsqcup_{d=1}^{g-1}\pi_{\bR_{2}}^{-1}((\overline{\pi}_{\bR_{1}}\circ r_{\bR_{1}})^{-1}(F_{d}))\sqcup\bigsqcup_{j=1}^{2^{2g}}\pi_{\bR_{2}}^{-1}(\PP(H^{0}(K_{X})\otimes sl(2))^{ss})$$
lies over $T$.

The fiber of $\pi_{\bR_{2}}$ over the points of $T\setminus\pi_{\bR_{1}}^{-1}(\ZZ_{2}^{2g})$ is $\PP\Psi^{-1}(0)$. Since $\CC^{*}$ acts on $\PP\Psi^{-1}(0)$ by $\mu\cdot[a,b,c,d]=[\mu a,b,\mu c,d]$, we have
$$(\PP\Psi^{-1}(0))^{\CC^{*}}=\{[a,b,c,d]\in\PP\Psi^{-1}(0)\,|\,b=d=0\}\sqcup\{[a,b,c,d]\in\PP\Psi^{-1}(0)\,|\,a=c=0\}$$
$$\cong P_{1}\sqcup P_{2},$$
where $P_{1}=\PP(H^{0}(L^{-2}K_{X})\oplus H^{0}(L^{2}K_{X}))\cong\PP^{2g-3}$ and $P_{2}=\PP(H^{1}(L^{2})\oplus H^{1}(L^{-2}))\cong\PP^{2g-3}$.

The fiber of $\pi_{\bR_{2}}$ over the points $[\varphi]$ of $T\cap\pi_{\bR_{1}}^{-1}(\ZZ_{2}^{2g})$ is $\PP\Hom^{\om_{\varphi}}(\ker\varphi,\im\varphi^{\perp}/\im\varphi)$. The $\CC^{*}$-action on $\PP\Hom^{\om_{\varphi}}(\ker\varphi,\im\varphi^{\perp}/\im\varphi)$ is given by
$$\mu\cdot[Z_{32},\cdots,Z_{2g,2},Z_{33},\cdots,Z_{2g,3}]$$
$$=[Z_{32},\mu Z_{42},\cdots, Z_{2g-1,2}, \mu Z_{2g,2}, Z_{33}, \mu Z_{43}, \cdots, Z_{2g-1,3}, \mu Z_{2g,3}].$$
Thus
$$(\PP\Hom^{\om_{\varphi}}(\ker\varphi,\im\varphi^{\perp}/\im\varphi))^{\CC^{*}}$$
$$=\{[Z_{32},\cdots,Z_{2g,2},Z_{33},\cdots,Z_{2g,3}]\in\PP\Hom^{\om_{\varphi}}(\ker\varphi,\im\varphi^{\perp}/\im\varphi)\,|$$
$$Z_{32}=Z_{52}=\cdots=Z_{2g-1,2}=Z_{33}=Z_{53}=\cdots=Z_{2g-1,3}=0\}$$
$$\sqcup\{[Z_{32},\cdots,Z_{2g,2},Z_{33},\cdots,Z_{2g,3}]\in\PP\Hom^{\om_{\varphi}}(\ker\varphi,\im\varphi^{\perp}/\im\varphi)\,|$$
$$Z_{42}=Z_{62}=\cdots=Z_{2g,2}=Z_{43}=Z_{63}=\cdots=Z_{2g,3}=0\}$$
$$=P_{1}\sqcup P_{2},$$
where
$$P_{1}=\{[Z_{32},\cdots,Z_{2g,2},Z_{33},\cdots,Z_{2g,3}]\in\PP\Hom^{\om_{\varphi}}(\ker\varphi,\im\varphi^{\perp}/\im\varphi)\,|$$
$$Z_{32}=Z_{52}=\cdots=Z_{2g-1,2}=Z_{33}=Z_{53}=\cdots=Z_{2g-1,3}=0\}\cong\PP^{2g-3}$$
and
$$P_{2}=\{[Z_{32},\cdots,Z_{2g,2},Z_{33},\cdots,Z_{2g,3}]\in\PP\Hom^{\om_{\varphi}}(\ker\varphi,\im\varphi^{\perp}/\im\varphi)^{ss}\,|$$
$$Z_{42}=Z_{62}=\cdots=Z_{2g,2}=Z_{43}=Z_{63}=\cdots=Z_{2g,3}=0\}\cong\PP^{2g-3}.$$

Let $T^{+}$ is the $P_{1}$-bundle over $T$. Since the $P_{2}$-bundle over $T$ is contained in the strict transform $\widetilde{r_{\bR_{1}}^{-1}(\bN_{1})}$, we get the conclusion.
\end{enumerate}
\end{proof}

\begin{lemma}\label{codim 0 stratum}
\begin{enumerate}
\item $\bR_{r_{\bR}^{-1}(\bN)}^{+}$ is open in $\bR$.

\item $(\bR_{1})_{\widetilde{r_{\bR}^{-1}(\bN)}}^{+}$ is open in $\bR_{1}$.
\end{enumerate}
\end{lemma}
\begin{proof}
\begin{enumerate}
\item We claim that $\bR_{r_{\bR}^{-1}(\bN)}^{+}$ has codimension $0$ in $\bR$. Let $\bN^{s}$ denote the stable locus of $\bN$. Since $r_{\bR}^{-1}(T^{*}\bN^{s})$ has codimension $0$ in $\bR$, it suffices to show that $r_{\bR}^{-1}(T^{*}\bN^{s})\subset\bR_{r_{\bR}^{-1}(\bN)}^{+}$. If $(E,\phi,\beta)\in r_{\bR}^{-1}(T^{*}\bN^{s})$ and $(E,\phi,\beta)$ is fixed under the $\CC^{*}$-action on $\bR$, then there exists $g\in\Aut(E)$ such that $(g\otimes \id_{K_{X}})^{-1}\phi g=\lambda\phi$ for all $\lambda\in\CC^{*}$. Since $E$ is stable, $g$ should be of the form $k\cdot\id$ for some $k\in\CC^{*}$. Then we have $\phi=\lambda\phi$ for all $\lambda\in\CC^{*}$. Thus $\phi=0$ and $(E,\phi,\beta)\in r_{\bR}^{-1}(\bN)$. Hence $r_{\bR}^{-1}(T^{*}\bN^{s})\subset\bR_{r_{\bR}^{-1}(\bN)}^{+}$.

\item We claim that $(\bR_{1})_{\widetilde{r_{\bR}^{-1}(\bN)}}^{+}$ has codimension $0$ in $\bR_{1}$. Since $r_{\bR}^{-1}(T^{*}\bN^{s})$ doesn't change under $\pi_{\bR_{1}}$, we have $(r_{\bR}\circ\pi_{\bR_{1}})^{-1}(T^{*}\bN^{s})\subset(\bR_{1})_{\widetilde{r_{\bR}^{-1}(\bN)}}^{+}$. Since $r_{\bR}^{-1}(T^{*}\bN^{s})$ has codimension $0$ in $\bR$, we get the conclusion.
\end{enumerate}
\end{proof}

\begin{lemma}\label{codim g-1+2d strata}
$\bR_{r_{\bR}^{-1}(F_{d})}^{+}$ has codimension $g-1+2d$ in $\bR$.
\end{lemma}
\begin{proof}
We claim that $\bR_{r_{\bR}^{-1}(F_{d})}^{+}$ is equal to the locus of stable triples $(E,\phi,\beta)\in\bR$ with unstable $E$ that fits into an exact sequence
$$0\to L\to E\to L^{-1}\to0$$
with $\deg L=d$ and $1\le d\le g-1$. Note that the locus of such triples has dimension $5g-5-2d$ from \cite[3.3]{KY08}.
\begin{itemize}
\item $\subset$ : Assume that $(E,\phi,\beta)\in\bR$ is semistable with respect to $\SL(2)$-action with semistable $E$. Then $\dss\lim_{\lambda\to 0}\lambda\cdot(E,\phi,\beta)=(E,0,\beta)$. Hence $(E,\phi,\beta)\in\bR_{r_{\bR}^{-1}(\bN)}^{+}$ which means that $(E,\phi,\beta)\not\in\bR_{r_{\bR}^{-1}(F_{d})}^{+}$.

\item $\supset$ : Assume that $(E,\phi,\beta)\in\bR$ is stable with respect to $\SL(2)$-action with unstable $E$ that fits into an exact sequence
$$0\to L\to E\to L^{-1}\to0$$
with $\deg L=d$ and $1\le d\le g-1$ and that $(E,\phi,\beta)$ is fixed under the $\CC^{*}$-action on $\bR$. Note that $(E,\phi)$ is also stable.

Let $\{U_{i}\}$ be an open cover of $X$ so that both $E$ and $K_{X}$ are trivialized on each $U_{i}$. Let $\left(\begin{matrix}a_{ij}&b_{ij}\\0&a_{ij}^{-1}\end{matrix}\right)$ be the transition matrices of $E$ and $\left(\begin{matrix}p_{i}&q_{i}\\r_{i}&-p_{i}\end{matrix}\right)$ the local Higgs fields, where $a_{ij}$ are the transition functions of $L$. Since $L$ is not $\phi$-invariant, $r_{i}\ne0$ for some $i$. Then for each $\lambda\in\CC^{*}$ there exists $A_{\lambda}\in\Aut(E)$ such that
$$A_{\lambda}|_{U_{i}}=\left(\begin{matrix}e(\lambda)&f_{i}(\lambda)\\g_{i}(\lambda)&h(\lambda)\end{matrix}\right)\in\Aut(E|_{U_{i}}),$$
\begin{equation}\label{fixed4}\left(\begin{matrix}e(\lambda)&f_{i}(\lambda)\\g_{i}(\lambda)&h(\lambda)\end{matrix}\right)\left(\begin{matrix}p_{i}&q_{i}\\r_{i}&-p_{i}\end{matrix}\right)=\lambda\left(\begin{matrix}p_{i}&q_{i}\\r_{i}&-p_{i}\end{matrix}\right)\left(\begin{matrix}e(\lambda)&f_{i}(\lambda)\\g_{i}(\lambda)&h(\lambda)\end{matrix}\right)\end{equation}
and
\begin{equation}\label{fixed5}\left(\begin{matrix}e(\lambda)&f_{j}(\lambda)\\g_{i}(\lambda)&h(\lambda)\end{matrix}\right)\left(\begin{matrix}a_{ij}&b_{ij}\\0&a_{ij}^{-1}\end{matrix}\right)=\left(\begin{matrix}a_{ij}&b_{ij}\\0&a_{ij}^{-1}\end{matrix}\right)\left(\begin{matrix}e(\lambda)&f_{i}(\lambda)\\g_{i}(\lambda)&h(\lambda)\end{matrix}\right).\end{equation}

It follows from (\ref{fixed5}) that $b_{ij}g_{i}(\lambda)=0$ for $i,j$ and
\begin{equation}\label{fixed6}b_{ij}(e(\lambda)-h(\lambda))=a_{ij}f_{i}(\lambda)-f_{j}(\lambda)a_{ij}^{-1}.\end{equation}

If $e(\lambda)=h(\lambda)$ for all $\lambda$, then $\{f_{i}(\lambda)\}$ gives a section $f(\lambda)$ of $H^{0}(L^{2})$ by (\ref{fixed6}). We may assume that $g_{i}(\lambda)=0$ for all $i$. Then from (\ref{fixed4}), we have $p_{i}=q_{i}=r_{i}=0$, which is a contradiction to the assumption that $E$ is unstable. Thus we have $e(\lambda)\ne h(\lambda)$ for some $\lambda$. Since $r_{i}\ne0$, $\lambda\ne 1$ from (\ref{fixed4}). Now we may assume that $f(\lambda)=0$. Then $p_{i}=0$ from (\ref{fixed4}). Then $b_{ij}=0$ from (\ref{fixed5}). Thus $E=L\oplus L^{-1}$ and $\phi=\left(\begin{matrix}0&q\\r&0\end{matrix}\right)$ with $r\ne0$. Further since
$$\left(\begin{matrix}\lambda&0\\0&1\end{matrix}\right)\left(\lambda\left(\begin{matrix}0&q\\r&0\end{matrix}\right)\right)\left(\begin{matrix}\lambda^{-1}&0\\0&1\end{matrix}\right)=\left(\begin{matrix}0&\lambda^{2}q\\r&0\end{matrix}\right),$$
$\phi$ is of the form $\left(\begin{matrix}0&0\\r&0\end{matrix}\right)$ with $r\ne0$ after taking limit as $\lambda\to0$.

Hence we can conclude that $(E,\phi,\beta)\in\bR_{r_{\bR}^{-1}(F_{d})}^{+}$.
\end{itemize}
\end{proof}

Now we give a conjecture and its corollary that guarantee the assumption of Proposition \ref{BB decomposition on R} and Proposition \ref{lgss}.

\begin{conjecture}\label{conjecture}
$\dss\lim_{\lambda\to 0}\lambda\cdot x$ uniquely exists in $\bR^{\CC^{*}}$ for every $x\in \bR.$
\end{conjecture}

\begin{corollary}\label{cor of conjecture}
Assume that Conjecture \ref{conjecture} holds.
\begin{enumerate}
\item\label{cor of conjecture1} $\dss\lim_{\lambda\to 0}\lambda\cdot x$ uniquely exists in $\bR_{1}^{\CC^{*}}$ for every $x\in\bR_{1}$.

\item $\dss\lim_{\lambda\to 0}\lambda\cdot x$ uniquely exists in $(\bR_{1}^{ss})^{\CC^{*}}$ for every $x\in\bR_{1}^{ss}$.

\item $\dss\lim_{\lambda\to 0}\lambda\cdot x$ uniquely exists in $\bR_{2}^{\CC^{*}}$ for every $x\in \bR_{2}$.
\end{enumerate}
\end{corollary}
\begin{proof}
\begin{enumerate}
\item Since $\pi_{\bR_{1}}:\bR_{1}\to\bR$ is proper, it follows from Conjecture \ref{conjecture} that $\dss\lim_{\lambda\to 0}\lambda\cdot x$ uniquely exists in $(\bR_{1})^{\CC^{*}}$ for every $x\in\bR_{1}$.

\item Since $\bR_{1}^{ss}$ is invariant under the $\CC^{*}$-action on $\bR_{1}$, the statement follows from the item (\ref{cor of conjecture1}).

\item Since $\pi_{\bR_{2}}:\bR_{2}\to\bR_{1}^{ss}$ is proper, it follows from Conjecture \ref{conjecture} that $\dss\lim_{\lambda\to 0}\lambda\cdot x$ uniquely exists in $(\bR_{2})^{\CC^{*}}$ for every $x\in\bR_{2}$.
\end{enumerate}
\end{proof}

Conjecture \ref{conjecture} and Corollary \ref{cor of conjecture} guarantee the existences of the spectral sequences (\ref{lgss2}) in the cases that $R=\bR$, $R=\bR_{1}$, $R=\bR_{1}^{ss}$ and $R=\bR_{2}$. So we have the following injective pullbacks $\pi_{\bR_{1}}^{*}$ and $\pi_{\bR_{2}}^{*}$ on the equivariant cohomologies.

\begin{lemma}\label{injectivity of pullback on equivariant cohomology}
Assume that Conjecture \ref{conjecture} holds.
\begin{enumerate}
\item\label{injectivity of pullback on equivariant cohomology-(1)} $\pi_{\bR_{1}}^{*}:H_{\SL(2)}^{i}(\bR)\to H_{\SL(2)}^{i}(\bR_{1})$ is injective for each $i$.
\item\label{injectivity of pullback on equivariant cohomology-(2)} $\pi_{\bR_{2}}^{*}:H_{\SL(2)}^{i}(\bR_{1}^{ss})\to H_{\SL(2)}^{i}(\bR_{2})$ is injective for each $i$.
\end{enumerate}
\end{lemma}
\begin{proof}
\begin{enumerate}
\item By Theorem \ref{fixed locus of bM} and \cite[Corollary 6.12]{Simp94II}, we have
$$\bM^{\CC^{*}}=\bN\sqcup\bigsqcup_{d=1}^{g-1}F_{d},$$
where $F_{0}=\bN$. Let $r_{\bR}:\bR\to\bR/\!/\SL(2)=\bM$ be the good quotient map. By Lemma \ref{1st assumption}-(\ref{1st assumption1}), we have
$$\bR^{\CC^{*}}=r_{\bR}^{-1}(\bN)\sqcup\bigsqcup_{d=1}^{g-1}r_{\bR}^{-1}(F_{d})=C_{0}\sqcup\bigsqcup_{d=1}^{g-1}\bigsqcup_{k}C_{d,k},$$
where $C_{0}=r_{\bR}^{-1}(\bN)$ and $C_{d,k}$ are connected components of $r_{\bR}^{-1}(F_{d})$.

Then it follows from Proposition \ref{lgss}, Lemma \ref{codim 0 stratum}, Lemma \ref{codim g-1+2d strata} and Conjecture \ref{conjecture} that the spectral sequence
$$E_1^{0j}=H_{\SL(2)}^{j}(\bR_{C_{0}}^{+},\QQ)\text{ or }E_{1}^{ij}=H_{\SL(2)}^{i+j-2c_{i}}(\bigsqcup_{k}C_{i,k},\QQ)\Rightarrow H_{\SL(2)}^{i+j}(\bR,\QQ)$$
which degenerates at $E_{1}$, where $c_{i}=g-1+2i$. Further, since the limit map $\bR_{C_{0}}^{+}\to C_{0}$ is affine by \cite[Theorem 1.4.2]{Drin15} and its fiber is connected, the Leray spectral sequence for the limit map gives
$$H_{\SL(2)}^{j}(\bR_{C_{0}}^{+},\QQ)=H_{\SL(2)}^{j}(C_{0},\QQ)=H_{\SL(2)}^{j}(r_{\bR}^{-1}(\bN),\QQ).$$

Let $r_{\bR_{1}}:\bR_{1}^{ss}\to\bR_{1}^{ss}/\!/\SL(2)=\bM$ be the good quotient map. 
By Lemma \ref{1st assumption}-(\ref{1st assumption2}), we have
$$\bR_{1}^{\CC^{*}}=\widetilde{r_{\bR}^{-1}(\bN)}\sqcup\bigsqcup_{d=1}^{g-1}(\overline{\pi}_{\bR_{1}}\circ r_{\bR_{1}})^{-1}(F_{d})\sqcup\bigsqcup_{j=1}^{2^{2g}}\PP(H^{0}(K_{X})\otimes sl(2))$$
$$=\tC_{0}\sqcup\bigsqcup_{l=1}^{2^{2g}}\tC_{1,l}\sqcup\bigsqcup_{d=1}^{g-1}\bigsqcup_{k}\tC_{1+d,k}$$
where $\tC_{0}=\widetilde{r_{\bR}^{-1}(\bN)}$, $\tC_{1,l}=\PP(H^{0}(K_{X})\otimes sl(2))$ for $1\le l\le 2^{2g}$ and $\tC_{1+d,k}$ are connected components of $(\overline{\pi}_{\bR_{1}}\circ r_{\bR_{1}})^{-1}(F_{d})$ for $1\le d\le g-1$. Note that $(\bR_{1})_{\tC_{1,l}}^{+}\cong \tC_{1,l}$. Then it follows from Proposition \ref{lgss}, Lemma \ref{codim 0 stratum}, Lemma \ref{codim g-1+2d strata}, Conjecture \ref{conjecture} and Corollary \ref{cor of conjecture} that the spectral sequence
$$E_{1}^{0j}=H_{\SL(2)}^{j}((\bR_{1})_{\tC_{0}}^{+},\QQ), E_{1}^{1j}=H_{\SL(2)}^{1+j-2c}(\bigsqcup_{l=1}^{2^{2g}}\tC_{1,l},\QQ)$$
$$\text{ or }E_{1}^{ij}=H_{\SL(2)}^{i+j-2d_{i}}(\bigsqcup_{k}\tC_{i,k},\QQ)\text{ for }i\ge2$$
$$\Rightarrow H_{\SL(2)}^{i+j}(\bR_{1},\QQ)$$
which degenerates at $E_{1}$, where $(\bR_{1})_{\tC_{0}}^{+}$ is a subvariety of $\bR_{1}$ of codimension $0$, $(\bR_{1})_{\tC_{1,l}}^{+}$ are smooth subvarieties of $\bR_{1}$ of codimension $c=3g-2$ and $(\bR_{1})_{\tC_{i,k}}^{+}$ are smooth subvarieties of $\bR_{1}$ of codimension $d_{i}=g-1+2(i-1)$. Further, since the limit map $(\bR_{1})_{\tC_{0}}^{+}\to\tC_{0}$ is affine by \cite[Theorem 1.4.2]{Drin15} and its fiber is connected, the Leray spectral sequence for the limit map gives
$$H_{\SL(2)}^{j}((\bR_{1})_{\tC_{0}}^{+},\QQ)=H_{\SL(2)}^{j}(\tC_{0},\QQ)=H_{\SL(2)}^{j}(\widetilde{r_{\bR}^{-1}(\bN)},\QQ).$$

Fix $x\in X$. By \cite[Proposition 3.3]{BGM10}, $r_{\bR}^{-1}(\bN)$ is isomorphic to the moduli space $\cN$ of $\tau$-stable pairs $(E,\beta)$ where $\tau$ is a real number, $E$ is a rank $2$ vector bundle and $\beta$ is an isomorphism $\beta:E|_{x}\to\CC^{2}$. Further $\cN$ is a nonsingular irreducible quasi-projective variety of dimension $3g$ for some $\tau$ by \cite[Lemma 1.2]{BGM10} and it consists of semistable points with respect to the $\SL(2)$-action $g\cdot(E,\beta)=(E,g\circ\beta)$ on $\cN$ by \cite[Lemma 3.2]{BGM10}. Moreover 
we see that $\widetilde{r_{\bR}^{-1}(\bN)}$ is isomorphic to the blowing-up $\cN_{1}$ of $\cN$ along $\ZZ_{2}^{2g}$. Thus
$$\pi_{\bR_{1}}^{*}:H_{\SL(2)}^{*}(r_{\bR}^{-1}(\bN))\to H_{\SL(2)}^{*}(\widetilde{r_{\bR}^{-1}(\bN)})$$
is injective.

Moreover since $\pi_{\bR_{1}}^{*}:H_{\SL(2)}^{*}(r_{\bR}^{-1}(F_{d}))\to H_{\SL(2)}^{*}((\overline{\pi}_{\bR_{1}}\circ r_{\bR_{1}})^{-1}(F_{d}))$ is an isomorphism for each $1\le d\le g-1$, $\pi_{\bR_{1}}^{*}:H_{\SL(2)}^{i}(\bR)\to H_{\SL(2)}^{i}(\bR_{1})$ is also injective for each $i$.

\item By Lemma \ref{1st assumption}-(\ref{1st assumption2}), we have
$$(\bR_{1}^{ss})^{\CC^{*}}=\widetilde{r_{\bR}^{-1}(\bN)}^{ss}\sqcup\bigsqcup_{d=1}^{g-1}(\overline{\pi}_{\bR_{1}}\circ r_{\bR_{1}})^{-1}(F_{d})\sqcup\bigsqcup_{j=1}^{2^{2g}}\PP(H^{0}(K_{X})\otimes sl(2))^{ss}$$
$$=\tC_{0}^{ss}\sqcup\bigsqcup_{l=0}^{2^{2g}}\tC_{1,l}^{ss}\sqcup\bigsqcup_{d=1}^{g-1}\bigsqcup_{k}\tC_{1+d,k}.$$
Then it follows from Proposition \ref{lgss}, Lemma \ref{codim 0 stratum}, Lemma \ref{codim g-1+2d strata}, Conjecture \ref{conjecture} and Corollary \ref{cor of conjecture} that the spectral sequence
$$E_{1}^{0j}=H_{\SL(2)}^{j}((\bR_{1}^{ss})_{\tC_{0}^{ss}}^{+},\QQ),E_{1}^{1j}=H_{\SL(2)}^{1+j-2c}(\bigsqcup_{l=1}^{2^{2g}}\tC_{1,l}^{ss},\QQ)$$
$$\text{ or }E_{1}^{ij}=H_{\SL(2)}^{i+j-2d_{i}}(\bigsqcup_{k}\tC_{i,k},\QQ)\text{ for }i\ge2$$
$$\Rightarrow H_{\SL(2)}^{i+j}(\bR_{1}^{ss},\QQ)$$
which degenerates at $E_{1}$, where $(\bR_{1}^{ss})_{\tC_{0}^{ss}}^{+}$ is a subvariety of $\bR_{1}^{ss}$ of codimension $0$, $(\bR_{1}^{ss})_{\tC_{1,l}^{ss}}^{+}$ are smooth subvarieties of $\bR_{1}^{ss}$ of codimension $c=3g-2$ and $(\bR_{1}^{ss})_{\tC_{i,k}}^{+}$ are smooth subvarieties of $\bR_{1}^{ss}$ of codimension $d_{i}=g-1+2(i-1)$. Further, since the limit map $(\bR_{1}^{ss})_{\tC_{0}^{ss}}^{+}\to\tC_{0}^{ss}$ is affine by \cite[Theorem 1.4.2]{Drin15} and its fiber is connected, the Leray spectral sequence for the limit map gives
$$H_{\SL(2)}^{j}((\bR_{1}^{ss})_{\tC_{0}^{ss}}^{+},\QQ)=H_{\SL(2)}^{j}(\tC_{0}^{ss},\QQ)=H_{\SL(2)}^{j}(\widetilde{r_{\bR}^{-1}(\bN)}^{ss},\QQ).$$

Let $r_{\bR_{2}}:\bR_{2}^{s}\to\bR_{2}^{s}/\!/\SL(2)$ be the good quotient map. By Lemma \ref{1st assumption}-(\ref{1st assumption3}), we have
$$\bR_{2}^{\CC^{*}}=\widetilde{r_{\bR_{1}}^{-1}(\bN_{1})}\sqcup\bigsqcup_{d=1}^{g-1}(\overline{\pi}_{\bR_{1}}\circ\overline{\pi}_{\bR_{2}}\circ r_{\bR_{2}})^{-1}(F_{d})\sqcup\bigsqcup_{j=1}^{2^{2g}}\pi_{\bR_{2}}^{-1}(\PP(H^{0}(K_{X})\otimes sl(2))^{ss})\sqcup T^{+}$$
$$=\ttC_{0}\sqcup\ttC_{1}\sqcup\bigsqcup_{l=0}^{2^{2g}+1}\ttC_{2,l}\sqcup\bigsqcup_{d=1}^{g-1}\bigsqcup_{k}\ttC_{2+d,k},$$
where $\ttC_{0}=\widetilde{r_{\bR_{1}}^{-1}(\bN_{1})}$, $\ttC_{1}=T^{+}$, $\ttC_{2,l}=\pi_{\bR_{2}}^{-1}(\PP(H^{0}(K_{X})\otimes sl(2)))$ for $1\le l\le 2^{2g}$ and $\ttC_{2+d,k}$ are connected components of $(\overline{\pi}_{\bR_{1}}\circ\overline{\pi}_{\bR_{2}}\circ r_{\bR_{2}})^{-1}(F_{d})$. Note that $(\bR_{2})_{\ttC_{1}}\cong\ttC_{1}$. Then it follows from Proposition \ref{lgss}, Lemma \ref{codim 0 stratum}, Lemma \ref{codim g-1+2d strata}, Conjecture \ref{conjecture} and Corollary \ref{cor of conjecture} that the spectral sequence
$$E_{1}^{0j}=H_{\SL(2)}^{j}((\bR_{2})_{\ttC_{0}}^{+},\QQ), E_{1}^{1j}=H_{\SL(2)}^{1+j-2c}(\ttC_{1},\QQ),$$
$$E_{1}^{2j}=H_{\SL(2)}^{2+j-2d}(\bigsqcup_{l=1}^{2^{2g}}\ttC_{2,l},\QQ)\text{ or }E_{1}^{ij}=H_{\SL(2)}^{i+j-2d_{i}}(\bigsqcup_{k}\ttC_{i,k},\QQ)\text{ for }i\ge3$$
$$\Rightarrow H_{\SL(2)}^{i+j}(\bR_{2},\QQ)$$
which degenerates at $E_{1}$, where $(\bR_{2})_{\ttC_{0}}^{+}$ is a subvariety of $\bR_{2}$ of codimension $0$, $(\bR_{2})_{\ttC_{1}}^{+}$ is a smooth subvariety of $\bR_{2}$ of codimension $c=3g-3$, $(\bR_{2})_{\ttC_{2,l}}^{+}$ are smooth subvarieties of $\bR_{2}$ of codimension $d=3g-2$ and $(\bR_{2})_{\ttC_{i,k}}^{+}$ are smooth subvarieties of $\bR_{2}$ of codimension $d_{i}=g-1+2(i-2)$. Further, since the limit map $(\bR_{2})_{\ttC_{0}}^{+}\to\ttC_{0}$ is affine by \cite[Theorem 1.4.2]{Drin15} and its fiber is connected, the Leray spectral sequence for the limit map gives
$$H_{\SL(2)}^{j}((\bR_{2})_{\ttC_{0}}^{+},\QQ)=H_{\SL(2)}^{j}(\ttC_{0},\QQ)=H_{\SL(2)}^{j}(\widetilde{r_{\bR_{1}}^{-1}(\bN_{1})},\QQ).$$

We see that $\widetilde{r_{\bR_{1}}^{-1}(\bN_{1})}$ is isomorphic to the blowing-up of $\cN_{1}^{ss}$ along $\cN_{1}^{ss}\cap\Sig$. Thus $\pi_{\bR_{2}}^{*}:H_{\SL(2)}^{*}(\widetilde{r_{\bR}^{-1}(\bN)}^{ss})\to H_{\SL(2)}^{*}(\widetilde{r_{\bR_{1}}^{-1}(\bN_{1})})$ is injective.

Moreover since both
$$\pi_{\bR_{2}}^{*}:H_{\SL(2)}^{*}((\overline{\pi}_{\bR_{1}}\circ r_{\bR_{1}})^{-1}(F_{d}))\to H_{\SL(2)}^{*}((\overline{\pi}_{\bR_{1}}\circ\overline{\pi}_{\bR_{2}}\circ r_{\bR_{2}})^{-1}(F_{d}))$$
and
$$\pi_{\bR_{2}}^{*}:H_{\SL(2)}^{*}(\PP(H^{0}(K_{X})\otimes sl(2))^{ss})\to H_{\SL(2)}^{*}(\pi_{\bR_{2}}^{-1}(\PP(H^{0}(K_{X})\otimes sl(2))^{ss}))$$
are isomorphisms for each $1\le d\le g-1$ and $1\le l\le 2^{2g}$, $\pi_{\bR_{2}}^{*}:H_{\SL(2)}^{i}(\bR_{1}^{ss})\to H_{\SL(2)}^{i}(\bR_{2})$ is also injective for each $i$.
\end{enumerate}
\end{proof}

%%%%%%%%%%%%%%%%%%%%%%

With Lemma \ref{injectivity of pullback on equivariant cohomology}, we can use a standard argument to get the following formula.

\begin{lemma}\label{equivariant cohomology blowing-up formula}
Assume that Conjecture \ref{conjecture} holds.
\begin{enumerate}
\item\label{equivariant cohomology blowing-up formula-(1)} $P_{t}^{\SL(2)}(\bR_{1})=P_{t}^{\SL(2)}(\bR)+2^{2g}(P_{t}^{\SL(2)}(\PP\Upsilon^{-1}(0))-P_{t}(\BSL(2)))$.

\item\label{equivariant cohomology blowing-up formula-(2)} $P_{t}^{\SL(2)}(\bR_{2})=P_{t}^{\SL(2)}(\bR_{1}^{ss})+P_{t}^{\SL(2)}(E_2)-P_{t}^{\SL(2)}(\Sig)$.
\end{enumerate}
\end{lemma}

\begin{proof}
\begin{enumerate}
\item Let $U_{x}$ be a sufficiently small open neighborhood of $x\in\ZZ_{2}^{2g}$ in $\bR$, let $\dss U_{1}=\sqcup_{x\in\ZZ_{2}^{2g}}U_{x}$ and let $\tU_{1}=\pi_{\bR_{1}}^{-1}(U_{1})$. Let $V_{1}=\bR\setminus\ZZ^{2g}$. We can identify $V_{1}$ with $\bR_{1}\setminus E_{1}$ under $\pi_{\bR_{1}}$. Then we have the following commutative diagram
    $$\xymatrix@C-=0.2cm{\cdots\ar[r]&H_{\SL(2)}^{i-1}(U_{1}\cap V_{1})\ar[r]^(0.55){\alpha}\ar@^{=}[d]&H_{\SL(2)}^{i}(\bR)\ar[r]\ar[d]&H_{\SL(2)}^{i}(U_{1})\ar@<-4ex>[d]\oplus H_{\SL(2)}^{i}(V_{1})\ar[r]^(0.6){\beta}\ar@<4ex>@^{=}[d]&H_{\SL(2)}^{i}(U_{1}\cup V_{1})\ar[r]\ar@{=}[d]&\cdots\\
    \cdots\ar[r]&H_{\SL(2)}^{i-1}(U_{1}\cap V_{1})\ar[r]^(0.55){\tal}&H_{\SL(2)}^{i}(\bR_{1})\ar[r]&H_{\SL(2)}^{i}(\tU_{1})\oplus H_{\SL(2)}^{i}(V_{1})\ar[r]^(0.6){\tbe}&H_{\SL(2)}^{i}(U_{1}\cup V_{1})\ar[r]&\cdots,}$$
where the horizontal sequences are Mayer-Vietoris sequences and the vertical maps are $\pi_{\bR_{1}}^{*}$. It follows from Conjecture \ref{conjecture} and Lemma \ref{injectivity of pullback on equivariant cohomology}-(\ref{injectivity of pullback on equivariant cohomology-(1)}) that the vertical maps are injective. So $\ker\alpha=\ker\tal$ and then $\im\beta=\im\tbe$. Thus we have
$$P_{t}^{\SL(2)}(\bR_{1})=P_{t}^{\SL(2)}(\bR)+P_{t}^{\SL(2)}(\tU_{1})-P_{t}^{\SL(2)}(U_{1}).$$
By \cite[Theorem 3.1]{GM88} and Proposition \ref{normal slice is the quadratic cone}, $U_{1}$ is analytically isomorphic to $2^{2g}$ copies of $\Upsilon^{-1}(0)$ and then $\tU_{1}$ is analytically isomorphic to $2^{2g}$ copies of $Bl_{0}\Upsilon^{-1}(0)$. Since $Bl_{0}\Upsilon^{-1}(0)$ is the tautological line bundle $\cO_{\PP\Upsilon^{-1}(0)}(-1)$ over $\PP\Upsilon^{-1}(0)$ and the cohomology of the fiber is trivial, the Leray spectral sequence of equivariant cohomology associated to the projection $Bl_{0}\Upsilon^{-1}(0)\to\PP\Upsilon^{-1}(0)$ has $E_{2}$-term given by
$$E_{2}^{p0}=H_{\SL(2)}^{p}(\PP\Upsilon^{-1}(0))\otimes H^{0}(\CC)=H_{\SL(2)}^{p}(\PP\Upsilon^{-1}(0)).$$
Thus the spectral sequence degenerates at the $E_{2}$-term and then we have
$$H_{\SL(2)}^{i}(Bl_{0}\Upsilon^{-1}(0))\cong H_{\SL(2)}^{i}(\PP\Upsilon^{-1}(0))$$
for each $i\ge0$.

Hence we get
$$P_{t}^{\SL(2)}(\bR_{1})=P_{t}^{\SL(2)}(\bR)+2^{2g}(P_{t}^{\SL(2)}(\PP\Upsilon^{-1}(0))-P_{t}(\BSL(2))).$$

\item Let $U_{2}$ be a sufficiently small open neighborhood of $\Sig$ and let $\tU_{2}=\pi_{\bR_{2}}^{-1}(U_{2})$. Let $V_{2}=\bR_{1}^{ss}\setminus\Sig$. We can identify $V_{2}$ with $\bR_{2}\setminus E_{2}$ under $\pi_{\bR_{2}}$. By Conjecture \ref{conjecture}, Lemma \ref{injectivity of pullback on equivariant cohomology}-(\ref{injectivity of pullback on equivariant cohomology-(2)}) and the same way as in the proof of item (1), we have
$$P_{t}^{\SL(2)}(\bR_{2})=P_{t}^{\SL(2)}(\bR_{1}^{ss})+P_{t}^{\SL(2)}(\tU_{2})-P_{t}^{\SL(2)}(U_{2}).$$
By \cite[Theorem 3.1]{GM88} and Proposition \ref{normal slice is the quadratic cone}, $U_{2}$ is analytically isomorphic to $C_{\Sig}\bR_{1}$ and then $\tU_{2}$ is analytically isomorphic to $Bl_{\Sig}(C_{\Sig}\bR_{1})$. Since $C_{\Sig}\bR_{1}$ is the fibration on $\Sig$ whose fiber is an affine cone and the cohomology of the fiber is trivial, the Leray spectral sequence of equivariant cohomology associated to the projection $C_{\Sig}\bR_{1}\to\Sig$ has $E_{2}$-term given by
$$E_{2}^{p0}=H_{\SL(2)}^{p}(\Sig)\otimes H^{0}(\CC)=H_{\SL(2)}^{p}(\Sig).$$
Thus the spectral sequence degenerates at the $E_{2}$-term and then we have
$$H_{\SL(2)}^{i}(C_{\Sig}\bR_{1})\cong H_{\SL(2)}^{i}(\Sig)$$
for each $i\ge0$. 

Since $Bl_{\Sig}(C_{\Sig}\bR_{1})$ is the tautological line bundle $\cO_{E_{2}}(-1)$ over $E_{2}$ and the cohomology of the fiber is trivial, the Leray spectral sequence of equivariant cohomology associated to the projection $Bl_{\Sig}(C_{\Sig}\bR_{1})\to E_{2}$ has $E_{2}$-term given by
$$E_{2}^{p0}=H_{\SL(2)}^{p}(E_{2})\otimes H^{0}(\CC)=H_{\SL(2)}^{p}(E_{2}).$$
Thus the spectral sequence degenerates at the $E_{2}$-term and then we have
$$H_{\SL(2)}^{i}(Bl_{\Sig}(C_{\Sig}\bR_{1}))\cong H_{\SL(2)}^{i}(E_{2})$$
for each $i\ge0$. 

Hence we get
$$P_{t}^{\SL(2)}(\bR_{2})=P_{t}^{\SL(2)}(\bR_{1}^{ss})+P_{t}^{\SL(2)}(E_{2})-P_{t}^{\SL(2)}(\Sig).$$
\end{enumerate}
\end{proof}

The following blowing-up formula for the first blowing-up $\pi_{\bR_{1}}:\bR_{1}\to\bR$ is what we desire.

\begin{conjecture}\label{equivariant blowing up formula conjecture}
$P_{t}^{\SL(2)}(\bR_{1}^{ss})=P_{t}^{\SL(2)}(\bR)+2^{2g}(P_{t}^{\SL(2)}(\PP\Upsilon^{-1}(0)^{ss})-P_{t}(\BSL(2)))$.
\end{conjecture}

Since $\bR_{1}$ is neither smooth nor projective, we cannot directly apply the Morse theory of \cite{K85-1} developed by F. Kirwan for a proof of Conjecture \ref{equivariant blowing up formula conjecture}.

On the other hand, since $\Sig$ is smooth (See the proof of Lemma \ref{geometric descriptions of second cones}-(1) and \cite[Proposition 1.7.10]{O99}), $\bR_{2}$ is smooth. So a blowing-up formula for the equivariant cohomology on the second blowing-up
$$\pi_{\bR_{2}}:\bR_{2}\to\bR_{1}^{ss}$$
follows from the same argument as in the proof of \cite[Proposition 7.4]{K85-2} under an assumption suggested in \cite[9.5]{K85-1}.

Consider $\bR_{2}$ as a subvariety of $\PP^{N}$ acted on linearly by $\SL(2)$. Assume that $\CC^{*}$ acts on $\PP^{N}$ by $t\cdot x=\diag(\alpha_{0}(t),\cdots,\alpha_{N}(t))x$ for $t\in\CC^{*}$, where $\alpha_{0},\cdots,\alpha_{N}$ are characters of $\CC^{*}$ identified with points of $\Lie(\CC^{*})^{*}$. For the Morse stratification $\{S_{\beta}\,|\,\beta\in\bfB\}$ of $\PP^{N}$ obtained from the norm square of the moment map $\mu:\PP^{N}\to sl(2)^{*}$, $\{\bR_{2}\cap S_{\beta}\,|\,\beta\in\bfB\}$ is the induced stratification of $\bR_{2}$ and we have
$$\bR_{2}\cap S_{\beta}\cong\SL(2)\times_{P_{\beta}}(\bR_{2}\cap Y_{\beta}^{ss})$$
with a retraction $p_{\beta}:Y_{\beta}^{ss}\to Z_{\beta}^{ss}$ of $Y_{\beta}^{ss}$, where $\bfB$ is a subset of $\Lie(\CC^{*})=\CC$ as the indexing set for the Morse stratification, $Y_{\beta}=\{(x_{0}:\cdots:x_{N})\in\PP^{N}\,|\,x_{j}=0\text{ unless }\alpha_{j}\cdot\beta=\|\beta\|^{2}\text{ and }x_{j}\ne 0\text{ for some }j\text{ with }\alpha_{j}\cdot\beta=\|\beta\|^{2}\}$, $Z_{\beta}=\{(x_{0}:\cdots:x_{N})\in\PP^{N}\,|\,x_{j}=0\text{ unless }\alpha_{j}\cdot\beta=\|\beta\|^{2}\}$, $p_{\beta}:Y_{\beta}\to Z_{\beta}$ is a retraction defined as the limit map by the path of steepest descent under the function $\mu_{\beta}:\PP^{N}\to\RR,x\mapsto\mu(x)\cdot\beta$, $Z_{\beta}^{ss}$ is the set of the semistable points of $Z_{\beta}$ under the action of a subgroup of $\SL(2)$ and $Y_{\beta}^{ss}=p_{\beta}^{-1}(Z_{\beta}^{ss})$.

\begin{conjecture}\label{assumption 9.5}
$p_{\beta}(x)\in\bR_{2}$ whenever $x\in\bR_{2}\cap Y_{\beta}^{ss}$ for each $\beta\in\bfB$.
\end{conjecture}

Conjecture \ref{assumption 9.5} implies that $\{\bR_{2}\cap S_{\beta}\,|\,\beta\in\bfB\}$ is equivariantly perfect (See \cite[Section 5 and 8]{K85-1}). If Conjecture \ref{assumption 9.5} holds, we can use the argument of the proof of \cite[Proposition 7.4]{K85-2} to get the following blowing-up formula.

\begin{proposition}\label{equivariant blowing up formula on the 2nd blowing-up}
Assume that Conjecture \ref{conjecture} and Conjecture \ref{assumption 9.5} holds. Then
$$P_{t}^{\SL(2)}(\bR_{2}^{s})=P_{t}^{\SL(2)}(\bR_{1}^{ss})+P_{t}^{\SL(2)}(E_2^{ss})-P_{t}^{\SL(2)}(\Sig).$$
\end{proposition}
\begin{proof}
Assume that Conjecture \ref{conjecture} and Conjecture \ref{assumption 9.5} holds. Then by \cite[Section 5 and 8]{K85-1}, $\{\bR_{2}\cap S_{\beta}|\beta\in\bfB\}$ and $\{\bR_{2}\cap S_{\beta}\cap E_{2}|\beta\in\bfB\}$ are equivariantly perfect. Thus we have
$$P_{t}^{\SL(2)}(\bR_{2})=P_{t}^{\SL(2)}(\bR_{2}^{s})+\sum_{\beta\ne0}t^{2d'(\beta)}P_{t}^{\SL(2)}(\bR_{2}\cap S_{\beta})$$
and
$$P_{t}^{\SL(2)}(E_{2})=P_{t}^{\SL(2)}(E_{2}^{ss})+\sum_{\beta\ne0}t^{2d(\beta)}P_{t}^{\SL(2)}(\bR_{2}\cap S_{\beta}\cap E_{2}),$$
where $d'(\beta)$ (respectively, $d(\beta)$) is the codimension of $\bR_{2}\cap S_{\beta}$ (respectively, $\bR_{2}\cap S_{\beta}\cap E_{2}$) in $\bR_{2}$ (respectively, $E_{2}$). Since $p_{\beta}:Y_{\beta}^{ss}\to Z_{\beta}^{ss}$ is a retraction, we also have
$$P_{t}^{\SL(2)}(\bR_{2}\cap S_{\beta})=P_{t}^{\SL(2)}(\SL(2)(\bR_{2}\cap Z_{\beta}^{ss}))$$
and
$$P_{t}^{\SL(2)}(\bR_{2}\cap S_{\beta}\cap E_{2})=P_{t}^{\SL(2)}(\SL(2)(\bR_{2}\cap Z_{\beta}^{ss})\cap E),$$
where $Z_{\beta}^{ss}$ denotes the set of points of $S_{\beta}$ fixed by the one-parameter subgroup generated by $\beta$. Since $\bR_{2}\cap Z_{\beta}^{ss}\subseteq E_{2}$ by \cite[Lemma 7.6]{K85-2} and $\bR_{2}\cap S_{\beta}\not\subseteq E_{2}$ for any $\beta\in\bfB$ by \cite[Lemma 7.11]{K85-2}, it follows from Lemma \ref{equivariant cohomology blowing-up formula}-(\ref{equivariant cohomology blowing-up formula-(2)}) that
$$P_{t}^{\SL(2)}(\bR_{2}^{s})=P_{t}^{\SL(2)}(\bR_{1}^{ss})+P_{t}^{\SL(2)}(E_{2}^{ss})-P_{t}^{\SL(2)}(\Sig)$$
$$+\sum_{\beta\ne0}(t^{2d'(\beta)}-t^{2d(\beta)})P_{t}^{\SL(2)}(\SL(2)(\bR_{2}\cap Z_{\beta}^{ss}))$$
$$=P_{t}^{\SL(2)}(\bR_{1}^{ss})+P_{t}^{\SL(2)}(E_{2}^{ss})-P_{t}^{\SL(2)}(\Sig).$$
\end{proof}

For the blowing-up $\pi:Bl_{\PP\Hom_{1}}\PP\Upsilon^{-1}(0)^{ss}\to\PP\Upsilon^{-1}(0)^{ss}$, the assumption of Proposition \ref{lgss} can be verified. So we have the following injective pullback $\pi^{*}$ on the equivariant cohomology.

\begin{proposition}\label{injectivity of pullback on equivariant cohomology of pi}
$$\pi^{*}:H_{\SL(2)}^{*}(\PP\Upsilon^{-1}(0)^{ss})\to H_{\SL(2)}^{*}(Bl_{\PP\Hom_{1}}\PP\Upsilon^{-1}(0)^{ss})$$
is injective.
\end{proposition}
\begin{proof}
Recall that
$$\pi:Bl_{\PP\Hom_{1}}\PP\Upsilon^{-1}(0)^{ss}\to\PP\Upsilon^{-1}(0)^{ss}$$
is the blowing-up of $\PP\Upsilon^{-1}(0)^{ss}$ along $\PP\Hom_{1}(sl(2),\HH^{g})^{ss}$. Let $\overline{\pi}:(Bl_{\PP\Hom_{1}}\PP\Upsilon^{-1}(0)^{ss})^{ss}/\!/\SL(2)\to\PP\Upsilon^{-1}(0)^{ss}/\!/\SL(2)$ be the blowing-up of $\PP\Upsilon^{-1}(0)^{ss}/\!/\SL(2)$ along $\PP\Hom_{1}(sl(2),\HH^{g})^{ss}/\!/\SL(2)$ induced from $\pi$.

We showed that
$$(\PP\Upsilon^{-1}(0)^{ss})^{\CC^{*}}\cong\PP(H^{1}(\cO_{X})\otimes sl(2))^{ss}\sqcup\PP(H^{0}(K_{X})\otimes sl(2))^{ss}$$
in the proof of Lemma \ref{1st assumption}-(\ref{1st assumption2}).

Let $s:=(\PP\Upsilon^{-1}(0)^{ss})^{\CC^{*}}\cap\PP\Hom_{1}(sl(2),\HH^{g})^{ss}$. Then we can see that
$$(Bl_{\PP\Hom_{1}}\PP\Upsilon^{-1}(0)^{ss})^{\CC^{*}}$$
$$=Bl_{\PP\Hom_{1}}\PP(H^{1}(\cO_{X})\otimes sl(2))^{ss}\sqcup\pi^{-1}(\PP(H^{0}(K_{X})\otimes sl(2))^{ss})\sqcup s^{+}$$
as in the proof of Lemma \ref{1st assumption}-(\ref{1st assumption3}), where $s^{+}$ is $\PP^{2g-3}$-bundle over $s$.

Now we need to check the assumption of Proposition \ref{lgss}. It is easy to check that $\dss\lim_{\lambda\to 0}\lambda\cdot x$ uniquely exists in $(\PP\Upsilon^{-1}(0)^{ss})^{\CC^{*}}$ for every $x\in\PP\Upsilon^{-1}(0)^{ss}$. Moreover since $\pi:Bl_{\PP\Hom_{1}}\PP\Upsilon^{-1}(0)^{ss}\to\PP\Upsilon^{-1}(0)^{ss}$ is proper, we see that $\dss\lim_{\lambda\to 0}\lambda\cdot\tx$ uniquely exists in $(Bl_{\PP\Hom_{1}}\PP\Upsilon^{-1}(0)^{ss})^{\CC^{*}}$ for every $
\tx\in Bl_{\PP\Hom_{1}}\PP\Upsilon^{-1}(0)^{ss}$. Thus the assumption of Proposition \ref{lgss} is verified.

Since the pullback $\pi^{*}:H_{\SL(2)}^{*}(\PP(H^{1}(\cO_{X})\otimes sl(2))^{ss})\to H_{\SL(2)}^{*}(Bl_{\PP\Hom_{1}}\PP(H^{1}(\cO_{X})\otimes sl(2))^{ss})$ is injective, it follows from Proposition \ref{lgss} that $\pi^{*}:H_{\SL(2)}^{*}(\PP\Upsilon^{-1}(0)^{ss})\to H_{\SL(2)}^{*}(Bl_{\PP\Hom_{1}}\PP\Upsilon^{-1}(0)^{ss})$ is injective.
\end{proof}

\begin{proposition}\label{equivariant cohomology local blowing-up formula}
$P_{t}^{\SL(2)}(Bl_{\PP\Hom_{1}}\PP\Upsilon^{-1}(0)^{ss})$
$$=P_{t}^{\SL(2)}(\PP\Upsilon^{-1}(0)^{ss})+P_{t}^{\SL(2)}(E_{\pi})-P_{t}^{\SL(2)}(\PP\Hom_{1}(sl(2),\HH^{g})^{ss}).$$
\end{proposition}
\begin{proof}
By the same way as in the proof of Lemma \ref{equivariant cohomology blowing-up formula}-(\ref{equivariant cohomology blowing-up formula-(1)}), we have
$$P_{t}^{\SL(2)}(Bl_{\PP\Hom_{1}}\PP\Upsilon^{-1}(0)^{ss})=P_{t}^{\SL(2)}(\PP\Upsilon^{-1}(0)^{ss})+P_{t}^{\SL(2)}(\tU)-P_{t}^{\SL(2)}(U)$$
for some sufficiently open neighborhood $U$ of $\PP\Hom_{1}(sl(2),\HH^{g})^{ss}$ and $\tU=\pi^{-1}(U)$.
By \cite[Theorem 3.1]{GM88} and Proposition \ref{normal slice is the quadratic cone}, $U$ is analytically isomorphic to
$$C_{\PP\Hom_{1}(sl(2),\HH^{g})^{ss}}\PP\Upsilon^{-1}(0)^{ss}$$
and then $\tU$ is analytically isomorphic to
$$Bl_{\PP\Hom_{1}(sl(2),\HH^{g})^{ss}}(C_{\PP\Hom_{1}(sl(2),\HH^{g})^{ss}}\PP\Upsilon^{-1}(0)^{ss}).$$
By using the Leray spectral sequence as in the proof of Lemma \ref{equivariant cohomology blowing-up formula}, we see that
$$H_{\SL(2)}^{*}(C_{\PP\Hom_{1}(sl(2),\HH^{g})^{ss}}\PP\Upsilon^{-1}(0)^{ss})\cong H_{\SL(2)}^{*}(\PP\Hom_{1}(sl(2),\HH^{g})^{ss})$$
and
$$H_{\SL(2)}^{*}(Bl_{\PP\Hom_{1}(sl(2),\HH^{g})^{ss}}(C_{\PP\Hom_{1}(sl(2),\HH^{g})^{ss}}\PP\Upsilon^{-1}(0)^{ss}))\cong H_{\SL(2)}^{*}(E_{\pi}).$$
Hence we get
$$P_{t}^{\SL(2)}(Bl_{\PP\Hom_{1}}\PP\Upsilon^{-1}(0)^{ss})=P_{t}^{\SL(2)}(\PP\Upsilon^{-1}(0)^{ss})+P_{t}^{\SL(2)}(E_{\pi})-P_{t}^{\SL(2)}(\PP\Hom_{1}(sl(2),\HH^{g})^{ss}).$$
\end{proof}

The blowing-up formula for the equivariant cohomology on the blowing-up
$$\pi:Bl_{\PP\Hom_{1}}\PP\Upsilon^{-1}(0)^{ss}\to\PP\Upsilon^{-1}(0)^{ss}$$
is obtained from the same argument as in the proof of \cite[Proposition 7.4]{K85-2}.

\begin{proposition}\label{equivariant cohomology local blowing-up formula}
$P_{t}^{\SL(2)}((Bl_{\PP\Hom_{1}}\PP\Upsilon^{-1}(0)^{ss})^{ss})$
$$=P_{t}^{\SL(2)}(\PP\Upsilon^{-1}(0)^{ss})+P_{t}^{\SL(2)}(E_{\pi}^{ss})-P_{t}^{\SL(2)}(\PP\Hom_{1}(sl(2),\HH^{g})^{ss}).$$
\end{proposition}
\begin{proof}
By Proposition \ref{smoothness of 2nd local blowing-up}, $Bl_{\PP\Hom_{1}}\PP\Upsilon^{-1}(0)^{ss}$ is a smooth projective variety. Thus the same argument as in the proof of \cite[Proposition 7.4]{K85-2} can be applied.

There is a Morse stratification $\{S_{\beta}|\beta\in\bfB\}$ of $Bl_{\PP\Hom_{1}}\PP\Upsilon^{-1}(0)^{ss}$ associated to the lifted action of $\SL(2)$. Then $\{S_{\beta}\cap E_{\pi}|\beta\in\bfB\}$ is the Morse stratification of $E_{\pi}$. By \cite[Section 5 and 8]{K85-1}, $\{S_{\beta}|\beta\in\bfB\}$ and $\{S_{\beta}\cap E_{\pi}|\beta\in\bfB\}$ are equivariantly perfect. Thus we have
$$P_{t}^{\SL(2)}(Bl_{\PP\Hom_{1}}\PP\Upsilon^{-1}(0)^{ss})=P_{t}^{\SL(2)}((Bl_{\PP\Hom_{1}}\PP\Upsilon^{-1}(0)^{ss})^{ss})+\sum_{\beta\ne0}t^{2d'(\beta)}P_{t}^{\SL(2)}(S_{\beta})$$
and
$$P_{t}^{\SL(2)}(E_{\pi})=P_{t}^{\SL(2)}(E_{\pi}^{ss})+\sum_{\beta\ne0}t^{2d(\beta)}P_{t}^{\SL(2)}(S_{\beta}\cap E_{\pi}),$$
where $d'(\beta)$ (respectively, $d(\beta)$) is the codimension of $S_{\beta}$ (respectively, $S_{\beta}\cap E_{\pi}$) in $Bl_{\PP\Hom_{1}}\PP\Upsilon^{-1}(0)^{ss}$ (respectively, $E_{\pi}$). We also have
$$P_{t}^{\SL(2)}(S_{\beta})=P_{t}^{\SL(2)}(\SL(2)Z_{\beta}^{ss})$$
and
$$P_{t}^{\SL(2)}(S_{\beta}\cap E_{\pi})=P_{t}^{\SL(2)}(\SL(2)Z_{\beta}^{ss}\cap E_{\pi}),$$
where $Z_{\beta}^{ss}$ denotes the set of points of $S_{\beta}$ fixed by the one-parameter subgroup generated by $\beta$. Since $Z_{\beta}^{ss}\subseteq E_{\pi}$ by \cite[Lemma 7.6]{K85-2} and $S_{\beta}\not\subseteq E_{\pi}$ for any $\beta\in\bfB$ by \cite[Lemma 7.11]{K85-2}, we have
$$P_{t}^{\SL(2)}((Bl_{\PP\Hom_{1}}\PP\Upsilon^{-1}(0)^{ss})^{ss})=P_{t}^{\SL(2)}(\PP\Upsilon^{-1}(0)^{ss})+P_{t}^{\SL(2)}(E_{\pi}^{ss})$$
$$-P_{t}^{\SL(2)}(\PP\Hom_{1}(sl(2),\HH^{g})^{ss})+\sum_{\beta\ne0}(t^{2d'(\beta)}-t^{2d(\beta)})P_{t}^{\SL(2)}(\SL(2)Z_{\beta}^{ss})$$
$$=P_{t}^{\SL(2)}(\PP\Upsilon^{-1}(0)^{ss})+P_{t}^{\SL(2)}(E_{\pi}^{ss})-P_{t}^{\SL(2)}(\PP\Hom_{1}(sl(2),\HH^{g})^{ss}).$$
\end{proof}

\subsection{Intersection Poincar\'{e} polynomial of the deepest singularity of $\bM$}\label{intersection poincare polynomial of the deepest singularity}
In order to use Theorem \ref{computable intersection blowing-up formula}-(1), we must calculate $IP_{t}(\PP\Upsilon^{-1}(0)^{ss}/\!/\PGL(2))$ and $IP_{t}(\Upsilon^{-1}(0)/\!/\PGL(2))$.

Recall that it follows from Proposition \ref{ss-local-first-blowup} that
$$\PP\Hom_{1}(sl(2),\HH^{g})^{ss}=\PGL(2)Z^{ss}\cong\PGL(2)\times_{\O(2)}Z^{ss}$$
and
$$\PP\Hom_{1}(sl(2),\HH^{g})^{ss}/\!/\PGL(2)\cong Z/\!/\O(2)=Z/\!/\SO(2)=Z_{1}=\PP^{2g-1},$$
where $Z=Z_{1}\cup Z_{2}\cup Z_{3}$, $Z^{ss}$ is the set of
semistable points of $Z$ for the action of $\O(2)$,
$Z_{1}=\PP\{v_{1}\otimes\HH^{g}\}=Z^{ss}$,
$Z_{2}=\PP\{v_{2}\otimes\HH^{g}\}$ and
$Z_{3}=\PP\{v_{3}\otimes\HH^{g}\}$. Here $\{v_{1},v_{2},v_{3}\}$ is the basis of $sl(2)$ chosen in Section \ref{local picture}.

We see that
$$P_{t}^{+}(Z/\!/\SO(2))=P_{t}(\PP^{2g-1})=1+t^{2}+\cdots+t^{2(2g-1)}=\frac{1-t^{4g}}{1-t^{2}}$$
and
$$P_{t}^{-}(Z/\!/\SO(2))=0,$$
where $P_{t}^{+}(Z/\!/\SO(2))$ and $P_{t}^{-}(Z/\!/\SO(2))$ are Poincar\'{e} polynomials of the invariant part and variant part of $H^{*}(Z/\!/\SO(2))$ with respect to the action of $\ZZ_{2}=\O(2)/\SO(2)$ on $Z/\!/\SO(2)$.

By \cite[Proposition 3.10]{Ma21} and Theorem \ref{computable intersection blowing-up formula}-(3),
$$IP_{t}(Bl_{\PP\Hom_{1}}\PP\Upsilon^{-1}(0)^{ss}/\!/\PGL(2))=IP_{t}(\PP\Upsilon^{-1}(0)^{ss}/\!/\PGL(2))$$
$$+\sum_{p+q=i}\dim[H^{p}(\PP^{2g-1})\otimes H^{t(q)}(I_{2g-3})]^{\ZZ_{2}}t^{i}$$
$$=IP_{t}(\PP\Upsilon^{-1}(0)^{ss}/\!/\PGL(2))+\sum_{p+q=i}\dim[H^{p}(\PP^{2g-1})^{\ZZ_{2}}\otimes H^{t(q)}(I_{2g-3})^{\ZZ_{2}}]t^{i}$$
\begin{equation}\label{2nd local blowing-up}=IP_{t}(\PP\Upsilon^{-1}(0)^{ss}/\!/\PGL(2))+\frac{1-t^{4g}}{1-t^{2}}\cdot\frac{t^2(1-t^{4g-6})(1-t^{4g-4})}{(1-t^2)(1-t^4)}.\end{equation}

Let
$\widetilde{\PP\Hom_{2}^{\omega}(sl(2),\HH^{g})}=Bl_{\PP\Hom_{1}}\PP\Hom_{2}^{\omega}(sl(2),\HH^{g})$
be the blowing-up of
$\PP\Hom_{2}^{\omega}(sl(2),\HH^{g})$ along
$\PP\Hom_{1}(sl(2),\HH^{g})^{ss}$ and let
$$Bl_{\widetilde{\PP\Hom_{2}^{\omega}}}Bl_{\PP\Hom_{1}}\PP\Upsilon^{-1}(0)^{ss}$$
be the blowing-up of
$Bl_{\PP\Hom_{1}}\PP\Upsilon^{-1}(0)^{ss}$ along
$\widetilde{\PP\Hom_{2}^{\omega}(sl(2),\HH^{g})}^{ss}$.

Assume that $g\ge3$. Denote
$D_{1}=Bl_{\widetilde{\PP\Hom_{2}^{\omega}}}Bl_{\PP\Hom_{1}}\PP\Upsilon^{-1}(0)^{ss}/\!/\PGL(2)$.
By \cite[Proposition 4.2]{KY08}, $D_{1}$ is a
$\widehat{\PP}^5$-bundle over $Gr^{\omega}(3,2g)$ where
$\widehat{\PP}^5$ is the blowing-up of $\PP^5$
(projectivization of the space of $3\times 3$ symmetric matrices)
along $\PP^2$ (the locus of rank $1$ matrices). Since $D_1$
is a nonsingular projective variety over $\CC$,
$$\dim_{\CC}H^{k}(D_{1};\CC)=\sum_{p+q=k}h^{p,q}(H^{k}(D_{1};\CC)).$$
Thus it follows from \cite[Proposition 5.1]{KY08} that
$$IP_{t}(D_{1})=P_{t}(D_{1})=E(D_{1};-t,-t)=(\frac{1-t^{12}}{1-t^{2}}-\frac{1-t^{6}}{1-t^{2}}+(\frac{1-t^{6}}{1-t^{2}})^{2})\cdot\prod_{1\leq i\leq 3}\frac{1-t^{4g-12+4i}}{1-t^{2i}}.$$
Moreover by the proof of \cite[Proposition 3.5.1]{O97}
$$\widetilde{\PP\Hom_{2}^{\omega}(sl(2),\HH^{g})}/\!/\PGL(2)\cong
\PP(S^{2}\cA)$$ where $\cA$ is the
tautological rank $2$ bundle over $Gr^{\omega}(2,2g)$. Following the proof of \cite[Lemma 3.5.4]{O97}, we can see that the exceptional divisor of $D_1$ is a $\PP^{2g-5}$-bundle over $\PP(S^{2}\cA)$.

By the usual blowing-up formula mentioned in \cite[p.605]{GH78}, we have
$$\dim H^{i}(D_1)=\dim H^{i}(Bl_{\PP\Hom_{1}}\PP\Upsilon^{-1}(0)^{ss}/\!/\PGL(2))$$
\begin{equation}\label{3rd local blowing-up}+\big(\sum_{p+q=i}\dim[H^{p}(\PP(S^{2}\cA))\otimes H^{q}(\PP^{2g-5})]-\dim H^{i}(\PP(S^{2}\cA))\big).\end{equation}

Since
$\PP(S^{2}\cA)$ is the $\PP^2$-bundle over
$Gr^{\omega}(2,2g)$,
$$P_{t}(\PP(S^{2}\cA))=P_{t}(\PP^2)P_{t}(Gr^{\omega}(2,2g))=\frac{1-t^{6}}{1-t^{2}}\cdot\prod_{1\leq i\leq 2}\frac{1-t^{4g-8+4i}}{1-t^{2i}}$$
by Deligne's criterion (see \cite{D68}).

Therefore it follows from (\ref{3rd local blowing-up}) that
$$IP_{t}(Bl_{\PP\Hom_{1}}\PP\Upsilon^{-1}(0)^{ss}/\!/\PGL(2))$$
$$=(\frac{1-t^{12}}{1-t^{2}}-\frac{1-t^{6}}{1-t^{2}}+(\frac{1-t^{6}}{1-t^{2}})^{2})\cdot\prod_{1\leq
i\leq
3}\frac{1-t^{4g-12+4i}}{1-t^{2i}}-\frac{1-t^{6}}{1-t^{2}}\cdot\prod_{1\leq
i\leq 2}\frac{1-t^{4g-8+4i}}{1-t^{2i}}$$
$$\times\frac{t^{2}(1-t^{2(2g-5)})}{1-t^{2}}.$$

Assume that $g=2$. In this case, we know from \cite[Proposition 2.0.1]{O97} that
$$Bl_{\PP\Hom_{1}}\PP\Upsilon^{-1}(0)^{ss}/\!/\PGL(2)$$
is already nonsingular and that it is a $\PP^2$-bundle over
$Gr^{\omega}(2,4)$. Then by Deligne's criterion (See \cite{D68}),
$$IP_{t}(Bl_{\PP\Hom_{1}}\PP\Upsilon^{-1}(0)^{ss}/\!/\PGL(2))=P_{t}(Bl_{\PP\Hom_{1}}\PP\Upsilon^{-1}(0)^{ss}/\!/\PGL(2))$$
$$=P_{t}(\PP^2)P_{t}(Gr^{\omega}(2,4))=\frac{1-t^{6}}{1-t^{2}}\cdot\prod_{1\leq i\leq 2}\frac{1-t^{4i}}{1-t^{2i}}=\frac{(1-t^{6})(1-t^{8})}{(1-t^{2})^2}.$$

Combining these with (\ref{2nd local blowing-up}), we obtain

\begin{proposition}
$$IP_{t}(\PP\Upsilon^{-1}(0)^{ss}/\!/\PGL(2))=\frac{(1-t^{8g-8})(1-t^{4g})}{(1-t^{2})(1-t^{4})}.$$
\end{proposition}

By Lemma \ref{suffices-to-show-on-projectivized-git}-(1), we also obtain

\begin{proposition}\label{intersection-Betti-of-normal-cone}
$$IP_{t}(\Upsilon^{-1}(0)/\!/\PGL(2))=\frac{1-t^{4g}}{1-t^{4}}$$
\end{proposition}

\subsection{Intersection Poincar\'{e} polynomial of $\bM$}

In this subsection, we compute a conjectural formula for $IP_{t}(\bM)$.

\subsubsection{Computation for $P_{t}^{\SL(2)}(\bR)$}\label{eqiv coh of R}

We start with the following result.

\begin{theorem}[Corollary 1.2 in \cite{DWW11}]\label{equiv coh of moduli space of Higgs bundles}
$$P_{t}^{\cG_{\CC}}(\cB^{ss})=\frac{(1+t^3)^{2g}-(1+t)^{2g}t^{2g+2}}{(1-t^2)(1-t^4)}$$
$$-t^{4g-4}+\frac{t^{2g+2}(1+t)^{2g}}{(1-t^2)(1-t^4)}+\frac{(1-t)^{2g}t^{4g-4}}{4(1+t^2)}$$
$$\frac{(1+t)^{2g}t^{4g-4}}{2(1-t^2)}(\frac{2g}{t+1}+\frac{1}{t^2-1}-\frac{1}{2}+(3-2g))$$
$$\frac{1}{2}(2^{2g}-1)t^{4g-4}((1+t)^{2g-2}+(1-t)^{2g-2}-2).$$
\end{theorem}

In this subsection, we show that $P_{t}^{\SL(2)}(\bR)=P_{t}^{\cG_{\CC}}(\cB^{ss})$. To prove this, we need some technical lemmas.

Choose a base point $x\in X$ as in Theorem \ref{construction of bR}. Let $E$ be a complex Hermitian vector bundle of rank $2$ and degree $0$ on $X$. Let $p:E\to X$ be the canonical
projection. Let $(\cG_{\CC})_0$ be the normal subgroup of
$\cG_{\CC}$ which fixes the fiber $E|_{x}$.

We first claim that $(\cG_{\CC})_0$ acts freely on $\cB^{ss}$. In
fact, assume that $g\cdot(A,\phi)=(A,\phi)$ for $g\in(\cG_{\CC})_0$. For an arbitrary point $y\in X$ and for any smooth path
$\gamma:[0,1]\to X$ starting at $\gamma(0)=x$ and ending at
$\gamma(1)=y\in X$, there is a parallel transport mapping
$P_{\gamma}:E|_{x}\to E|_{y}$ defined as follows. If $v\in E|_{x}$,
there exists a unique path $\gamma_v:[0,1]\to V$ such that
$p\circ\gamma_v=\gamma$, $\gamma_v(0)=v$ given by $A$. Define
$P_{\gamma}(v)=\gamma_v(1)$. By the assumption, $P_{\gamma}\circ
g|_{x}=g|_{y}\circ P_{\gamma}$. Since $g|_{x}$ is the identity on
$E|_{x}$, $g|_{y}$ is also the identity on $E|_{y}$. Therefore $g$
is the identity on $E$.

Since the surjective map $\cG_{\CC}\to\SL(2)$ given by $g\mapsto g|_{x}$ has the kernel $(\cG_{\CC})_{0}$, we have

\begin{equation}\label{G mod G0 isom SL2}
\cG_{\CC}/(\cG_{\CC})_{0}\cong\SL(2).
\end{equation}

Let $\cG_{\CC}/(\cG_{\CC})_0\times_{\cG_{\CC}}\cB^{ss}$ be the quotient space of
$\cG_{\CC}/(\cG_{\CC})_0\times\cB^{ss}$ by the action of
$\cG_{\CC}$ given by
$$h\cdot(\ovg,(A,\phi))=(\ovg\ovh^{-1},h\cdot(A,\phi))$$
where $\overline{f}$ is the image of $f\in\cG_{\CC}$ under the
quotient map $\cG_{\CC}\to\cG_{\CC}/(\cG_{\CC})_0$. Since
$(\cG_{\CC})_0$ acts freely on $\cB^{ss}$, $\cG_{\CC}$ acts freely
on $\cG_{\CC}/(\cG_{\CC})_0\times\cB^{ss}$.

\begin{lemma}\label{framing}
There exists a homeomorphism between
$\SL(2)\times_{\cG_{\CC}}\cB^{ss}$ and $\cB^{ss}/(\cG_{\CC})_0$.
\end{lemma}
\begin{proof}
By (\ref{G mod G0 isom SL2}), it suffices to show that there exists a homeomorphism between
$\cG_{\CC}/(\cG_{\CC})_0\times_{\cG_{\CC}}\cB^{ss}$ and $\cB^{ss}/(\cG_{\CC})_0$.

Consider the continuous surjective map
$$q:\cG_{\CC}\times\cB^{ss}\to\cG_{\CC}/(\cG_{\CC})_0\times\cB^{ss}$$
given by $(g,(A,\phi))\mapsto(\ovg,(A,\phi))$. If $\cG_{\CC}$ acts on $\cG_{\CC}\times\cB^{ss}$ by
$h\cdot(g,(A,\phi))=(gh^{-1},h\cdot(A,\phi))$, $q$ is
$\cG_{\CC}$-equivariant.

Taking quotients of both spaces
by $\cG_{\CC}$, $q$ induces the continuous surjective map
$$\ovq:\cG_{\CC}\times_{\cG_{\CC}}\cB^{ss}\to\cG_{\CC}/(\cG_{\CC})_0\times_{\cG_{\CC}}\cB^{ss}$$
given by $[g,(A,\phi)]\mapsto[\ovg,(A,\phi)]$.

If $(\cG_{\CC})_0$ acts on $\cG_{\CC}\times_{\cG_{\CC}}\cB^{ss}$ by
$h\cdot[g,(A,\phi)]=[g,h\cdot(A,\phi)]$, $\ovq$ is
$(\cG_{\CC})_0$-invariant. Precisely for $g_0\in(\cG_{\CC})_0$,
$\ovq([g,g_0\cdot(A,\phi)])=[\ovg,g_0\cdot(A,\phi)]=[\ovg\overline{g_0},(A,\phi)]=[\ovg,(A,\phi)]=\ovq([g,(A,\phi)])$.

Thus $\ovq$ induces the continuous surjective map
$$\tq:\frac{\cG_{\CC}\times_{\cG_{\CC}}\cB^{ss}}{(\cG_{\CC})_0}\to\cG_{\CC}/(\cG_{\CC})_0\times_{\cG_{\CC}}\cB^{ss}$$
given by $\overline{[g,(A,\phi)]}\mapsto[\ovg,(A,\phi)]$.

Furthermore $\tq$ is injective. In fact, assume that
$$\tq(\overline{[g_1,(A_1,\phi_1)]})=\tq(\overline{[g_2,(A_2,\phi_2)]}),$$
that is,
$$[\overline{g_1},(A_1,\phi_1)]=[\overline{g_2},(A_2,\phi_2)].$$
Then there is $k\in\cG_{\CC}$ such that
$(\overline{g_1},(A_1,\phi_1))=(\overline{g_2}\overline{k}^{-1},k\cdot(A_2,\phi_2))$.
Then $g_1=g_2 k^{-1}l$ for some $l\in(\cG_{\CC})_0$. Thus
$\overline{[g_1,(A_1,\phi_1)]}=\overline{[g_2
k^{-1}l,k\cdot(A_2,\phi_2)]}=\overline{[g_2,
k^{-1}lk\cdot(A_2,\phi_2)]}=\overline{[g_2,(A_2,\phi_2)]}$
because $(\cG_{\CC})_0$ is the normal subgroup of $\cG_{\CC}$.

On the other hand, since both $q$ and the quotient map $\cG_{\CC}/(\cG_{\CC})_0\times\cB^{ss}\to\cG_{\CC}/(\cG_{\CC})_0\times_{\cG_{\CC}}\cB^{ss}$ are open, $\ovq$ is open. Moreover since the quotient map $\dss\cG_{\CC}\times_{\cG_{\CC}}\cB^{ss}\to\frac{\cG_{\CC}\times_{\cG_{\CC}}\cB^{ss}}{(\cG_{\CC})_0}$ is open,
$\tq$ is also open.

Hence $\tq$ is a homeomorphism. Since there is a homeomorphism $\xymatrix{\cG_{\CC}\times_{\cG_{\CC}}\cB^{ss}\ar[r]^{\quad\quad\cong}&\cB^{ss}}$ given by $[g,(A,\phi)]\mapsto g\cdot(A,\phi)$, we get the conclusion.
\end{proof}

\begin{lemma}\label{from-conn-to-bdl}
There is an isomorphism of complex analytic spaces
$$\SL(2)\times_{\cG_{\CC}}\cB^{ss}\cong\bR.$$
\end{lemma}
\begin{proof}
There is a bijection between
$\SL(2)\times_{\cG_{\CC}}\cB^{ss}$ and
$\bR$. In fact, consider a map
$$f:\SL(2)\times\cB^{ss}\to\bR$$
given by
$(\beta,(A,\phi))\mapsto((E,A''),\phi,\beta)$, where $[\beta,(A,\phi)]$ is the
image of $(\beta,(A,\phi))$ of the quotient map
$\SL(2)\times\cB^{ss}\to
\SL(2)\times_{\cG_{\CC}}\cB^{ss}$. $f$ induces the map $\overline{f}:\SL(2)\times_{\cG_{\CC}}\cB^{ss}\to\bR$ given by $[\beta,(A,\phi)]\mapsto((E,A''),\phi,\beta)$. Now we claim that $\overline{f}$ is bijective. Since $\bM\cong\cB^{ss}/\cG_{\CC}$ by Theorem \ref{Hitchin construction}, we can see that $\overline{f}$ is surjective. If $((E_{1},A_{1}''),\phi_{1},\beta_{1})\cong((E_{2},A_{2}''),\phi_{2},\beta_{2})$, then there exists $g\in\cG_{\CC}$ such that $(A_{2},\phi_{2})=g\cdot(A_{1},\phi_{1})$ and $\beta_{2}=\beta_{1}\overline{g}=\beta_{1}$. Then $g\in(\cG_{\CC})_{0}$ and $[\beta_{2},(A_{2},\phi_{2})]=[\beta_{1},g\cdot(A_{1},\phi_{1})]=[\beta_{1}\overline{g},(A_{1},\phi_{1})]=[\beta_{1},(A_{1},\phi_{1})]$. Thus $\overline{f}$ is injective.

Further, the family $E\times(\SL(2)\times_{\cG_{\CC}}\cB^{ss})$ over $X\times(\SL(2)\times_{\cG_{\CC}}\cB^{ss})$ gives a complex
analytic map $\bg:\SL(2)\times_{\cG_{\CC}}\cB^{ss}\to\bR$ by \cite[Lemma 5.7]{Simp94I}, and $\overline{f}([\beta,(A,\phi)])=\bg([\beta,(A,\phi)])$ for all $[\beta,(A,\phi)]\in\SL(2)\times_{\cG_{\CC}}\cB^{ss}$.

Hence $f$ is an isomorphism of complex analytic spaces $\SL(2)\times_{\cG_{\CC}}\cB^{ss}\cong\bR$.
\end{proof}

There is a technical lemma for equivariant cohomologies.

\begin{lemma}\label{quot-in-stage}
Let $H$ be a closed normal subgroup of $G$ and $M$ be a $G$-space on
which $H$ acts freely. Then $G/H$ acts on $M/H$ and
$$H_{G}^*(M)=H_{G/H}^*(M/H).$$
\end{lemma}
\begin{proof}
Use the fibration $\E G\times_{G}M\cong(\E G\times \E(G/H))\times_{G}M\to
\E(G/H)\times_{G}M\cong \E(G/H)\times_{G/H}(M/H)$ whose fibers $\E G$ is
contractible.
\end{proof}

The following equality is an immediate consequence from Lemma
\ref{framing}, Lemma \ref{from-conn-to-bdl} and Lemma \ref{quot-in-stage}.

\begin{proposition}
$$P_{t}^{\SL(2)}(\bR)=P_{t}^{\cG_{\CC}}(\cB^{ss})$$
\end{proposition}

Thus we get the same formula for $P_{t}^{\SL(2)}(\bR)$ as Theorem
\ref{equiv coh of moduli space of Higgs bundles}.

\subsubsection{Computation for $P_{t}^{\SL(2)}(\Sigma)$}\label{eqiv coh of Sigma}
In the proof of Lemma \ref{geometric descriptions of second cones}-(1), we observed that
$$\Sigma\cong\PP\Isom(\cO_{\widetilde{T^{*}J}}^{2},\cL|_{x}\oplus\cL^{-1}|_{x})/\!/\O(2).$$
Since $\{\pm\id\}\subset\SL(2)$ acts trivially on $\PP\Isom(\cO_{\widetilde{T^{*}J}}^{2},\cL|_{x}\oplus\cL^{-1}|_{x})$, $\{\pm\id\}\subset\SL(2)$ also acts trivially on $\Sigma$. Then
$$\ESL(2)\times_{\SL(2)}\Sigma\cong\EPGL(2)\times_{\PGL(2)}\Sigma.$$
Since $\O(2)$ acts on $\PP\Isom(\cO_{\widetilde{T^{*}J}}^{2},\cL|_{x}\oplus\cL^{-1}|_{x})$ freely and both actions of $\PGL(2)$ and $\O(2)$ commute,
$$\EPGL(2)\times_{\PGL(2)}\Sigma\sim\EPGL(2)\times_{\PGL(2)}(\EO(2)\times_{\O(2)}\PP\Isom(\cO_{\widetilde{T^{*}J}}^{2},\cL|_{x}\oplus\cL^{-1}|_{x}))$$
$$\cong\EO(2)\times_{\O(2)}(\EPGL(2)\times_{\PGL(2)}\PP\Isom(\cO_{\widetilde{T^{*}J}}^{2},\cL|_{x}\oplus\cL^{-1}|_{x}))\sim\EO(2)\times_{\O(2)}\widetilde{T^{*}J}$$
$$\cong(\ESO(2)\times_{\SO(2)}\widetilde{T^{*}J})/(O(2)/SO(2))\cong(\BSO(2)\times\widetilde{T^{*}J})/\ZZ_{2},$$
where $\sim$ denotes the homotopic equivalence. Thus
$$P_{t}^{\SL(2)}(\Sigma)=P_{t}^{+}(\BSO(2))P_{t}^{+}(\widetilde{T^{*}J})+P_{t}^{-}(\BSO(2))P_{t}^{-}(\widetilde{T^{*}J}),$$
where $P_{t}^{+}(W)$ (respectively, $P_{t}^{-}(W)$) denotes the Poincar\'{e} polynomial of the invariant (respectively, variant) part of $H^{*}(W)$ with respect to the action of $\ZZ_{2}$ on $W$ for a $\ZZ_{2}$-space $W$.

\begin{lemma}
$$P_{t}^{\SL(2)}(\Sigma)=\frac{1}{(1-t^4)}(\frac{1}{2}((1+t)^{2g}+(1-t)^{2g})+2^{2g}(\frac{1-t^{4g}}{1-t^2}-1))+\frac{t^2}{(1-t^4)}\frac{1}{2}((1+t)^{2g}-(1-t)^{2g}).$$
\end{lemma}
\begin{proof}
Note that $\BSO(2)\cong\PP^{\infty}$. Since the action of $\ZZ_{2}\setminus\{\id\}$ on $H^*(\BSO(2))$ represents reversing of orientation and $\PP^{n}$ possess an orientation-reversing self-homeomorphism only when $n$ is odd, we have $\dss P_{t}^{+}(\BSO(2))=\frac{1}{1-t^{4}}$ and $\dss P_{t}^{-}(\BSO(2))=\frac{t^{2}}{1-t^{4}}$.

Further, by the computation mentioned in \cite[Lemma 4.3]{CK06} and \cite[Section 5]{CK07}, we have $$P_{t}^{+}(\widetilde{T^{*}J})=\frac{1}{2}((1+t)^{2g}+(1-t)^{2g})+2^{2g}(\frac{1-t^{4g}}{1-t^2}-1)$$
and
$$P_{t}^{-}(\widetilde{T^{*}J})=\frac{1}{2}((1+t)^{2g}-(1-t)^{2g}).$$
\end{proof}

\subsubsection{Computation for $P_{t}^{\SL(2)}(\PP\Upsilon^{-1}(0)^{ss})$}\label{eqiv coh of deepest sing}
Since $E/\!/\SL(2)$ has an orbifold singularity and $E/\!/\SL(2)\cong\PP\cC/\!/\SL(2)$ is a free $\ZZ_{2}$-quotient of $I_{2g-3}$-bundle over $\PP^{2g-1}$ by Lemma \ref{geometric descriptions of second cones}-(5) and Lemma \ref{isomorphic to incidence variety}, we use \cite[Proposition 3.10]{Ma21} to have
$$P_{t}^{\SL(2)}(E^{ss})=P_{t}(E/\!/\SL(2))=P_{t}^{+}(I_{2g-3})P_{t}(\PP^{2g-1})$$
$$=\frac{(1-t^{4g-4})^{2}}{(1-t^2)(1-t^4)}\cdot\frac{1-t^{4g}}{1-t^2}.$$

By Proposition \ref{equivariant cohomology local blowing-up formula},
$$P_{t}^{\SL(2)}((Bl_{\PP\Hom_{1}}\PP\Upsilon^{-1}(0)^{ss})^{s})$$
$$=P_{t}^{\SL(2)}(\PP\Upsilon^{-1}(0)^{ss})+P_{t}^{\SL(2)}(E^{ss})-P_{t}^{\SL(2)}(\PP\Hom_{1}(sl(2),\mathbb{H}^{g})^{ss})$$
$$=P_{t}^{\SL(2)}(\PP\Upsilon^{-1}(0)^{ss})+\frac{(1-t^{4g-4})^{2}}{(1-t^2)(1-t^4)}\cdot\frac{1-t^{4g}}{1-t^2}-\frac{1}{1-t^4}\frac{1-t^{4g}}{1-t^2}$$
On the other hand
$$P_{t}^{\SL(2)}((Bl_{\PP\Hom_{1}}\PP\Upsilon^{-1}(0)^{ss})^{s})=P_{t}(Bl_{\PP\Hom_{1}}\PP\Upsilon^{-1}(0)^{ss}/\!/\SL(2))$$
$$=\big(\frac{1-t^{12}}{1-t^2}-\frac{1-t^6}{1-t^2}+(\frac{1-t^6}{1-t^2})^{2}\big)\frac{(1-t^{4g-8})(1-t^{4g-4})(1-t^{4g})}{(1-t^2)(1-t^4)(1-t^6)}$$
$$-\frac{1-t^6}{1-t^2}\frac{(1-t^{4g-4})(1-t^{4g})}{(1-t^2)(1-t^4)}\frac{t^2(1-t^{2(2g-5)})}{1-t^2}$$
from the subsection \ref{intersection poincare polynomial of the deepest singularity} for any $g\ge2$.

Hence
$$P_{t}^{\SL(2)}(\PP\Upsilon^{-1}(0)^{ss})$$
$$=\big(\frac{1-t^{12}}{1-t^2}-\frac{1-t^6}{1-t^2}+(\frac{1-t^6}{1-t^2})^{2}\big)\frac{(1-t^{4g-8})(1-t^{4g-4})(1-t^{4g})}{(1-t^2)(1-t^4)(1-t^6)}$$
$$-\frac{1-t^6}{1-t^2}\frac{(1-t^{4g-4})(1-t^{4g})}{(1-t^2)(1-t^4)}\frac{t^2(1-t^{2(2g-5)})}{1-t^2}$$
$$-\frac{(1-t^{4g-4})^{2}}{(1-t^2)(1-t^4)}\cdot\frac{1-t^{4g}}{1-t^2}+\frac{1}{1-t^4}\frac{1-t^{4g}}{1-t^2}.$$

\subsubsection{Computation for $P_{t}^{\SL(2)}(E_{2}^{ss})$}\label{eqiv coh of E2}
Since $E_2/\!/\SL(2)$ has an orbifold singularity and $E_2/\!/\SL(2)\cong\PP\cC_{2}/\!/\SL(2)$ is a free $\ZZ_{2}$-quotient of a $I_{2g-3}$-bundle over $\widetilde{T^{*}J}$ by Lemma \ref{geometric descriptions of second cones}-(1) and Lemma \ref{geometric descriptions of second cones}-(3), we use \cite[Proposition 3.10]{Ma21} to have
$$P_{t}(E_2/\!/\SL(2))=P_{t}^{+}(\widetilde{T^{*}J})P_{t}^{+}(I_{2g-3})+P_{t}^{-}(\widetilde{T^{*}J})P_{t}^{-}(I_{2g-3})$$
$$=(\frac{1}{2}((1+t)^{2g}+(1-t)^{2g})+2^{2g}(\frac{1-t^{4g}}{1-t^2}-1))\frac{(1-t^{4g-4})^{2}}{(1-t^2)(1-t^4)}$$
$$+\frac{1}{2}((1+t)^{2g}-(1-t)^{2g})\frac{t^2(1-t^{4g-4})(1-t^{4g-8})}{(1-t^2)(1-t^4)}.$$

\subsubsection{A conjectural formula for $IP_{t}(\bM)$}

Combining Theorem \ref{computable intersection blowing-up formula}, Lemma \ref{equivariant cohomology blowing-up formula}, Proposition \ref{intersection-Betti-of-normal-cone}, section \ref{eqiv coh of R}, section \ref{eqiv coh of Sigma}, section \ref{eqiv coh of deepest sing} and section \ref{eqiv coh of E2}, we get a conjectural formula for $IP_{t}(\bM)$ as following. The residue calculations show that the coefficients of the terms of $t^{i}$ are zero for $i>6g-6$ and the coefficient of the term of $t^{6g-6}$ is nonzero.

\begin{theorem}\label{A formula for IP(M)}
Assume that Conjecture \ref{conjecture}, Conjecture \ref{equivariant blowing up formula conjecture} and Conjecture \ref{assumption 9.5} hold. Then
$$IP_{t}(\bM)=\frac{(1+t^3)^{2g}-(1+t)^{2g}t^{2g+2}}{(1-t^2)(1-t^4)}$$
$$-t^{4g-4}+\frac{t^{2g+2}(1+t)^{2g}}{(1-t^2)(1-t^4)}+\frac{(1-t)^{2g}t^{4g-4}}{4(1+t^2)}$$
$$+\frac{(1+t)^{2g}t^{4g-4}}{2(1-t^2)}(\frac{2g}{t+1}+\frac{1}{t^2-1}-\frac{1}{2}+(3-2g))$$
$$+\frac{1}{2}(2^{2g}-1)t^{4g-4}((1+t)^{2g-2}+(1-t)^{2g-2}-2)$$
$$+2^{2g}\big[\big(\frac{1-t^{12}}{1-t^2}-\frac{1-t^6}{1-t^2}+(\frac{1-t^6}{1-t^2})^{2}\big)\frac{(1-t^{4g-8})(1-t^{4g-4})(1-t^{4g})}{(1-t^2)(1-t^4)(1-t^6)}$$
$$-\frac{1-t^6}{1-t^2}\frac{(1-t^{4g-4})(1-t^{4g})}{(1-t^2)(1-t^4)}\frac{t^2(1-t^{2(2g-5)})}{1-t^2}$$
$$-\frac{(1-t^{4g-4})^{2}}{(1-t^2)(1-t^4)}\cdot\frac{1-t^{4g}}{1-t^2}+\frac{1}{1-t^4}\frac{1-t^{4g}}{1-t^2}\big]-\frac{2^{2g}}{1-t^4}$$
$$+(\frac{1}{2}((1+t)^{2g}+(1-t)^{2g})+2^{2g}(\frac{1-t^{4g}}{1-t^2}-1))\frac{(1-t^{4g-4})^{2}}{(1-t^2)(1-t^4)}$$
$$+\frac{1}{2}((1+t)^{2g}-(1-t)^{2g})\frac{t^2(1-t^{4g-4})(1-t^{4g-8})}{(1-t^2)(1-t^4)}$$
$$-\frac{1}{(1-t^4)}(\frac{1}{2}((1+t)^{2g}+(1-t)^{2g})+2^{2g}(\frac{1-t^{4g}}{1-t^2}-1))$$
$$-\frac{t^2}{(1-t^4)}\frac{1}{2}((1+t)^{2g}-(1-t)^{2g})$$
$$-\frac1 2(
(1+t)^{2g}+(1-t)^{2g})+2^{2g}(\frac{1-t^{4g}}{1-t^2}-1))\frac{t^2(1-t^{4g-4})(1-t^{4g-6})}{(1-t^2)(1-t^4)}$$
$$-\frac1 2(
(1+t)^{2g}-(1-t)^{2g})
(\frac{t^4(1-t^{4g-4})(1-t^{4g-10})}{(1-t^2)(1-t^4)}+t^{4g-6})$$
$$-2^{2g}\big[\frac{(1-t^{8g-8})(1-t^{4g})}{(1-t^{2})(1-t^{4})}-\frac{1-t^{4g}}{1-t^{4}}\big]$$
which is a polynomial with degree $6g-6$.
\end{theorem}

In low genus, we have $IP_{t}(\bM)$ as follows :
\begin{itemize}
\item $g=2$ : $IP_{t}(\bM)=1+t^{2}+17t^{4}+17t^{6}$

\item $g=3$ : $IP_{t}(\bM)=1+t^{2} +6t^{3} +2t^{4} +6t^{5} +17t^{6} +6t^{7} +81t^{8} +12t^{9} +396t^{10} +6t^{11} +66t^{12}$

\item $g=4$ : $IP_{t}(\bM)=1+t^{2} +8t^{3} +2t^{4} +8t^{5} +30t^{6} +16t^{7} +31t^{8} +72t^{9} +59t^{10} +72t^{11} +385t^{12}+ 80t^{13} + 3955t^{14} + 80t^{15} + 3885t^{16} + 16t^{17} + 259t^{18}$

\item $g=5$ : $IP_{t}(\bM)=1+t^{2} +10t^{3} +2t^{4} +10t^{5} +47t^{6} +20t^{7} +48t^{8} +140t^{9} +93t^{10} +150t^{11}+ 304t^{12} + 270t^{13} + 349t^{14} + 522t^{15} + 1583t^{16} + 532t^{17} + 29414t^{18} + 532t^{19}+ 72170t^{20} + 280t^{21}+ 28784t^{22} + 30t^{23} + 1028t^{24}$.
\end{itemize}

\section*{Acknowledgments}
I would like to thank Young-Hoon Kiem for suggesting problem, and for his help and encouragement. This work is based and developed on the second topic in my doctoral dissertation \cite{Y09}.

\end{document}